\documentclass[a4paper, leqno, 10pt]{amsart}

\usepackage[T1]{fontenc}	
\usepackage{graphicx}
\usepackage{amssymb,centernot}
\usepackage{upgreek}
\usepackage{mathrsfs}
\usepackage[usenames,dvipsnames]{xcolor}
\usepackage[inline,shortlabels]{enumitem}
\usepackage[sort,numbers]{natbib}						
\usepackage{verbatim}								
\usepackage{dsfont}									
\usepackage{microtype}

\usepackage{pdflscape}
\usepackage{afterpage}

\usepackage{faktor}
\usepackage{anyfontsize}
\usepackage[
						colorlinks,				
						breaklinks,unicode,
						hypertexnames=false,
						citecolor=OliveGreen,
						linkcolor=Maroon
						]{hyperref} 

\usepackage{multirow}

\usepackage{xypic}
\usepackage{mathtools}
\usepackage{xspace}

\usepackage[a4paper,top=0.8in,bottom=0.8in,left=1in,right=1in, bindingoffset=0mm]{geometry}
\usepackage{booktabs}

\usepackage{tikz}

\usetikzlibrary{calc, positioning, shapes.geometric, patterns, cd}

\newlength{\hatchspread}
\newlength{\hatchthickness}
\newlength{\hatchshift}
\newcommand{\hatchcolor}{}
\tikzset{hatchspread/.code={\setlength{\hatchspread}{#1}},
         hatchthickness/.code={\setlength{\hatchthickness}{#1}},
         hatchshift/.code={\setlength{\hatchshift}{#1}},
         hatchcolor/.code={\renewcommand{\hatchcolor}{#1}}}
\tikzset{hatchspread=3pt,
         hatchthickness=0.4pt,
         hatchshift=0pt,
         hatchcolor=black}
\pgfdeclarepatternformonly[\hatchspread,\hatchthickness,\hatchshift,\hatchcolor]
   {custom north west lines}
   {\pgfqpoint{\dimexpr-2\hatchthickness}{\dimexpr-2\hatchthickness}}
   {\pgfqpoint{\dimexpr\hatchspread+2\hatchthickness}{\dimexpr\hatchspread+2\hatchthickness}}
   {\pgfqpoint{\dimexpr\hatchspread}{\dimexpr\hatchspread}}
   {
    \pgfsetlinewidth{\hatchthickness}
    \pgfpathmoveto{\pgfqpoint{0pt}{\dimexpr\hatchspread+\hatchshift}}
    \pgfpathlineto{\pgfqpoint{\dimexpr\hatchspread+0.15pt+\hatchshift}{-0.15pt}}
    \ifdim \hatchshift > 0pt
      \pgfpathmoveto{\pgfqpoint{0pt}{\hatchshift}}
      \pgfpathlineto{\pgfqpoint{\dimexpr0.15pt+\hatchshift}{-0.15pt}}
    \fi
    \pgfsetstrokecolor{\hatchcolor}
    \pgfusepath{stroke}
   }

\pgfdeclarepatternformonly[\hatchspread,\hatchthickness,\hatchshift,\hatchcolor]
   {custom north east lines}
   {\pgfqpoint{\dimexpr-2\hatchthickness}{\dimexpr-2\hatchthickness}}
   {\pgfqpoint{\dimexpr\hatchspread+2\hatchthickness}{\dimexpr\hatchspread+2\hatchthickness}}
   {\pgfqpoint{\dimexpr\hatchspread}{\dimexpr\hatchspread}}
   {
    \pgfsetlinewidth{\hatchthickness}
    \pgfpathmoveto{\pgfqpoint{\dimexpr\hatchshift-0.15pt}{-0.15pt}}
    \pgfpathlineto{\pgfqpoint{\dimexpr\hatchspread+0.15pt}{\dimexpr\hatchspread-\hatchshift+0.15pt}}
    \ifdim \hatchshift > 0pt
      \pgfpathmoveto{\pgfqpoint{-0.15pt}{\dimexpr\hatchspread-\hatchshift-0.15pt}}
      \pgfpathlineto{\pgfqpoint{\dimexpr\hatchshift+0.15pt}{\dimexpr\hatchspread+0.15pt}}
    \fi
    \pgfsetstrokecolor{\hatchcolor}
    \pgfusepath{stroke}
   }

\makeatletter
\newcommand*{\centerfloat}{%
  \parindent \z@
  \leftskip \z@ \@plus 1fil \@minus \textwidth
  \rightskip\leftskip
  \parfillskip \z@skip}
\makeatother

\usepackage{xparse}

\ExplSyntaxOn
\NewDocumentCommand{\makeabbrev}{mmm}
 {
  \yoruk_makeabbrev:nnn { #1 } { #2 } { #3 }
 }

\cs_new_protected:Npn \yoruk_makeabbrev:nnn #1 #2 #3
 {
  \clist_map_inline:nn { #3 }
   {
    \cs_new_protected:cpn { #2 } { #1 { ##1 } }
   }
 }
 \ExplSyntaxOff

\makeabbrev{\textbf}{tbf#1}{a,b,c,d,e,f,g,h,i,j,k,l,m,n,o,p,q,r,s,t,u,v,w,x,y,z,A,B,C,D,E,F,G,H,I,J,K,L,M,N,O,P,Q,R,S,T,U,V,W,X,Y,Z}

\makeabbrev{\textbf}{bf#1}{a,b,c,d,e,f,g,h,i,j,k,l,m,n,o,p,q,r,s,t,u,v,w,x,y,z,A,B,C,D,E,F,G,H,I,J,K,L,M,N,O,P,Q,R,S,T,U,V,W,X,Y,Z}

\makeabbrev{\textsf}{tsf#1}{a,b,c,d,e,f,g,h,i,j,k,l,m,n,o,p,q,r,s,t,u,v,w,x,y,z,A,B,C,D,E,F,G,H,I,J,K,L,M,N,O,P,Q,R,S,T,U,V,W,X,Y,Z}

\makeabbrev{\mathsf}{mss#1}{a,b,c,d,e,f,g,h,i,j,k,l,m,n,o,p,q,r,s,t,u,v,w,x,y,z,A,B,C,D,E,F,G,H,I,J,K,L,M,N,O,P,Q,R,S,T,U,V,W,X,Y,Z}

\makeabbrev{\mathfrak}{mf#1}{a,b,c,d,e,f,g,h,i,j,k,l,m,n,o,p,q,r,s,t,u,v,w,x,y,z,A,B,C,D,E,F,G,H,I,J,K,L,M,N,O,P,Q,R,S,T,U,V,W,X,Y,Z}

\makeabbrev{\mathrm}{mrm#1}{a,b,c,d,e,f,g,h,i,j,k,l,m,n,o,p,q,r,s,t,u,v,w,x,y,z,A,B,C,D,E,F,G,H,I,J,K,L,M,N,O,P,Q,R,S,T,U,V,W,X,Y,Z}

\makeabbrev{\mathbf}{mbf#1}{a,b,c,d,e,f,g,h,i,j,k,l,m,n,o,p,q,r,s,t,u,v,w,x,y,z,A,B,C,D,E,F,G,H,I,J,K,L,M,N,O,P,Q,R,S,T,U,V,W,X,Y,Z}

\makeabbrev{\mathcal}{mc#1}{A,B,C,D,E,F,G,H,I,J,K,L,M,N,O,P,Q,R,S,T,U,V,W,X,Y,Z}

\makeabbrev{\mathbb}{mbb#1}{A,B,C,D,E,F,G,H,I,J,K,L,M,N,O,P,Q,R,S,T,U,V,W,X,Y,Z}

\makeabbrev{\mathscr}{ms#1}{A,B,C,D,E,F,G,H,I,J,K,L,M,N,O,P,Q,R,S,T,U,V,W,X,Y,Z}

\makeabbrev{\mathrm}{#1}{
Id,id,ran,rk,diag,stab,ann,conv,pr,ev,tr,End,Hom,sgn,im,op,can,fin,ext,red,tot,
rot,usc,lsc,Lip,LocLip,lip,bSymLip,osc,AC,loc,uloc,spec,coz,z,ul,
supp,Opt,Adm,Cpl,Geo,GeoSel,GeoOpt,GeoAdm,GeoCpl,reg,
bd,co,Ric,Exp,dExp,dist,seg,Seg,cut,fcut,Cut,SDiff,Iso,Isom,diam,cl,Homeo,Diff,Der,vol,dvol,inj,relint, Graph, sub,codim,
var,law,Var,Poi,Gam,pa,so,iso,fs,inv,pqi,mix,
TestF,
}

\makeabbrev{\mathsf}{#1}{DP,CD,BE,MCP,wMTW,MTW,RCD,ncRCD,QCD,EVI,Irr,IH,SC,wFe,VA,UP,Curv,Alex,CAT}

\newcommand{\Ent}{\mathcal H}

\newcommand{\T}{\tau} 
\newcommand{\A}{\Sigma} 
\newcommand{\Bo}[1]{\msB_{#1}} 

\newcommand{\Bdd}[1]{\msO_{#1}}
\newcommand{\Ed}{{\msE_\mssd}}  

\renewcommand{\complement}{\mathrm{c}}

\newcommand{\emparg}{{\,\cdot\,}}
\DeclareMathOperator*{\argmin}{argmin}

\newcommand{\Ch}{\mathsf{Ch}}

\renewcommand{\Cap}{\mathrm{Cap}}

\newcommand{\dom}[1]{\msD(#1)}

\newcommand{\Lipu}{\mathrm{Lip}^1}
\newcommand{\Lipua}{\mathrm{Lip}^\alpha}

\DeclareMathOperator{\eqdef}{\coloneqq}

\let\epsilon\varepsilon

\let\temp\phi
\let\phi\varphi
\let\varphi\temp

\newcommand{\rar}{\rightarrow}

\newcommand{\diff}{\mathop{}\!\mathrm{d}}

\newcommand{\tabs}[1]{\big\lvert#1\big\rvert}	

\newcommand{\norm}[1]{\left\lVert#1\right\rVert}					
\newcommand{\set}[1]{\left\{#1\right\}}							
\newcommand{\paren}[1]{\left(#1\right)}							

\newcommand{\ttonde}[1]{\big({#1}\big)}
\newcommand{\class}[2][]{\left[#2\right]_{#1}}						
\newcommand{\tclass}[2][]{\big [#2\big]_{#1}}						

\newcommand{\tym}[1]{{\scriptscriptstyle{\times #1}}}
\newcommand{\otym}[1]{{\scriptscriptstyle{\otimes #1}}}
\newcommand{\osym}[1]{{\scriptscriptstyle{\odot #1}}}

\DeclareSymbolFont{symbolsC}{U}{pxsyc}{m}{n}
\SetSymbolFont{symbolsC}{bold}{U}{pxsyc}{bx}{n}
\DeclareFontSubstitution{U}{pxsyc}{m}{n}
\DeclareMathSymbol{\medcirc}{\mathbin}{symbolsC}{7}
\DeclareSymbolFont{symbolsZ}{OMS}{pxsy}{m}{n}
\SetSymbolFont{symbolsZ}{bold}{OMS}{pxsy}{bx}{n}
\DeclareFontSubstitution{OMS}{pxsy}{m}{n}

\newcommand{\seq}[1]{\paren{#1}}								

\newcommand{\Cb}{\mcC_b}									
\newcommand{\Cc}{\mcC_c}									

\newcommand{\pfwd}{\sharp}
\DeclareMathOperator*{\esssup}{esssup}
\DeclareMathOperator*{\essinf}{essinf}

\DeclareMathOperator{\car}{\mathds 1}

\DeclareMathOperator{\emp}{\varnothing} 
\newcommand{\N}{{\mathbb N}}
\newcommand{\R}{{\mathbb R}}

\newcommand{\vC}{\mathcal C} 
\DeclareMathOperator{\Z}{{\mathbb Z}}

\newcommand{\TLDS}{\textsc{tlds}\xspace}

\newcommand{\parEMLDS}{\textsc{mlds}\xspace}

\newcommand{\quot}{{\sf P}}

\newcommand{\mrestr}[1]{\!\downharpoonright_{#1}}

\allowdisplaybreaks

\usetikzlibrary{shapes.misc}
\tikzset{cross/.style={cross out, draw=black, minimum size=2*(#1-\pgflinewidth), inner sep=0pt, outer sep=0pt},
cross/.default={4pt}}

\newcommand{\iref}[1]{\ref{#1}}

\newcommand{\comma}{\,\,\mathrm{,}\;\,}
\newcommand{\semicolon}{\,\,\mathrm{;}\;\,}
\newcommand{\fstop}{\,\,\mathrm{.}}

\newcommand{\cdc}{\Gamma}
\newcommand{\vcdc}{\Gamma}

\newcommand{\EE}[2]{\mcE^{#1, #2}}

\usepackage{scrextend}						

\newcommand{\purple}[1]{{\color{purple}#1}}

\newcommand{\Dz}{\mcD} 
 
\newcommand{\cpl}{q}
\newcommand{\QP}{{\mu}}

\newcommand{\vmssd}{\hat\mssd}

\newcommand{\EN}{\overline{\N}}
\newcommand{\dUpsilon}{{\boldsymbol\Upsilon}}
\newcommand{\proj}{{\sf proj}}

\newcommand{\cquad}{\comma \quad}

\newcommand{\e}{\epsilon}

\newcommand{\U}{\dUpsilon}
\newcommand{\E}{\mathcal E}
\newcommand{\vE}{{\mathcal E}}
\newcommand{\vvE}{\underline{\mathcal E}}
\newcommand{\F}{\mathcal F}
\renewcommand{\1}{\mathbf 1}

\newcommand{\comm}{\,\,\mathrm{,}\;\,}

\newcommand{\MM}{\mathsf{MM}}

\newcommand{\sine}{\mathsf{sine}}
\newcommand{\Airy}{\mathsf{Airy}}
\newcommand{\Bessel}{\mathsf{Bessel}}
\newcommand{\Ginibre}{\mathsf{Ginibre}}

\newcommand{\Riesz}{{\sf Riesz}}

\newcommand{\Fis}{\mathsf{F}}

\numberwithin{equation}{section}
\theoremstyle{plain}	
\newtheorem{thm}{Theorem}[section]
\newtheorem*{thm*}{Theorem A}
\newtheorem*{thm**}{Theorem B}
\newtheorem*{mthm*}{Main Theorem}

\newtheorem{prop}[thm]{Proposition}
\newtheorem{lem}[thm]{Lemma}
\newtheorem{cor}[thm]{Corollary}
\newtheorem*{cor*}{Corollary A}
\newtheorem*{cor**}{Corollary B}
\newtheorem*{cor***}{Corollary C}

\theoremstyle{definition}
\newtheorem{defs}[thm]{Definition}
\newtheorem*{defs*}{Definition}
\theoremstyle{remark}

\newtheorem{rem}[thm]{\bf Remark}
\newtheorem{ese}[thm]{\bf Example}
\newtheorem{ass}[thm]{\bf Assumption}

\newtheorem*{asm*}{{\bf List of Assumptions}}

\renewcommand{\paragraph}[1]{\medskip\emph{#1}.}
\renewcommand{\#}{\sharp}

\newcommand{\scolon}{\ ;}
\newcommand{\Ai}{\mathsf{Ai}}
\renewcommand{\supp}{{\rm supp}}

\makeatletter
\@namedef{subjclassname@2020}{%
  \textup{2020} Mathematics Subject Classification}
\makeatother

\begin{document}
\title[Infinite Interacting Brownian Motions and  EVI Gradient Flows]{Infinite Interacting Brownian Motions and  EVI Gradient Flows}

\author[K.~Suzuki]{Kohei Suzuki}
\address{Department of Mathematical Science, Durham University, South Road, DH1 3LE, United Kingdom}
\email{kohei.suzuki@durham.ac.uk}
\thanks{\hspace{-5.5mm}  Department of Mathematical Science, Durham University, South Road, DH1 3LE, United Kingdom
\\
\hspace{2.0mm} E-mail: kohei.suzuki@durham.ac.uk
}

\keywords{\vspace{2mm}Interacting Brownian Motions, EVI Gradient Flow, RCD, Number rigidity}

\subjclass[2020]{37A30, 31C25, 30L99, 70F45, 60G55}

\renewcommand{\abstractname}{\normalsize  Abstract}
\begin{abstract} \normalsize
We give a sufficient condition under which the time-marginal law of $\QP$-reversible infinite interacting Brownian motions is characterised as the steepest gradient descent of the relative entropy in the Wasserstein space in the sense of evolution variational inequality (EVI). This is an infinite-dimensional generalisation of Jordan--Kinderlehrer--Otto/Ambrosio--Gigli--Savar\'e theory.  Consequently, the configuration space (the space of locally finite point measures) endowed with the reversible measure~$\QP$ is an $\RCD$ space and the time-marginal law  is identified to the heat flow on this space.  Our result covers the infinite-dimensional Dyson Brownian motion with bulk and soft-edge limits; the latter yields Airy line ensemble as its stationary process, a central object in KPZ universality. Our result therefore provides an optimal transport characterisation of these models as Wasserstein gradient flows, and establishes a range of new functional inequalities (HWI, distorted Brunn--Minkowski, dimension-free Harnack  and many others)  as a corollary.  As an application, we discover  the new phenomena, {\it dynamical  number rigidity} and {\it dynamical tail triviality},  that the time-marginal law possesses number rigidity and tail triviality for every time $t>0$,   revealing a propagation of random crystal and extremal structures by the Dyson Brownian motions. 
\end{abstract}

\maketitle

\setcounter{tocdepth}{1}
\makeatletter
\def\l@subsection{\@tocline{2}{0pt}{2.5pc}{5pc}{}}
\def\l@subsubsection{\@tocline{3}{0pt}{4.75pc}{5pc}{}}
\makeatother

\tableofcontents
\section{Introduction}

\subsection{Overview}
In recent years, the theory of gradient flows in the Wasserstein space has become a central tool in analysis and geometry. Following Jordan--Kinderlehrer--Otto~\cite{JorKinOtt98} and Ambrosio--Gigli--Savaré~\cite{AmbGigSav08}, the time-marginal law of  diffusion processes (equivalently, the solution to Fokker--Planck equations) can be characterised  as  \emph{gradient flows}, or in a stronger sense,  as \emph{Evolution Variational Inequality (EVI)-gradient flows} of an entropy in  the space of probability measures endowed with transportation distance (a.k.a.~Monge--Kantorovich--Rubinstein--Wasserstein distance), 
thereby linking probabilistic dynamics to the steepest gradient descent (i.e., gradient flow) of the entropy in the Wasserstein space. 


In this paper we extend this paradigm to the configuration space~$\U=\U(\R^n)$ (the space of  locally finite point measures on~$\R^n$) --- an infinite-dimensional state space  embracing  a variety of interacting particle diffusion processes.  As a main result, we provide a sufficient condition under which interacting Brownian motions can be identified with the EVI gradient flow in the Wasserstein space. A central example is   the infinite-dimensional Dyson Brownian motions (DBM) of inverse temperature $\beta=2$ with bulk/soft-edge limit. 
Notably the the soft-edge limit case yields the {\it Airy line ensemble} as its stationary solution
(Corwin--Hammond~\cite{CorHam11} and Osada--Tanemura~\cite{OsaTan16}), which is an infinite collection of non-intersecting random curves known as scaling limit of many models in  Kardar--Parisi--Zhang (KPZ) universality describing a universal model of random growths of interfaces. Our result therefore provides an optimal transport characterisation of these models  as a Wasserstein steepest gradient descent.

As an application, we obtain 
\begin{enumerate}[(a)]
\item  {\it Functional inequalities} (HWI, distorted Brunn--Minkowski,  geodesic (displacement) convexity of the entropy, dimension-free Harnack,  and several others) providing new quantitative estimates of model of our interests; 
\vspace{2mm}
\item {\it $\RCD$ property}:  We identify the time-marginal of interacting Brownian motions as the heat flows on the infinite-dimensional space $\U$ whose synthetic Ricci curvature is uniformly bounded from below in the sense of Lott--Villani~\cite{LotVil09}, Sturm~\cite{Stu06a,Stu06b} and Gigli~\cite{Gig12a}; 
\vspace{2mm}
\item {\it Dynamical number rigidity/tail triviality}. We discovered new phenomena of  infinite interacting Brownian motions:  its time-marginal law for any positive time $t>0$ obtains the number rigidity (in the sense of Ghosh--Peres~\cite{GhoPer17}) as well as  the tail triviality. The former is in particular a sharp contrast to the case of independent Brownian particles. 
\end{enumerate}
The number rigidity was discovered in Ghosh--Peres~\cite{GhoPer17}, which is a sort of crystal-like structure in random point clouds: you can precisely know how many particles are  located inside a given bounded set if you know the configuration outside. For Gibbs point processes, the tail triviality is equivalent to the extremeness in the space of Gibbs measures with a given specification. From this perspective, our result (c) reveals that the time-marginal of infinite interacting Brownian motions {\it propagates random crystal-like structures as well as extremeness} for all $t>0$. 

\subsection{Main results in particular cases} 
Before giving the fully general statement of our main results, we state them in two particular models of one-dimensional interacting Brownian motions: infinite Dyson Brownian motion with bulk/soft-edge limits. 
\begin{itemize}
\item bulk limit Dyson Brownian motion:
\begin{align} \label{DBMB} \tag*{${\rm (DBM)}_{{\rm bulk}}$}
\diff X^i_t =  \lim_{r \to \infty}\sum_{j: |X_t^i-X_t^j|<r}^\infty \frac{\diff t}{X_t^i-X_t^j} + \diff B_t^i \cquad i \in \N  \semicolon
\end{align}
\item soft-edge limit Dyson Brownian motion 
\begin{align}   \label{DBME} \tag*{${\rm (DBM)}_{{\rm edge}}$} 
\diff X^i_t =    \lim_{r \to \infty}\Biggl( \biggl(\sum_{\substack{j: j \neq i\\  |X_t^j|<r}}^\infty \frac{1}{X_t^i-X_t^j} \biggr) -\int_{-r}^r \frac{\hat{\rho}(x)}{-x} \diff x \Biggr) \diff t + \diff B_t^i \cquad i \in \N  \comma
\end{align}
\end{itemize}
where $B_t^1, B_t^2, \ldots$ are infinitely many independent Brownian motions in~$\R$  and $\hat{\rho}(x)=\frac{\sqrt{-x}}{\pi} \1_{\{x<0\}}$.  In both cases, the interaction (i.e., the first term in the above SDEs) comes from a logarithmic potential, which  is long-range and singular.
For each SDE, under suitable initial conditions, the infinite sum in the interaction part stays finite due to a cancellation between negative and positive terms as well as  the centralisation in the case of~\ref{DBME}, and the unique strong solution $\mbbX=(X_t^1, X_t^2, \ldots)_{t \in \R_+}$  exists. For each fixed $t \ge 0$, the sequence $(X_t^i)_{i \in \N}$ does not have accumulation points in $\R$ almost surely. Hence, the symmetric quotient 
$$(x_1,x_2,\ldots)  \mapsto \sum_{i=1}^\infty \delta_{x_i} \in \U$$ maps the solution $\mbbX$ to a diffusion process in~$\U=\U(\R)$, called {\it  unlabelled solution}.  The reversible probability measure~$\QP$ for the unlabelled solution is $\QP=\sine_2$ for \ref{DBMB} and $\QP=\Airy_2$ for~\ref{DBME},  each of which is the law of the point process in $\R$ obtained as the bulk and the soft-edge scaling limit of the eigenvalue distributions of Gaussian random matrices. The unlabelled solutions to \ref{DBMB} and \ref{DBME} are well-posed for almost every initial configuration with respect to~$\QP$. We refer the readers to~\cite{Spo87, Osa96, Osa12, Tsa16, OsaTan20} for \ref{DBMB}, and \cite{OsaTan14} for \ref{DBME}. For more details, see Examples~\ref{exa: S} and \ref{exa: A}.

\paragraph{Airy line ensemble} If~\ref{DBME} starts with its reversible measure~$\QP=\Airy_2$ (i.e., stationary solution), the solution is a space-time determinantal point process with~extended Airy kernel, see~\cite[Thm.~2.1 and Thm.~2.2]{OsaTan16} also~\cite{NagFor98, KatTan09, KatTan10, KatTan11}. The space time point processt is called  {\it Airy line ensemble}, see \cite[Thm.~3.1]{CorHam11} and \cite[Dfn.~1.3]{CorSun14}.  It is known that, in multi-line (multi-layer) versions of KPZ universality models such as multi-layer polynuclear growth and non-intersecting line ensembles behind last-passage percolation,  the entire stack of top $k$ interfaces, properly centred and scaled, converges to the Airy line ensemble. We refer the readers to e.g.,~\cite{PraSpo01, Joh02, AggHua23} and references therein. 
%

\paragraph{Optimal transport in the configuration space} 
We endow $\U$ with the vague topology $\tau_\mrmv$ (the duality by compactly supported continuous functions on $\R$). 
To describe a metric structure on the configuration space~$\U$, we use 
{\it $\ell^2$-matching extended distance} $\mssd_\U$ on $\U$: 
\begin{align*}
\mssd_{\dUpsilon}(\gamma,\eta)\eqdef \inf_{\cpl\in\Cpl(\gamma,\eta)} \paren{\int_{\R \times \R} |x-y|^2 \diff\cpl(x,y)}^{1/2} \quad \gamma, \eta \in \U \comma
\end{align*}
where~$\Cpl(\gamma,\eta)$ is the set of all couplings of~$\gamma$ and~$\eta$ and we define $\inf{\emp}=+\infty$. Note that the topology $\tau_{\mssd_\U}$ induced by~$\mssd_\U$ is strictly stronger than $\tau_\mrmv$. The {\it $L^p$-transportation extended distance} $\mssW_{p, \mssd_\U}$ on the space $\mcP(\U)=\mcP(\U, \msB(\tau_\mrmv))$ of Borel probability measures on $\U$  is  defined as 
\begin{align*}
\mssW_{p, \mssd_\U}(\nu, \sigma)^p:=\inf_{{\sf c} \in {\sf Cpl}(\nu, \sigma)} \int_{\U^{\times 2}} \mssd_{\U}^p(\gamma, \eta) \diff {\sf c}(\gamma, \eta) \fstop 
\end{align*}
Let $\mcP_\QP(\U) \subset \mcP(\U)$ be the space of absolutely continuous probability measures with respect to $\QP$. We write $\nu=\rho\cdot\QP$ when $\rho=\frac{\diff \nu}{\diff \QP}$ for $\nu \in \mcP_\QP(\U)$. 
The {\it relative entropy} is defined as  
$$\Ent_{\QP}(\nu):=\int_\U \rho \log \rho \diff \QP \qquad \nu=\rho\cdot \QP$$
with the domain $\dom{\Ent_\QP}:=\{\nu \in \mcP_\QP(\U): \Ent_\QP(\nu)<+\infty\}$. 

\smallskip
\paragraph{EVI property of time-marginal law}
We denote by $(\mbbP_\gamma, \mssX^{\QP}_t)$ the unlabelled solution to~\ref{DBMB} (if $\QP=\sine_2$) and \ref{DBME} (if $\QP=\Airy_2$) with the initial configuration~$X_0^\QP=\gamma$, which is well-posed for $\QP$-a.e.~$\gamma$. 
Let  $\mcT^{\U, \QP}_t \sigma$ be the time-marginal law of $\mssX^{\QP}_t$ with an initial law $\sigma$. If $\sigma\ll\QP$, it is well-defined and  we have the expression 
$$\mcT^{\U, \QP}_t \sigma(A)=\int_\U \mbbP_\gamma[\mssX^{\QP}_t \in A] \diff \sigma(\gamma) \qquad A \in \msB(\U) \fstop $$
The well-definedness of $\mcT^{\U, \QP}_t \sigma$ further extends to the case without the absolute continuity~$\sigma \ll \QP$ (see Theorem~A and \ref{c:FWD-2} of Thm.~\ref{c:FWD} below for admissible initial laws).  
For a function $f: [0, \infty) \to \R$ and $t \in [0, \infty)$, {\it the upper right Dini derivative} is defined as $\frac{\diff^+}{\diff t}f(t):=\limsup_{h \downarrow 0} (f(t+h)-f(t))/h$.
Our main result is to show the EVI of the time-marginal. 
 \begin{thm*}[$\EVI$] \label{ov:ETD} Let $\mu=\sine_2$ or $\mu=\Airy_2$. 
For $\nu\in \dom{\Ent_\QP}$ and $\sigma \in \mcP(\U)$ with $\mssW_{2, \mssd_{\U}}(\sigma, \nu)<\infty$,  
\begin{itemize}
\item $\mathcal T^{\U, \QP}_t \sigma$ is well-defined, $\mathcal T^{\U, \QP}_t \sigma \in \dom{\Ent_{\QP}}$ and ${\mssW}_{2, \mssd_{\U}}\bigl({\mathcal T^{\U, \QP}_t \sigma}, \nu \bigr)<\infty$ for every $t>0$$;$
\item  $t \mapsto \mathcal T^{\U, \QP}_t \sigma$ satisfies the $\EVI_0$ inequality:
  \begin{align} \label{EVI0} 
\frac{1}{2}\frac{\diff^+}{\diff t}{\mssW}_{2, \mssd_{\U}}\bigl({\mathcal T^{\U, \QP}_t \sigma}, \nu \bigr)^2  \le \Ent_{\mu}({\nu}) - \Ent_{\mu}({\mathcal T^{\U, \QP}_t \sigma})\fstop 
\end{align}
\end{itemize}

\end{thm*}

\subsection{Application} 
The strength of the $\EVI$~\eqref{EVI0} is that it \emph{identifies} the time-marginal law $\nu_t=\mcT_t^{\U, \QP}\nu$ as the {\it unique} Wasserstein gradient flow of the entropy
$$``\partial_t \nu_t = - \nabla_{\mssW_{2, \mssd_\U}} \Ent_\QP(\nu_t)" \quad t>0$$
in the sense of Jordan--Kinderlehrer--Otto scheme in \cite{JorKinOtt98} as well as the curve of maximal slope in Ambrosio--Gigli--Savar\'e~\cite{AmbGigSav08}.  Our result therefore unifies the two central one-dimensional universality regimes --- the Dyson Brownian motions  in the bulk ($\sine_2$)  and the soft edge ($\Airy_2$) regimes --- within a single framework of the Wasserstein gradient flows of the relative entropy and provides a new optimal transport characterisation of these models.
Let 
$$\mcP_e(\U):=\{\nu \in \mcP(\U): \mssW_{2, \mssd_\U}(\nu, \sigma)<+\infty \ \text{for some $\sigma \in \dom{\Ent_\QP}$}\} \fstop$$

 \begin{cor*}[{Gradient flow: Cor.~\ref{c:GF},~\ref{c:MM}}] \label{i:GC} Let $\mu=\sine_2$ or $\mu=\Airy_2$.    The following hold: 
\begin{enumerate}[{\rm (a)}]

\item \label{c:2-0} $(${\bf  curves of maximal slope}$)$ The dual semigroup~$\bigl(\mathcal T_t^{\U, \QP}\nu_0\bigr)_{t \ge 0}$ is the unique solution to the $\mssW_{2, \mssd_{\U}}$-gradient flow of $\Ent_{\QP}$ starting at $\nu_0$ in the sense of {\it  curves of maximal slope}. 
Namely, for every $\nu_0 \in \mcP_e(\U)$, the curve $[0, \infty) \ni t \mapsto \nu_t=\mathcal T_t^{\U, \QP} \nu_0$ satisfies the energy equality:
\begin{align} \label{e:CMS}
\frac{\diff}{\diff t} \Ent_\QP({\nu_t}) = -|\dot\nu_t|^2 = -|{\sf D}_{\mssW_{2, \mssd_{\U}}}^- \Ent_\QP|^2(\nu_t) \quad\text{a.e.~$t>0$} \comma
\end{align}
where $|{\sf D}_{\mssW_{2, \mssd_{\U}}}^- \mcH_\QP|(\nu) := \limsup_{\sigma \to \nu} \frac{(\mcH_\QP(\sigma)-\mcH_\QP(\nu))^-}{\mssW_{2, \mssd_\U}(\sigma, \nu)}$ and $|\dot\nu_t| :=\lim_{s \to t} \frac{\mssW_{2, \mssd_\U}(\nu_s, \nu_t)}{|s-t|}$. 
\item  $(${\bf  minimising movement (JKO) scheme}$)$  The time-marginal $t \mapsto \mathcal T_t^{\U, \QP}\nu_0$ is the unique element of the minimising movement (a.k.a JKO scheme) $\MM(\mcP_{[\nu_0]}(\U), \mssW_{2, \mssd_\U}, \Ent_\QP)$ for every $\nu_0 \in \dom{\Ent_\QP}$ (see Dfn.~\ref{d:MMS} for the precise definition of $\MM$). 
\end{enumerate}
\end{cor*}
Due to the ergodicity of~unlabelled solutions for $\QP=\sine_2$ (see \cite{OsaOsa23,Suz23b}) and $\QP= \Airy_2$ (see \cite{Suz23b}), we know that  $\nu_t \to \mu$ weakly as $t \to \infty.$ 
 Combined with~Cor.~\ref{i:GC}, we obtain the following picture about the trajectory of the time marginal~$\nu_t$ of the law: 
 \vspace{1.0mm}
 \begin{align*}
 &\text{\it The curve $t \mapsto \nu_t$ moves downward along the $\mssW_{2, \mssd_\U}$-steepest gradient descent of the relative entropy~$\Ent_\QP$}
 \\ 
  &\text{\it and weakly converges to the unique minimiser~$\QP$ of~$\Ent_\QP$ as $t \to \infty$}.  
  \end{align*}
 
 Furthermore, the $\EVI$ property simultaneously yields various functional inequalities: $\mssW_2$-geodesic (displacement) convexity of the relative entropy, HWI and distorted Brunn--Minkowski inequalities and several others --- providing new quantitative estimates on the Dyson Brownian motions. 
Moreover, the time marginal $\mcT^{\U, \QP}_t$ can be identified with the heat flow on the space $(\U, \tau_\mrmv, \mssd_\U, \QP)$ having  synthetic Ricci lower bound $\RCD$. 
For other functional inequalities including the  dimension-free Harnack inequality, the $\mssW_{2, \mssd_\U}$-contraction property of $\mcT_t^{\U,\QP}$ and the Bakry--\'Emery gradient estimate,  see Thm.~\ref{c:FWD} below.
\begin{cor**}[{Cor.~\ref{c:LST}, \ref{c:GC}}] \label{i:GC} Let $\mu=\sine_2$ or $\mu=\Airy_2$.    The following hold: 
\begin{enumerate}[{\rm (a)}]
\item \label{c:2-i} $(${\bf geodesical convexity}$)$ The space~$(\dom{\Ent_\QP}, {\mssW}_{2, \mssd_\U})$ is an extended geodesic metric space. Furthermore, the entropy~$\Ent_\QP$ is geodesically convex along every ${\mssW}_{2, \mssd_\U}$-geodesic $(\nu_t)_{t \in [0,1]}$$:$ for every $\nu_0, \nu_1 \in \dom{\Ent_\QP}$ with $\mssW_{2, \mssd_\U}(\nu_0, \nu_1)<\infty$, there exists a $\mssW_{2, \mssd_\U}$-constant speed geodesic $(v_t)_{t \in [0,1]}$ connecting $\nu_0$ and $\nu_1$ such that 
\begin{align*}
\Ent_\QP(\nu_t) \le (1-t)\Ent_{\QP}(\nu_0) + t \Ent_{\QP}(\nu_1)\cquad t \in [0,1] \fstop
\end{align*}
\item  \label{c:6-i-00} $(${\bf RCD property}$)$ The extended metric measure space $(\U, \tau_\mrmv, \mssd_\U, \QP)$ is an $\RCD(0,\infty)$ space. For $\QP=\sine_2$, the curvature bound $K=0$ is the optimal constant. 
Furthermore, $(\mcT_t^{\U, \QP})_{t \ge 0}$ is identified with the heat flow on $\U$ (see also Thm.~\ref{it:IDD}).
\vspace{2mm}
\item  \label{c:6-i}  $(${\bf Distorted Brunn--Minkowski inequality}$)$ Let $A_0, A_1 \subset \U$ be Borel sets with $\QP(A_0)\QP(A_1)>0$ and define $\nu_0:=\frac{\1_{A_0}}{\QP(A_0)} \cdot \mu$ and $\nu_1:=\frac{\1_{A_1}}{\QP(A_1)} \cdot \mu$. Suppose $\mssW_{2, \mssd_{\U}}(\nu_0, \nu_1)<\infty$. Define the $t$-barycentre $(0 \le t \le 1)$ as 
$$[A_0, A_1]_t:=\{\eta \in \U: \eta=\gamma_t,\ \gamma \in {\rm Geo}(\U, \mssd_\U),\ \gamma_0 \in A_0,\ \gamma_1 \in A_1\} \comma$$
where ${\rm Geo}(\U, \mssd_\U)$ is the space of constant speed geodesics in $\U$. 
 Then, 
\begin{align*}
\log \frac{1}{\mu\bigl([A_0, A_1]_t \bigr)} \le (1-t)	\log \frac{1}{\mu(A_0)} + t \log \frac{1}{\mu(A_1)}  \fstop
\end{align*}
\item  \label{c:7-i}  $(${\bf HWI inequality}$)$ For~$\nu_0 \in \dom{\Ent_\QP}, \nu_1 \in \mcP(\U)$, 
$$\Ent_{\QP}(\nu_0) \le \Ent_{\QP}(\nu_1)+2\mssW_{2, \mssd_{\U}}(\nu_0, \nu_1)\Fis_{\QP}(\nu_0)^{1/2} \comma$$
where $\Fis$ is {\it the Fisher information} ${\sf F}_\QP(\nu) =\frac{1}{4}|{\sf D}_{{\mssW}_{2, \mssd_\U}}^- \Ent_\QP(\nu)|^2$.
\end{enumerate}
\end{cor**}
The optimality of $K=0$ for $\QP=\sine_2$ in \ref{c:6-i-00} follows by the lack of the spectral gap~proved in~\cite{Suz24}.

 \subsection*{Dynamical number rigidity/tail triviality} We present a dynamical rigidity/tail triviality of interacting Brownian motions. 
We recall that a Borel probability measure $\nu$ on $\U$ has 
 \begin{itemize}
 \item  {\it number rigidity} \ref{ass:NR} if, for every bounded Borel set~$E\subset X$, there exists a $\nu$-measurable set $\Omega \subset \U$ with $\nu(\Omega)=1$ and a Borel function $\Phi: \U(E^c) \to \N_0$ such that 
 \begin{align*}\tag*{$(\mathsf{NR})$}\label{ass:NR}
 \gamma(E) = \Phi(\gamma|_{E^c}) \qquad \text{for every $\gamma \in \Omega$} \scolon
\end{align*}
%
 \item {\it tail triviality} \ref{ass:T} if 
\begin{align}\tag*{$(\mathsf{TT})$}\label{ass:T}
\nu(\Xi)\in \{0, 1\} \quad  \text{whenever} \quad \Xi \in \mathscr T(\U) \comma
\end{align}
where $\mathscr T(\U):=\cap_{r \in \N}\sigma({\rm pr}_{B_r^c})$ is the {\it tail $\sigma$-algebra} and $\pr_{B_r^c}: \U \to \U(B_r^c)$ is defined as $\gamma \mapsto \gamma_{B_r^c}:=\gamma\mrestr{B_r^c}$ with $B_r=\{x \in \R^n: |x|<r\}$. 
 \end{itemize}
\begin{cor***}[Dynamical number rigidity/tail triviality,  Cor.~\ref{t:NRF}]  \label{i:PNT} Let $\QP=\sine_2$ or $\Airy_2$. For every $\nu \in \mcP_e(\U)$, the time-marginal~$\mathcal T_t^{\U, \QP}\nu$ satisfies \ref{ass:NR} and~\ref{ass:T} for every $t>0$.
\end{cor***}
We note that an initial law $\nu \in \mcP_e(\U)$ is in general not absolutely continuous with respect to $\QP$ nor possessing~\ref{ass:T} (see Prop.~\ref{p:ENT} for such an element), however, interestingly there properties emerge in the time-marginal~$\mcT_t^{\U, \QP}\nu$ for any positive $t>0$. 
Corollary C reveals  new dynamical rigidity/tail triviality phenomena for time-marginals of infinite interacting Brownian motions: 
$$\text{the curve $(0, \infty] \ni t \mapsto \mcT_t^{\U, \QP}\nu$ stays in the subspace $\mcP_{\rm NR}(\U) \cap \mcP_{\rm TT}(\U)$} \comma$$
of probability measures possessing the number rigidity and the tail triviality. This is a sharp contrast to the case of independent Brownian particles.

\subsection{A general framework} To state our main result in full generality, we work on the configuration space~$\U=\U(\R^n)$ over $\R^n$ endowed with a Borel probability measure $\QP \in \mcP(\U)$.
The class of diffusion processes in $\U$ studied here  is associated with  a $\QP$-symmetric Dirichlet form. We will recall the precise definition and related notions of Dirichlet forms in \S\ref{subsec:D}. 
Our standing hypothesis are {\it (a) conditional absolute continuity of $\QP$}; {\it (b) conditional closability of finite-volume Dirichlet forms}. These conditions are fairly general: they cover all quasi-Gibbs measures having lower semi-continuous densities are included, e.g., Gibbs measures with Ruelle class potentials, $\sine_\beta$ $(\beta \ge 0)$, $\Airy_\beta$ $(\beta=1,2,4)$, $\Riesz_\beta$ $(\beta \ge 0)$, $\Bessel_{\alpha, 2}$ $(\alpha \ge 1)$ and $\Ginibre$ point processes.

\paragraph{\bf (a) Conditional absolutely continuity} We write $B_r=B_r(0):=\{x \in \R^n: |x|<r\}$ and  $\mssm_r$ for the Lebesgue measure restricted in $B_r$. 
For a Borel probability~$\QP$ on $(\U, \tau_\mrmv)$, we write 
\begin{itemize}
\item $\QP_r^\eta(\cdot)=\QP(\cdot \mid \cdot\mrestr{B_r^c}=\eta_{B_r^c})$ for the conditional probability of $\QP$ conditioned in the complement $B_r^c$ to be $\eta_{B_r^c}:=\eta\mrestr{B_r^c}$. Here we regarded $\QP_r^\eta$ as a probability measure on $\U(B_r)$ instead of $\U$ via the push-forward by the map~$\gamma+\eta \mapsto \gamma$ for $\gamma \in \U(B_r)$ (see Rem.~\ref{r:CP});
\vspace{1mm}
\item $\U^k(B_r)=\{\gamma \in \U(B_r): \gamma(B_r)=k\}$ for the $k$-particle configuration space in~$B_r$;
\vspace{1mm}
\item $\mssm_r^{\odot k}$ for the $k$-symmetric tensor-product of~$\mssm_r$  on $\U^k(B_r) \cong B_r^{\times k}/\mathfrak S(k)$, where $\mathfrak S(k)$ is the $k$-symmetric group;
\vspace{1mm}
\item $\QP_r^{k, \eta}=\QP_r^\eta\mrestr{\U^k(B_r)}$ for the restriction of $\QP_r^\eta$ in $\U^k(B_r)$. 
\end{itemize}

We say that $\QP$
is \emph{conditionally absolutely continuous}  if 
\begin{equation}\tag*{$(\mathsf{CAC})_{\ref{d:ConditionalAC}}$}
\label{ass:CE}
 \QP^{\eta, k}_{r} \ll \mssm_{r}^{\odot k}  \quad r\in \N\comma \ k \in \N_0  \comma \  \QP\text{-a.e.~$\eta$}\fstop
\end{equation}

\paragraph{\bf (b) Conditional closability of finite-volume Dirichlet forms} We write 
\begin{itemize}
\item $\mssd_{\U^k}$ for the restriction of $\mssd_\U$ in $\U^k$; 
\vspace{1mm}
\item ${\rm LIP}_b(\U(B_{r}), \mssd_{\U})$ for the space of bounded functions $u$ in $\U(B_r)$ such that $u|_{\U^k(B_r)}$ is $\mssd_{\U^k}$-Lipschitz for every $k \in \N_0$;
\vspace{1mm}
\item  $\nabla^{\odot k}$ for the distributional gradient in $\U^{k}$ descendent from $\R^{k}$ via the symmetric quotient; 
\vspace{1mm}
\item $\cdc^{\dUpsilon(B_{r})}$ for the square field operator on $\U(B_r)$ defined as  
$$\cdc^{\dUpsilon(B_{r})}(u) := \sum_{k=0}^\infty\Bigl|\nabla^{\odot k}\bigl(u|_{\U^k(B_r)}\bigr)\Bigr|^2  \cquad u \in{\rm LIP}_b(\U(B_{r}), \mssd_{\U}) \fstop$$
\end{itemize}

We say that~$\QP$ satisfies the \emph{conditional closability}~\ref{ass:ConditionalClos} if the bilinear form~$(\EE{\dUpsilon(B_{r})}{\QP^\eta_{r}}, \mcC_r^\eta)$ given by 
\begin{align}\tag*{$(\mathsf{CC})_{\ref{ass:ConditionalClosability}}$}\label{ass:ConditionalClos}
&\EE{\dUpsilon(B_{r})}{\QP^\eta_{r}}(u):=\frac{1}{2}\int_{\dUpsilon(B_{r})} \cdc^{\dUpsilon(B_{r})}(u) \diff\QP^\eta_{r}
\\
&u \in  \mcC_r^\eta:={\rm LIP}_b(\U(B_{r}), \mssd_{\U}) \cap \{u: \U(B_r) \to \R: \EE{\dUpsilon(B_{r})}{\QP^\eta_{r}}(u)<\infty\} \notag
\end{align}
is closable in~$L^2\ttonde{\dUpsilon(B_r),\QP^\eta_{r}}$ for every $r \in \N$ and $\QP$-a.e.~$\eta\in\dUpsilon$. 

\paragraph{Infinite-volume Dirichlet form} To construct infinite-volume limit of finite-volume Dirichlet forms,  we introduce a space of test functions based on the Lipschitz functions. For doing so, we introduce a pseudo-distance $\mssd_{\U}^{(r)}$  on $\U$ called {\it partial $\ell^2$-matching pseudo-distance subjected to~$B_r$}:
\begin{align*} 
 \mssd_{\U}^{(r)}(\gamma, \eta):= \inf_{\alpha, \beta \in \U(\partial B_r)}\mssd_{\U(\bar{B}_r)}(\gamma_{B_r}+\alpha, \eta_{B_r}+\beta) \qquad  \gamma, \eta \in \U \comma
\end{align*}
where $\mssd_{\U(\bar{B}_r)}$ is the $\ell^2$-matching extended distance in the configuration space~$\U(\bar{B}_r)$ over $\bar{B}_r=\{x \in \R^n: |x| \le r\}$. We prove that $\mssd_{\U}^{(r)}$ is $\tau_\mrmv^{\times 2}$-continuous and $\mssd_{\U}^{(r)} \nearrow \mssd_\U$ as $r \nearrow  \infty$ (Prop.~\ref{p:PDF}). 
Under \ref{ass:CE} and \ref{ass:ConditionalClos}, the following limit $(\E^{\U, \QP}, \mcC)$ is well-defined and closable in $L^2(\U, \QP)$ (Prop.~\ref{p:CES}):
\begin{align*} 
\E^{\U, \QP}(u)=\lim_{r \to \infty}\int_{\U} \E^{\U, \QP_r^\eta}(u_r^\eta) \diff \QP(\eta) \cquad u \in \mcC:=\bigcup_{r>0} \Lip_b(\U, \mssd_\U^{(r)})  \comma 
\end{align*}
where
$u_r^\eta(\gamma):=u(\gamma+\eta_{B_r^c})$ for $\gamma \in \U(B_r)$ and $\eta \in \U$, and $\Lip_b(\U, \mssd_\U^{(r)})$ is the space of bounded $\mssd_\U^{(r)}$-Lipschitz functions. %
The closure $(\E^{\U, \QP}, \dom{\E^{\U, \QP}})$ is a strongly local Dirichlet form with the domain~$\dom{\E^{\U, \QP}} \subset L^2(\U, \QP)$ and the square field  $\cdc^\U: \dom{\E^{\U, \QP}} \to \R$ exists (see Dfn~\ref{d:DFF} and Prop.~\ref{p:CES}):
\begin{align}\label{e:DDFS}
\E^{\U, \QP}(u)=\frac{1}{2}\int_{\U}\cdc^{\U}(u) \diff \QP \fstop
\end{align}
The form~\eqref{e:DDFS} associates the $L^2$-symmetric semigroup $T_t^{\U, \QP}: L^2(\U,\QP) \to L^2(\U, \QP)$ by solving
\begin{align} \label{e:SGGF}
\partial_t u = -\nabla_{L^2}\E^{\U, \QP}(u) \comma
\end{align}
where $\nabla_{L^2}$ is the Fr\'echet derivative in the Hilbert space~$L^2(\U, \QP)$. Under \ref{ass:CE} and \ref{ass:ConditionalClos}, the form~$(\E^{\U, \QP}, \dom{\E^{\U, \QP}})$ satisfies {\it $\tau_\mrmv$-quasi-regularity} (Cor.~\ref{c:QR}), thus,  there exists a $\QP$-reversible diffusion process $\mssX$ in~$\U$ whose transition semigroup is identified with~$T_t^{\U, \QP}$ outside a set of capacity-zero in~$\U$.
In this case, we say that $(\E^{\U, \QP}, \dom{\E^{\U, \QP}})$ {\it is associated with} $\mathsf X$. From the Markov process viewpoint, the operator defined as $u \mapsto -\nabla_{L^2}\E^{\U, \QP}(u)$ is the infinitesimal generator of the Markov process~$\mssX$. 
In Thm.~\ref{t:IUL}, under a mild assumption of $\QP$ on $\U=\U(\R)$ including, e.g., every determinantal/Pfaffian point process with local trace class kernel, we prove that
$$\text{$(\E^{\U, \QP}, \dom{\E^{\U, \QP}})$ is identical to {\it upper Dirichlet form} $(\bar\E^{\U, \QP}, \dom{\bar \E^{\U, \QP}})$  in~\cite[Thm.~1]{Osa96}}$$
where the upper Dirichlet form is the closure of smooth local functions (Dfn.~\ref{d:SLF}) instead of our Lipschitz core~$\mcC$.
E.g., for $\QP=\sine_2$ (resp.~$\QP=\Airy_2$),  the form~$(\E^{\U, \QP}, \dom{\E^{\U, \QP}})$ is therefore associated with the unique unlabelled solution to~\ref{DBMB} (resp.~\ref{DBME}), 
see \cite[Thm.~24]{Osa12}, \cite[\S8]{Tsa16}, \cite{OsaTan20} for $\sine_2$;  \cite[Thm.~2.3, 2.4]{OsaTan14} for $\Airy_2$. 

\paragraph{Literature about Dirichlet-form-based construction} For the construction of other interacting diffusions based on Dirichlet forms, we refer the readers to \cite{Osa12}, \cite{HonOsa15} and \cite{OsaOsa25} for the cases $\QP=\Ginibre$, $\Bessel_{2, \alpha}$ $(\alpha \ge 1)$ and the $d$-dimensional $\mathsf{Coulomb}_{\beta}$ $(\beta >0,\ d \ge 2)$ point processes respectively. For Gibbs measures admitting either Ruelle class potentials, an integration by parts property, or a determinantal structure with finite-range kernel, see \cite{Yos96, AlbKonRoe98b, Yoo05} respectively. For the configuration space over manifolds or general diffusion spaces, see~\cite{AlbKonRoe98, AlbKonRoe98b, RoeSch99, MaRoe00, ErbHue15, LzDSSuz21,  LzDSSuz22} and references therein.

\subsection{The main result in full generality} 
The main assumption is the following:
 \begin{ass}[{Assumption~\ref{a:AP}}] \label{ia:AP}
 Let $\QP$ be a fully supported Borel probability measure in~$(\U, \tau_\mrmv)$ and $K \in \R$.
Suppose that there exists a $\QP$-measurable set~$\Theta \subset \U$ with $\QP(\Theta)=1$ such that 
\begin{enumerate}[(a)]
\item \label{AP-11} $\{\gamma \in \U: \mssd_\U(\gamma, \eta)<+\infty \text{\ for some $\eta \in \Theta$}\} = \Theta$; 
\item \label{AP-2} for every $\gamma \in \Theta$ and $t>0$, there exists a Borel probability measure $p^{\U, \QP}_t(\gamma, \diff \eta)$ on $(\U, \tau_\mrmv)$ such that 
\begin{align}\label{e:IKP}
\tilde{T}_t^{\U, \QP}u(\gamma):=\int_{\U}u(\eta) p^{\U, \QP}_t(\gamma, \diff \eta)
\end{align}
is a $\QP$-representative of $T_t^{\U, \QP}u$ for every bounded Borel  function $u \in \mcB_b(\U)$;
\item \label{AP-4} there exists a subset~$\{k\} \subset \N$ of infinite cardinality and a probability measure~$\mu^k$ on $\U^k$ for each $k$ such that the metric measure space~$(\U^k, \mssd_{\U^k}, \mu^k)$ is an $\RCD(K,\infty)$ space for every $k$, and  satisfies the following property: for every $\gamma, \zeta \in \Theta$ with $\mssd_\U(\gamma, \zeta)<\infty$, there exist $\gamma^{(k)}, \zeta^{(k)} \in \U^k$ such that, as $k \to \infty$ 
\begin{align*}
&p^{\U^{k}, \QP^{k}}_t(\gamma^{(k)}, \diff \eta) \xrightarrow{\tau_\mrmw} p^{\U, \QP}_t(\gamma, \diff \eta)\cquad p^{\U^{k}, \QP^{k}}_t(\zeta^{(k)}, \diff \eta) \xrightarrow{\tau_\mrmw} p^{\U, \QP}_t(\zeta, \diff \eta)  \comma
\\
&  \limsup_{k \to \infty} \mssd_\U(\gamma^{(k)}, \zeta^{(k)}) \le \mssd_\U(\gamma, \zeta) 
\end{align*}
for every $t>0$, where $p^{\U^{k}, \QP^{k}}_t(\gamma, \diff \eta)$ is the heat kernel on $(\U^k, \mssd_{\U^k}, \mu^k)$, see~\eqref{d:HK}. 
\end{enumerate}

\end{ass}
The $\RCD(K,\infty)$ condition is a generalised notion of  uniform lower Ricci curvature bound $\Ric \ge K$  in the framework of metric measure spaces.
We will recall the definition in Dfn.~\ref{d:RCD}.

\paragraph{Main result in full generality}
Let $\mathcal T_t^{\U, \QP}$ be the dual semigroup of the~semigroup $T_t^{\U, \QP}$ given in~\eqref{e:SGGF}. Namely, $\mathcal T_t^{\U, \QP}\nu = (T_t^{\U, \QP} \rho)\cdot \QP$ for $\nu=\rho\cdot \QP \in \mcP_\QP(\U)$. Due to the $\QP$-symmetry of $T_t^{\U, \QP}$, the dual semigroup is identified with the time marginal law $\mcL(\mssX_t)$ of the  diffusion process $\mssX$  associated with the Dirichlet form~$(\E^{\U, \QP}, \dom{\E^{\U, \QP}})$ with initial condition $\nu$. 
In the following theorem, $\mathcal T_t^{\U, \QP}\nu$ is extended to $\nu \in \mcP_e(\U)$, which is not necessarily absolutely continuous to $\QP$.

 \begin{thm}[{Thm.~\ref{t:main}}]\label{c:ETD} Suppose~\ref{ass:CE}, ~\ref{ass:ConditionalClos} and Assumption~\ref{ia:AP}. 
Then, the dual semigroup $(\mcT_t^{\U, \QP})_{t \ge 0}$ is  the unique $\EVI_K$-gradient flow in $(\mcP_e(\U), \mssW_{2, \mssd_\U})$.
Namely, for every $\nu\in \dom{\Ent_\QP}$ and $\sigma \in \mcP(\U)$ with $\mssW_{2, \mssd_{\U}}(\sigma, \nu)<\infty$, 
\begin{itemize}
\item $\mathcal T^{\U, \QP}_t \sigma$ is well-defined, $\mathcal T^{\U, \QP}_t \sigma \in \dom{\Ent_{\QP}}$ and ${\mssW}_{2, \mssd_{\U}}\bigl({\mathcal T^{\U, \QP}_t \sigma}, \nu \bigr)<\infty$ for every $t>0$;
\item  $t \mapsto \mathcal T^{\U, \QP}_t \sigma$ is the unique curve satisfying the following inequality:
 \begin{align} \tag{$\EVI_K$} \label{EVI} 
\frac{1}{2}\frac{\diff^+}{\diff t}{\mssW}_{2, \mssd_{\U}}\bigl({\mathcal T^{\U, \QP}_t \sigma}, \nu \bigr)^2  + \frac{K}{2}{\mssW}_{2, \mssd_{\U}}\bigl({\mathcal T^{\U, \QP}_t \sigma}, \nu \bigr)^2 \le \Ent_{\mu}({\nu}) - \Ent_{\mu}({\mathcal T^{\U, \QP}_t \sigma})\fstop 
\end{align}
%
\end{itemize}
 \end{thm}
 Similarly to Corollaries~A--C, for general $\QP$, we obtain that $\mcT^{\U, \QP}_t\nu$ is the $\mssW_{2, \mssd_\U}$-gradient flow of $\Ent_\QP$; functional inequalities with dependency on $K$ (log--Sobolev and Talagrand inequalities follow if $K>0$ and $\E^{\U, \QP}$ is irreducible); and dynamical number rigidity/tail triviality  as a corollary.  See Cor.~\ref{c:LST}, \ref{c:GC}, ~\ref{c:GF},~\ref{c:MM} and \ref{t:NRF} for full statements.

\smallskip
\subsection{Ideas for proofs} \label{ss:IP}
The proofs split into the following steps:
\begin{enumerate}[(I)]
\item We identify $(\E^{\U, \QP}, \dom{\E^{\U, \QP}})$ with an energy functional on the extended metric measure space $(\U, \mssd_\U, \QP)$ called {\it Cheeger energy}.
\item We establish the density of $\dom{\Ent_\QP}$ in $\mcP_e(\U)$ with respect to $\mssW_{2, \mssd_\U}$ (Prop.~\ref{p:DE}), the $\mssW_{p, \mssd_\U}$-contraction property of $\mcT_t^{\U, \QP}$ (Thm.~\ref{t:WC}), Wang's log--Harnack inequality (Thm.~\ref{t:DFH}) and the $p$-Bakry--\'Emery gradient estimate (Thm.~\ref{p:PBE}).
\end{enumerate}
The $\EVI$ property is then proven by using (I) and (II) along the line of \cite{AmbGigSav15} in metric measure spaces. We remark that our space $(\U, \tau_\mrmv, \mssd_\U, \QP)$ is not really a metric measure space, rather an extended metric measure space (see the definition below), due to which  many standard techniques are not a priori available in our setting. The core difficulty stems from a mismatch between the distance~$\mssd_\U$ and the reference measure $\QP$. The topology $\tau_{\mssd_\U}$ induced by $\mssd_\U$ is `{too fine}' in the sense that it is neither separable, nor  giving the support of the reference measures. Indeed, the supports of our reference measure $\QP$ (as for $\sine_2$ and $\Airy_2$) with respect to~$\tau_{\mssd_\U}$ is typically {\it empty}, meaning that every metric ball with respect to $\mssd_\U$ is a set of measure zero. 

%
\paragraph{{\rm (I)} Extended metric measure spaces and Cheeger energies} 
{\it An extended metric measure space} in this paper is a quadruplet~$(X, \tau, \mssd, \nu)$ satisfying that 
\begin{itemize}
\item $(X, \tau)$ is a Polish topological space; 
\item $\mssd$ is a complete distance  that could take $+\infty$ (called {\it extended distance});
\item $\mssd$ is the monotone increasing limit of a net of $\tau^{\times 2}$-continuous pseudo-distances $\{\mssd_i: i \in I\}$ on $\U$; 
\item $\tau$ is generated by $\Lip_b(X, \mssd) \cap \mcC_b(X, \tau)$, 
\item $\nu$ is a finite measure on the Borel $\sigma$-algebra~$\msB(\tau)$,
\end{itemize}
where $\Lip_b(X, \mssd)$ (resp.~$\mcC_b(X, \tau)$) is the space of bounded $\mssd$-Lipschitz (resp.~$\tau$-continuous) functions.
Significant differences from usual metric measure spaces are that $\mssd$ could take $+\infty$, $\mssd$ does not necessarily metrise~$\tau$, and $\nu$ is not necessarily a Borel measure with respect to the topology~$\tau_\mssd$ metrised by~$\mssd$. These properties cause pathologies such as non-measurability of $\mssd$-Lipschitz functions and $\nu$-negligibility of $\mssd$-metric balls. Nevertheless, 
one can define the energy functional~$\Ch^{\mssd, \nu}: \dom{\Ch^{\mssd, \nu}}\subset L^2(X, \nu) \to [0,\infty)$, called {\it the Cheeger energy}, which is  the $L^2(X, \nu)$-lower semi-continuous envelope of $\frac{1}{2}\int_{X}|\mssD_{\mssd}u|^2\diff \nu$, where $|\mssD_\mssd u|(x)=\limsup_{y \to x}\frac{|u(y)-u(x)|}{\mssd(x, y)}$ is called {\it the slope} (also called {\it local Lipschitz constant}). 
Based on the Cheeger energy, one can define the Laplace operator  by means of the subdifferential of $\Ch^{\mssd, \nu}$ and define the heat semigroup. If the heat semigroup is a linear operator,  then $\Ch^{\mssd, \nu}$ forms a local Dirichlet form and we say that $X$ is {\it infinitesimally Hilbertian}. 
In this case, there exists an associated $\nu$-reversible diffusion,  which is called  {\it Brownian motion in~$(X, \tau, \mssd, \nu)$} (also called {\it distorted Brownian motion} in the context of weighted Riemannian manifolds). We refer the readers to~\cite{AmbGigSav14, AmbGigSav14b, AmbErbSav16, Gig12a, Gig18, Sav19}. 
Regarding the configuration space~$\U$, 
the quadruplet~$(\U, \tau_\mrmv, \mssd_\U, \QP)$ forms {an extended metric measure space} (see Prop.~\ref{p:CEEM}).
%
%
%
%
A natural question  would be whether the Cheeger energy $\Ch^{\mssd_\U, \QP}$ coincides with the Dirichlet form $\E^{\U, \QP}$. 
When an unlabelled solution $\mathsf X$ to~an infinite-dimensional SDE is associated with $\E^{\U, \QP}$ for some $\QP$, this question is equivalent to ask whether we can regard $\mathsf X$ as the {Brownian motion on the extended metric measure space~$(\U, \tau_\mrmv, \mssd_\U, \QP)$.
We answer this question affirmatively in a fairly general setting. 
   \begin{thm}\label{it:IDD} (Thm.~\ref{t:E=Ch}) 
   Suppose~\ref{ass:CE} and \ref{ass:ConditionalClos}. Then,  
$$(\E^{\U, \QP}, \dom{\E^{\U, \QP}})  = \bigl(\Ch^{\mssd_{\U}, \QP}, \mathscr D(\Ch^{\mssd_{\U}, \QP})\bigr)\fstop$$
 \end{thm}
As a result, the time-marginal law of the $\QP$-reversible interacting Brownian motions associated with~$\E^{\U, \QP}$ is identified with {\it the heat flow} on $(\U, \tau_\mrmv, \mssd_\U, \QP)$ under the fairly mild assumptions \ref{ass:CE} and \ref{ass:ConditionalClos} covering all quasi-Gibbs measures such as Gibbs measures with Ruelle class potentials, $\sine_\beta$ $(\beta \ge 0)$, $\Airy_\beta$ $(\beta=1,2,4)$, $\Riesz_\beta$ $(\beta \ge 0)$, $\Bessel_{\alpha, 2}$ $(\alpha \ge 1)$ and $\Ginibre$ point processes.
 \paragraph{{\rm (II)} Functional inequalities related to lower Ricci curvature bound} 
 To prove the $\EVI$ property, we show the following functional inequalities related to the lower Ricci curvature bound $``\Ric \ge K"$.
 {Let $\Theta$ be the set given in Assumption~\ref{ia:AP}. We define $\mcP_\Theta:=\{\nu \in \mcP(\U): \nu(\Theta)=1\}$. Note that $\mcP_\QP(\U) \subset \mcP_\Theta$ as well as  $\delta_\gamma \in \mcP_\Theta$ for $\gamma \in \Theta$. 
 Let $\overline{\mcP_{\Theta}}^{\mssW_{p, \mssd_\U}}$ be the closure of $\mcP_\Theta$ with respect to $\mssW_{p,\mssd_\U}$.}
\begin{thm}[{Thm.~\ref{t:WC}, \ref{p:PBE}, \ref{t:DFH}}] \label{c:FWD}Suppose~\ref{ass:CE}, ~\ref{ass:ConditionalClos} and Assumption~\ref{ia:AP}.  
\begin{enumerate}[{\rm (a)}]
\item  \label{c:FWD-2} $(${\bf  $\mssW_{p, \mssd_\U}$-contraction}$)$ Let $1 \le p <\infty$.  {Then, $\mcT_t^{\U, \QP} \nu$ is well-defined for $\nu \in \overline{\mcP_{\Theta}}^{\mssW_{p, \mssd_\U}}$. Furthermore, for $\nu, \sigma \in \overline{\mcP_{\Theta}}^{\mssW_{p, \mssd_\U}}$,}
$$\mssW_{p, \mssd_\U}(\mcT_t^{\U, \QP} \nu, \mcT_t^{\U, \QP}\sigma) \le e^{-Kt}\mssW_{p, \mssd_\U}(\nu, \sigma)  \fstop$$
\item  $(${\bf  Wang's log-Harnack inequality}$)$ for every non-negative bounded Borel function $u: \U \to \R_+$ and $t>0$, 
\begin{align} \label{i:LH1}
T^{\U, \QP}_t\log u(\gamma) \le \log (T^{\U, \QP}_tu(\eta)) + \frac{{\mssd}_\U(\gamma, \eta)^2}{4I_{2K}(t)}\comma \quad \text{$\gamma, \eta \in \Theta$} \scolon
\end{align}
where $I_K(t) :=\int_0^t e^{Kr}\diff r=\frac{e^{Kt}-1}{K}$. 
\item  $(${\bf  $p$-Bakry--\'Emery gradient estimate}$)$   \label{c:FWD-2i} The Bakry--\'Emery $p$-gradient estimate $\BE_p(K,\infty)$ holds for every $1 \le p <\infty$$:$
 $$\cdc^{\U}(T_t^{\U, \QP}u)^{\frac{p}{2}} \le  e^{-pKt}T_t^{\U, \QP}\bigl(\cdc^{\U}(u)^{\frac{p}{2}}\bigr) \quad \text{$\QP$-a.e.~} \quad u \in \dom{\E^{\U, \QP}} \fstop$$
\end{enumerate}
\end{thm}

\subsection{Comparison with literature} \
\\
\vspace{-5mm}
 \subsection*{Comparison with the stability of $\EVI$  in metric measure spaces} Recalling that  the $\RCD$ condition is equivalent to the $\EVI$ property of the dual semigroup due to~\cite{AmbGigSav14b} in  the framework of metric measure spaces,  Thm.~\ref{c:ETD} can be regarded as a sort of  stability of the $\EVI$ property under the approximation of extended metric measure spaces by metric measure spaces. In the generality of (non-extended) metric measure spaces, the stability of $\EVI$ of the heat flow (equivalently, the $\RCD$ property) was established, e.g., under the pointed measure Gromov convergence in Gigli--Mondino--Savar\'e~\cite{GigMonSav15}; under the concentration topology in Funano--Shioya~\cite{FunShi13} combined with Ozawa--Yokota~\cite{OzaYok19}, see also \cite{GigVin24}. However, none of those can be applied to our setting of extended metric measure spaces because the existing approaches essentially rely on the structure that  the support of the reference measure is given by the topology induced by the distance, which does not hold in our setting as remarked in~\S\ref{ss:IP}.   

 \vspace{0mm}
\subsection*{Comparison with \cite{Suz22b}} \ In \cite{Suz22b}, the author studied the Bakry--\'Emery gradient estimate, a weaker property than the $\RCD$ and the $\EVI$ with respect to $\mssW_{2,\mssd_\U}$,   for a Dirichlet form associated with a class of $\QP$ satisfying Dobrushin--Lanford--Ruelle (DLR) equations.
 We compare it to our result below. 
\begin{itemize}
\item The Dirichlet form used in \cite{Suz22b} is associated to what is called {\it the lower Dirichlet form} (see \cite{KawOsaTan21} for the terminology), which is obtained as a monotone increasing limit of truncated Dirichlet forms, see~Dfn.~\ref{d:DFF}.   It is in general smaller (the domain is larger) than the {Dirichlet form} $\E^{\U, \QP}$ used in this paper, which is associated to what is called {\it the upper Dirichlet form} (see~also Thm.~\ref{t:IUL}).   
 We do not know whether the lower Dirichlet form coincides with the upper Dirichlet form nor the Cheeger energy in the generality of Thm.~\ref{it:IDD}. We also do not know whether the lower Dirichlet form associates a diffusion process in the generality of Thm.~\ref{it:IDD} due to the lack of the quasi-regularity while our Dirichlet form $\E^{\U, \QP}$ always does (see Cor.~\ref{c:QR}). 

\item 
Our approach here does not need any DLR equation and is applicable to, e.g., the case $\QP=\Airy_2$ where the DLR equation remains unknown. 

\item   In~\cite[Cor.~1.7]{Suz22b}, for $\QP=\sine_\beta$, the $\EVI$ in terms of a {\it variational-type extended distance~$\mssW_\E$} was proved; however, this concept is much weaker than  the $\EVI$ property with respect to $\mssW_{2, \mssd_\U}$ in~Thm.~\ref{c:ETD}. The definition of~$\mssW_\E$  does not involve any metric space information $(\U, \mssd_\U)$, indeed it  can be defined for arbitrary strongly local Dirichlet space admitting square field, which  does not even need to be a topological space, see \cite[Dfn.~10.1, Thm.~11.1]{AmbErbSav16}. 
Consequently, \cite[Cor.~1.7]{Suz22b} per se has nothing to do with the metric structure $\mssd_\U$, hence it does not associate the gradient flow with respect to the Wasserstein distance $\mssW_{2, \mssd_\U}$. 
%

 \end{itemize}
 
 \vspace{-2mm}
 \subsection*{Other literature}   \paragraph{Diffusion case}
In the \emph{independent} Brownian particle case (i.e., \(\QP\) Poisson), the \(\EVI\) characterisation of the heat flow on configuration space is known: in \cite{ErbHue15} for  the case where the base space is a Riemannian manifold; in~\cite{LzDSSuz22} for the case where the base space is a metric measure space.
There the infinite-particle semigroup \emph{tensorises} and the \(\EVI\) property essentially reduces to the one-particle flow, for which the independence is crucial.
We also note that because our reference measure $\QP$ is typically singular to the Poisson point process (as in $\sine_2$ and $\Airy_2$ cases), there is no way that reduces to the Poisson case by computing the Hessian of the logarithmic density. 
In ~\cite{ErbHueJalMul25},  a \emph{specific} (scaled) entropy and a \emph{scaled} transport distance on the space of translation-invariant point processes were introduced and proved an \(\EVI\) property for flows induced by {independent} Brownian motions started from stationary laws. 
The Benamou--Brenier type formula has been established in~\cite{HueMue23} in terms of {\it scaled} optimal transport distance. On the spin space over a lattice, the $\EVI$ property has been proven for interacting Brownian motions with short-range (but not necessarily finite-range) interaction potentials in~\cite{HerLeb25}. 
For the $n$-particle {\it labelled} Dyson Brownian motion with general $\beta$, the phase transition phenomenon for the $\BE$ property at the threshold $\beta=1$ was studied in \cite{TasSuz25}. 

\paragraph{Jump case} In~\cite{LzDSHerSuz24}, the $\EVI$ property for the non-local Ornstein--Uhlenbeck semigroup was proven in the Poisson space, for which a non-local Benamou--Brenier type extended distance $\mcW$ was introduced in $\mcP_\QP(\U)$ with $\QP=\pi$. The corresponding dynamics is a Glauber dynamics (a birth-death process) on Polish spaces. In~\cite{HueSta25}, a specific variant of $\mcW$ was introduced in the space of the laws of stationary point processes and the $\EVI$ property of the OU semigroup was proved with respect to the specific entropy.

\subsection*{Prospects and open problems} 
 Thm.~\ref{c:ETD} provides a general guideline to show the $\EVI$ property of $\QP$-reversible interacting Brownian motions on $\U$.  We expect that many examples will be covered by Assumption~\ref{ia:AP} such as $\sine_\beta$ and $\Airy_\beta$ for general $\beta>0$ as well as many other point processes whose finite-dimensional counterpart have interaction potentials with $K$-convexity (along geodesics) in the symmetric quotient of the base space. We note that our result does not require {\it the convexity} (i.e., the $0$-convexity), only the $K$-convexity, which is weaker than the convexity when $K<0$, see~\eqref{d:GV}.

\subsection*{Outlook for other infinite-dimensional spaces} 
We  believe that the method we develop in this paper can be applied to other infinite-dimensional settings whose $\EVI$ property remains largely uncharted such as spin spaces, path and loop spaces.
Indeed, Assumption~\ref{ia:AP} does not heavily rely upon the structure of the configuration space.
We also prove in Appendix a measurable selection theorem of geodesics in general extended metric measure spaces; this is of independent interest and underpins the variational arguments used here.

 \smallskip

 \vspace{-1mm}
\subsection*{Acknowledgement}
The author expresses his great appreciation to Xiaodan Zhou and Qing Liu for hosting him in the Theoretical Sciences Visiting Program (TSVP) at the Okinawa Institute of Science and
Technology. He is also very grateful to Nicola Gigli for inviting him to SISSA and giving him an opportunity to do a lecture series on this subject. He thanks Sunil Chhita  and Mustazee Rahman for their suggestions about references regarding Airy line ensemble.  He also thanks Lorenzo Dello Schiavo for his constructive suggestions about the presentation of the paper. 

\subsection*{Data Availability Statement}
No datasets were generated during this research.

\subsection*{Notation}
We write $\N:=\{1, 2, 3, \ldots\}$, $\N_0:=\{0, 1, 2, \ldots \}$, $\EN:=\N \cup \{+\infty\}$ and $\EN_0:=\N_0 \cup\{+\infty\}$. 
We adhere to the following conventions:
\begin{itemize}
\item the superscript~${\square}^\tym{N}$  denotes $N$-fold \emph{product objects};

\item the superscript~${\square}^\otym{N}$ denotes $N$-fold \emph{tensor objects};

\item the superscript~${\square}^\osym{N}$ denotes $N$-fold \emph{symmetric tensor objects}.
\end{itemize}
Let~$(X, \tau)$ be a topological space with $\sigma$-finite Borel measure~$\nu$. 
For a subset~$A \subset X$, we write $\nu\mrestr{A}$ for the restriction of the measure~$\nu$ to $A$, and $u|_A$ for the restriction of the function~$u$ to $A$. We use
\begin{enumerate}[$(a)$]
\item  {$L^0(X, \nu)$ for the space of $\nu$-equivalence classes of functions $X \to \R\cup\{\pm\infty\}$;} for $1 \le p \le \infty$, $L^p(X, \nu):=\{u \in L^0(X, \nu): \|u\|_{L^p(\nu)}<\infty\}$, where $\|u\|^p_p=\|u\|^p_{L^p(\nu)}:=\int_X |u|^p \diff \nu$ for $1 \le p<\infty$ and $\|u\|_{\infty}=\|u\|_{L^\infty(\nu)}:=\nu\text{-}\esssup_X u$ for $p=\infty$.  In the case of~$p=2$, the inner-product is denoted by $(u, v)_{L^2(\nu)}:=\int_X uv \diff \nu$. If no confusion could occur, we simply write $L^p(\nu)=L^p(X, \nu)$;


\item $\mcC(X, \tau)$ for the space of $\T$-continuous functions on~$X$; if $X$ is locally compact, $\mcC_c(X, \tau)$ denotes the space of $\tau$-continuous and compactly supported  functions on~$X$; $\mcC_c^\infty(\R)$ for the space of compactly supported smooth functions on~$\R$. If no confusion could occur, we simply write $\mcC(X)$, $\mcC_b(X)$ and $\mcC_c(X)$ respectively. 

\item $\mathscr B(X, \tau)$ for the Borel $\sigma$-algebra with respect to~$\tau$; $\mathscr B(X, \tau)^{\nu}$ for the completion of $\mathscr B(X, \tau)$ with respect to~$\nu$; $\mathscr B(X, \tau)^*$ for the universal $\sigma$-algebra, i.e., the intersection of $\mathscr B(X, \tau)^{\rho}$ among all Borel probability measures $\rho$ on $X$. If no confusion could occur, we simply write, e.g., $\msB(X)$ or $\msB(\tau)$;
A  measurable function $u: X \to \R$ with respect to $\mathscr B(X, \tau)$, $\mathscr B(X, \tau)^{\nu}$, $\mathscr B(X, \tau)^*$ is called {\it Borel measurable, $\nu$-measurable, universally measurable} respectively and denoted by $u\in \mathcal B(X, \tau), \mathcal B(X, \tau)^{\nu}, \mathcal B(X, \tau)^*$ respectively. If no confusion could occur, we simply write, e.g., $\mcB(X)$ or $\mcB(\tau)$;

\item $F_\# \nu$ for the push-forward measure, i.e., $F_\# \nu(\cdot)=\nu(F^{-1}(\cdot))$ given a measurable space $(Y, \Sigma)$ and a measurable map $F: (X, \mathcal B(X)^{\nu}) \to (Y, \Sigma)$;
\item $\1_{A}$ for the indicator function on $A$, i.e., $\1_{A}(x)=1$ if and only if $x \in A$, and $\1_A(x)=0$ otherwise; $\delta_x$ for the Dirac measure at $x$, i.e., $\delta_x(A)=1$ if and only if $x \in A$, and $\delta_{x}(A)=0$ otherwise;
\item $\square_+$ (resp.~$\square_b$) for a subspace of  nonnegative (resp.~bounded) functions from $X$ to $\R$, where $\square$ is a placeholder.  For instance, $\mcC_{+}(X):=\{u \in \mcC(X): u \ge 0\}$;
\item $\square_{sym}$ for a subspace of  symmetric functions from $X^{\times k}$ to $\R$. Here, $u:X^{\times k} \to \R$ is said to be {\it symmetric} if $u(x_1, \ldots, x_k)=u(x_{\sigma(1)}, \ldots, x_{\sigma(k)})$ for every $\sigma \in \mathfrak S(k)$ in the $k$-symmetric group. For instance, $\mcC_{sym}(X^{\times k}):=\{u \in \mcC(X^{\times k}): u\ \text{is symmetric}\}$.
\end{enumerate}

\section{Preliminaries} \label{sec: Pre}
%
%
\subsection{Dirichlet form}\label{subsec:D}
Throughout this article, a Hilbert space always means a separable Hilbert space with inner product~$(\cdot, \cdot)_H$ taking values in $\R$. 
Let $H$ be a Hilbert space and $\dom{Q}$ be a dense linear subspace in~$H$. We call a pair $(Q, \dom{Q})$ 
\begin{itemize}
\item  {\it symmetric form} or simply {\it form} if it is  a non-negative-definite symmetric bilinear form~$Q: \dom{Q}\times \dom{Q} \to \R$, i.e., $Q(u, v)=Q(v, u)$, $Q(u+v, w)=Q(u, w)+Q(v, w)$, $Q(\alpha u, v)=\alpha Q(u, v)$ and $Q(u, u) \ge 0$ for $u, v, w \in \dom{Q}$ and $\alpha \in \R$.  We write 
$Q(u)\eqdef Q(u,u)$, $Q_\alpha(u,v)\eqdef Q(u,v)+\alpha(u, v)_H, \alpha>0.$
We define $Q(u)=+\infty$ whenever $u \in H \setminus \dom{Q}$;
\item   \emph{closed} if the space $\dom{Q}$ endowed with the norm
$\norm{\emparg}_{\dom{Q}}:= Q_1(\emparg)^{1/2}=\sqrt{Q(\emparg)+\norm{\emparg}_{L^2(\nu)}^2} $
 is a Hilbert space;
 \item {\it closable} if for every $u_n \in \dom{Q}$ such that $Q(u_n-u_m) \xrightarrow{n, m \to \infty} 0$ and  $\|u_n\|_{H} \xrightarrow{n \to \infty}0$,  it holds  $Q(u_n) \xrightarrow{n \to \infty}0$.
If $(Q, \dom{Q})$ is closable, there exists the smallest closed extension (also called {\it closure}) $(Q', \dom{Q'})$ of  $(Q, \dom{Q})$, i.e., $(Q', \dom{Q'})$ is the smallest form satisfying that $\dom{Q} \subset \dom{Q'}$, $Q'=Q$ on $\dom{Q}^{\times 2}$ and $(Q', \dom{Q'})$ is closed.
\end{itemize}

\paragraph{Dirichlet form} Let $(X, \Sigma, \nu)$ be a $\sigma$-finite measure space. A \emph{symmetric Dirichlet form} $($or simply, {\it Dirichlet form$)$ on~$L^2(\nu)$} is a densely defined closed symmetric form~$(Q,\dom{Q})$ on~$L^2(\nu)$ satisfying the Markov property (i.e., the contraction property under the truncation by constants):
\begin{align*}
u_0\eqdef 0\vee u \wedge 1\in \dom{Q} \quad \text{and} \quad Q(u_0)\leq Q(u)\comm \quad u\in\dom{Q}\fstop
\end{align*}

\paragraph{Resolvent and semigroup}
Let $(Q, \dom{Q})$ be a symmetric closed  form on a Hilbert space~$H$.  
\begin{itemize}
\item {\it The resolvent operator} $\seq{R_\alpha}_{\alpha  \ge 0}$ is the unique bounded linear operator on~$H$  satisfying 
$$Q_\alpha(R_\alpha u, v) = (u, v)_{H} \cquad  u \in H \quad v \in \dom{Q} \fstop$$
\item {\it The semigroup} $\seq{T_t}_{t  \ge 0}$ is the unique bounded linear operator on~$H$ satisfying 
$$R_\alpha u = \int_{0}^\infty e^{-\alpha t} T_{t}u \diff t \quad u \in H\fstop$$
\end{itemize}
Note that $u_t=T_tu$ can be equivalently characterised as the unique solution to 
$\partial u_t = -\nabla_HQ(u_t)$ with $u_0=u,$ where $\nabla_H$ is the Fr\'echet derivative in~$H$.
Due to the symmetry of $(Q, \dom{Q})$, the semigroup $(T_t)_{t \ge 0}$ is $H$-symmetric: $(T_t u, v)_{H} = (u, T_tv)$ for $u, v \in H$ and $t>0$.
The semigroup $\seq{T_t}_{t \ge 0}$ has the following contraction properties (see, e.g., \cite[\S1.3, p.16 \& Lem.~1.3.3]{FukOshTak11}): for every $t>0$
\begin{align}\label{e:con}
\|T_t u\|_{H} \le \|u\|_H \quad u \in H  \comma \quad Q(T_tu) \le Q(u) \quad u \in \dom{Q}\fstop
\end{align}
If $(Q, \dom{Q})$ is a Dirichlet form on $L^2(X, \nu)$, then $0 \le T_t u \le 1$ whenever $0 \le u \le 1$, see, e.g., \cite[Thm.~1.4.1]{FukOshTak11}. In this case,  the contraction \eqref{e:con} extends to the $L^p$ space for all $1 \le p \le \infty$:
\begin{align}\label{e:con1}
\|T_t u\|_{L^p} \le \|u\|_{L^p} \quad u \in L^p(\nu)  \fstop
\end{align}
See, e.g., \cite[Thm.~1.3.3]{Dav89}. Furthermore, $\nu$ is an invariant measure for $T_t$: 
\begin{align}\label{e:inv1}
 \int_X T_t u \diff \nu = \int_X u \diff \nu \quad u \in L^1(\nu) \fstop
\end{align}

\paragraph{Quasi-notion}
Let $(X, \tau)$ be a Polish topological space (i.e., a separable topological space whose topology is metrised by a complete distance), $\nu$ be a $\sigma$-finite Borel measure on $X$ and~$(Q, \dom{Q})$ be a  Dirichlet form on $L^2(\nu)$.
For~$A\in\mathscr B(X)$, define
$\dom{Q}_A\eqdef \set{u\in \dom{Q}: u= 0 \ \text{$\nu$-a.e. on}\ X\setminus A}.$
A sequence $\seq{A_n}_{n \in \N}\subset \mathscr B(X)$ is 
\begin{itemize}
\item a \emph{Borel nest} if $\cup_{n \in \N} \dom{Q}_{A_n}$ is dense in~$\dom{Q}$;
\item a \emph{closed (resp.~compact) nest} if it is a Borel nest consisting of closed (resp.~compact) sets. 
\end{itemize}
A set~$N\subset X$ is \emph{exceptional} if there exists a closed nest~$\seq{F_n}_{n  \in \N}$ such that~$N\subset X\setminus \cup_n F_n$. It is a standard fact that any exceptional set $N$ is $\nu$-negligible (see, e.g., \cite[Exe.~2.3]{MaRoe90}). 
%
A property~$(p_x)$ depending on~$x\in X$ holds \emph{quasi-everywhere} (in short: q.e.) if there exists an exceptional set~$N$ such that~$(p_x)$ holds for every~$x\in X\setminus N$. For a closed nest $\seq{F_n}_{n \in \N}$, define 
$$\mathcal C(\seq{F_n}_{n \in \N}, \tau):=\{u: A \to \R: \cup_{n \ge 1} F_n  \subset A \subset X,\ u|_{F_n} \ \text{is $\tau$-continuous for every $n \in \N$}\} \fstop$$
A function $u$ defined quasi-everywhere on $X$ is {\it quasi-continuous} if there exists a closed nest $\seq{F_n}_{n \in \N}$ such that $u \in \mathcal C(\seq{F_n}_{n \in \N}, \tau)$.  

\paragraph{Quasi-regularity} A Dirichlet form $(Q, \dom{Q})$ on $L^2(\nu)$ is {\it quasi-regular} if the following conditions hold:
\begin{enumerate}[{\rm (QR1)}]
\item \label{QR1} there exists a compact nest $(A_n)_{n \in \N}$; 
\item \label{QR2}  there exists a dense subspace $\mathcal D \subset \dom{Q}$ such that every $u \in \mathcal D$ has a quasi-continuous~$\nu$-representative~$\tilde{u}$;
\item \label{QR3}  there exists $\{u_n: n \in \N\} \subset \dom{Q}$ and an exceptional set $N \subset X$ such that every $u_n$ has a quasi-continuous~$\nu$-representative~$\tilde{u}_n$ and $\{\tilde{u}_n: n \in \N\}$ separates points in $X \setminus N$. 
\end{enumerate}
If we want to stress the topology~$\tau$, we say that  $(Q, \dom{Q})$ is {\it $\tau$-quasi-regular}. 
If $(Q, \dom{Q})$ is quasi-regular, there exists a $\nu$-reversible continuous-time strong Markov process  on~$X$ 
whose transition semigroup is a quasi-continuous $\nu$-representative of  the $L^2$-semigroup of $(Q, \dom{Q})$~(see~\cite[Thm~2.13 and Prop.~2.15 in Chapter V]{MaRoe90}). 
%
%
%
%

\paragraph{Capacity} Let $(X, \tau)$ be a Polish space, $\nu$ be a $\sigma$-finite Borel measure on $X$ and $(Q, \dom{Q})$ be a Dirichlet form on $L^2(\nu)$.
For an open set $U \subset X$, define
$\mathcal L_U:=\{u \in \dom{Q}: u \ge 1\ \text{$\nu$-a.e.~on $U$}\}$,
and 
\begin{align*}
{\rm Cap}_{{Q}}(U)=
\begin{cases}
\inf_{u \in \mathcal L_U}\bigl(Q(u)+\|u\|_{L^2}^2\bigr) \quad &\text{if $ \mathcal L_U \neq \emptyset$;}
\\
+\infty \quad &\text{if $ \mathcal L_U = \emptyset$.}
\end{cases}
\end{align*}
For any subset $A \subset X$, define 
${\rm Cap}_{{Q}}(A):=\inf\{{\rm Cap}_{{Q}}(U): A \subset U \subset X,\ \text{$U$ open}\}.$
The condition~\ref{QR1} is equivalent to the tightness of $\Cap_Q$ (see e.g.,~\cite[Rem.~3.2 on p.101]{MaRoe90}): There exists a sequence of $\tau$-compact sets $(K_n)_{n \in \N} \subset X$ such that 
\begin{align}\label{d:TC}
\Cap_Q(K_n^c) \xrightarrow{n \to \infty} 0 \fstop
\end{align}

\paragraph{Locality}  Let $(X, \Sigma, \nu)$ be a $\sigma$-finite measure space and let $(Q, \dom{Q})$ be a  Dirihclet form on $L^2(\nu)$. 
The form~$(Q, \dom{Q})$ is called {\it strongly local} if 
$Q(u, v)=0$ whenever $u, v \in \dom{Q}$ and  $u(v-c)=0$ on $X$ for some constant $c \in \R$.
Note that this definition does not require any topology in $X$.  If $(Q, \dom{Q})$ is quasi-regular and strongly local, it associates a $\nu$-reversible continuous-time strong Markov process having continuous trajectories (i.e.,  $\nu$-reversible diffusion process) and having no killing inside~$X$, see~\cite[Rmk.~2.4.4 and Thm.~4.3.4]{CheFuk11}.


\paragraph{Square field} A Dirichlet form $(Q, \dom{Q})$ {\it admits square field $\cdc$} if  there exists a dense subspace $H \subset \dom{Q} \cap L^\infty(\nu)$ such that for every~$u \in H$, there exists $v \in L^1(\nu)$ such that 
$2Q(uh, u) -Q(h, u^2) = \int_X h v \diff \nu$ for every $h \in \dom{Q} \cap L^\infty(\nu).$
In this case, $v$ is unique, and denoted by $\Gamma(u)$. The square field $\Gamma$ can be uniquely extended to an operator on $\dom{Q}  \times \dom{Q} \to L^1(\nu)$ (\cite[Thm.\ I.4.1.3]{BouHir91}).

\paragraph{Intrinsic extended distance} Let $(X, \tau)$ be a Polish space and $\nu$ be a $\sigma$-finite Borel measure~$\nu$. Let $(Q, \dom{Q})$ be a strongly local Dirichlet form in~$L^2(X, \nu)$ admitting a square field $\cdc$. {\it The intrinsic extended distance} $\mssd_{Q}$ is defined as
\begin{align}\label{d:IDS}
\mssd_Q(x,y):=\sup\{u(x)-u(y): \cdc(u) \le 1 \quad u \in \mathcal C_b(X, \tau) \cap \dom{Q}\}\fstop
\end{align}

\subsection{Extended metric measure spaces} \label{subsec:EMM} 
Let~$X$ be any non-empty set. A function $\mssd\colon X^\tym{2}\rar [0,\infty]$ is called 
\begin{itemize}

\item  an {\it extended} \emph{distance} if it is symmetric, satisfying the triangle inequality  and not vanishing outside the diagonal in~$X^{\tym{2}}$, i.e.~$\mssd(x,y)=0$ iff~$x=y$;
\item  a  {\it  distance} if it is an extended distance and finite, i.e., $\mssd(x, y)<\infty$ for every $x, y \in X$;
\item a \emph{pseudo-distance} if it is symmetric, satisfying the triangle inequality and finite.
\end{itemize}
A space endowed with an extended distance (resp.~distance) is called {\it extended metric space} (resp.~{\it metric space}). 
An extended distance~$\mssd$ is {\it complete} if every $\mssd$-Cauchy sequence is a converging sequence with respect to the topology $\tau_\mssd$ generated by $\mssd$.
 We write~$B_r(x_0)\eqdef \set{\mssd_{x_0} < r}$, where $\mssd_{x_0}:=\mssd(x_0, \cdot)$.

%

\paragraph{Lipschitz algebra} Let~$\mssd$ be a pseudo-distance or an extended distance on~$X$. A function~$f\colon X\rar \R$ is called {\it $\mssd$-Lipschitz} if there exists a constant~$L>0$ such that
\begin{align}\label{eq:Lipschitz}
\tabs{ u(x)-u(y)}\leq L\, \mssd(x,y) \comm \qquad x,y\in X \fstop
\end{align}
The smallest constant~$L$ satisfying~\eqref{eq:Lipschitz}  is called {\it the (global) Lipschitz constant of~$u$}, denoted by~$\Lip_\mssd{(u)}$.
We denote by $\Lip(X,\mssd)$  the space of all $\mssd$-Lipschitz functions.
If no confusion could occur, we simply write 
$\Lip(\mssd)=\Lip(X,\mssd)$. 
For a given measure $\nu$ on a $\sigma$-algebra $\Sigma$ in $X$ (not necessarily the Borel $\sigma$-algebra $\mathscr B(X, \tau_{\mssd})$), we set 
$\Lip(X, \mssd, \nu):=\{u \in \Lip(\mssd): \text{$u$ is $\Sigma^\nu$-measurable}\},$
where $\Sigma^\nu$ is the completion of the $\sigma$-algebra~$\Sigma$ with respect to~$\nu$.
Similarly, we simply write $\Lip(\mssd, \nu)=\Lip(X, \mssd, \nu)$ if no confusion could occur.

\paragraph{Absolutely continuous curve}
Let $(X, \mssd)$ be an extended metric space, $\tau_\mssd$ be the topology induced by $\mssd$,  and $J \subset \R$ be an open interval. A continuous map $\rho : J \to (X, \tau_\mssd)$ is {\it $p$-absolutely continuous} and denoted by $\rho=(\rho_t)_{t \in J} \in {\rm AC}^p(J, (X, \mssd))$ if there exists $g \in L^p(J, \diff x)$ such that 
\begin{align}\label{d:ABC}
\mssd(\rho_s, \rho_t) \le \int_s^t g(r) \diff r \cquad s, t \in J \quad s<t \fstop
\end{align}
If $p=1$, we simply say that $\rho$ is {\it absolutely continuous} and denoted by $\rho \in AC(J, (X, \mssd))$. The minimal $g$ among those satisfying~\eqref{d:ABC} exists and this is identical to 
\begin{align}\label{d:MD}
|\dot\rho_t| :=\lim_{s \to t}\frac{\mssd(\rho_s, \rho_t)}{|s-t|} \cquad \text{the limit exists in a.e.~$t \in J$}\comma
\end{align}
which is called {\it metric speed, or metric derivative of $\rho$}. Namely, $|\dot\rho_t|$ satisfies \eqref{d:ABC} and $|\dot\rho_t| \le g(t)$ for a.e.~$t \in J$ for every $g$ satisfying~\eqref{d:ABC}, see \cite[Thm.~1.1.2]{AmbGigSav08}. We say that an absolutely continuous curve $\rho$ is of {\it constant speed} if $|\dot\rho_t|$ is a constant for a.e. $t \in J$.


\paragraph{Geodesic convexity} We say that an extended metric space~$(X, \mssd)$ is {\it geodesic} if, for every pair of points~$x_0, x_1 \in X$ with $\mssd(x_0, x_1)<\infty$, there exists a constant speed geodesic $(x_t)_{t \in [0,1]}$ connecting $x_0$ and $x_1$, i.e., 
$\mssd(x_t, x_s) = |t-s|\mssd(x_0, x_1)$ for $t, s \in [0,1].$
By a simple application of the triangle inequality,  it is equivalent to 
\begin{align}\label{e:GEI}
\mssd(x_t, x_s) \le  |t-s|\mssd(x_0, x_1) \cquad t, s \in [0,1] \fstop
\end{align}
 We write ${\rm Geo}(X, \mssd)$ for the space of constant speed geodesics. Let $(X, \mssd)$ be a {\it geodesic} extended metric space. We say that $U: X \to \R\cup\{+\infty\}$ is {\it $K$-geodesically convex} for $K \in \R$ if for any $x_0, x_1 \in \dom{U}$ with $\mssd(x_0, x_1)<\infty$, there exists a constant speed geodesic $(x_t)_{t \in [0,1]}$ connecting $x_0$ and $x_1$ such that  
\begin{align} \label{d:GV}
U(x_t)\le (1-t)U(x_0)+tU(x_1)-\frac{K}{2}t(1-t)\mssd^2(x_0, x_1) \quad t \in [0,1] \fstop
\end{align}
When $K=0$, we say that $U$ is {\it geodesically convex}. When $X=\R^n$ and $\mssd$ is the standard Euclidean distance, the geodesical convexity coincides with the usual convexity. 

\paragraph{Slopes}
Let~$(X, \mssd)$ be an extended metric space and $u: X \to \R\cup\{\pm \infty\}$ be a function. For $x \in \dom{u}=\{x \in X: u(x) \in \R\}$,  {\it the slope (or local Lipschitz constant) of $u$ at $x$} is defined as 
\begin{align} \label{d:SLP}
|\mathsf D_{\mssd}u|(x):=
\begin{cases} \displaystyle
\limsup_{y \to x}\frac{|u(y)-u(x)|}{\mssd(y, x)} &\text{if $x$ is not isolated}; 
\\
0 &\text{otherwise} \fstop
\end{cases}
\end{align}
 {\it The descending slope} is defined as 
\begin{align} \label{d:DLP}
|\mathsf D_{\mssd}^-u|(x):=
\begin{cases} \displaystyle
\limsup_{y \to x}\frac{(u(y)-u(x))^-}{\mssd(y, x)} &\text{if $x$ is not isolated}; 
\\
0 &\text{otherwise} \fstop
\end{cases}
\end{align}
It is straightforward to see 
\begin{align} \label{e:SLC}
|\mathsf D_{\mssd}u| \le \Lip_{\mssd}(u) \cquad u \in \Lip(X, \mssd) \fstop
\end{align}

\begin{defs}[{Extended metric measure space cf.~\cite[Dfn.~4.7]{AmbErbSav16}}] \label{d:EMM} We say that $(X, \tau, \mssd, \nu)$ is an {\it extended metric measure space} if 
\begin{enumerate}[(a)]
\item \label{d:EMM1} $(X, \tau)$ is a Polish space;
\item  \label{d:EMM2} $\mssd$ is a complete extended distance, and there exists a family of $\tau \times \tau$-continuous bounded pseudo-distances $\mssd_i: X^{\times 2} \to [0,\infty)$ with $i \in I$ such that 
$\mssd = \sup_{i \in I}\mssd_i \ ; $
\item  \label{d:EMM3} the topology~$\tau$ is generated by $\Lip_b(X, \mssd) \cap \mcC(X, \tau)$;
\item  \label{d:EMM4} $\nu$ is a finite measure on $(X, \msB(\tau))$.
\end{enumerate}
Note that the cardinality of the index set $I$ may be uncountable. 
We say that $(X, \tau, \mssd, \nu)$ is a {\it metric measure space} if $\tau$ is generated by $\mssd$ and $\mssd$ is a distance. In this case, we simply write $(X, \mssd, \nu)$. We say that $(X, \tau, \mssd)$ is an {\it extended metric topological space} if \ref{d:EMM1}--\ref{d:EMM3} hold. 
\end{defs}
\begin{rem}\ \label{r:REW}
\begin{enumerate}[(a)]
\item \label{r:REW1} Our definition is a particular case of \cite[Dfn.~4.7]{AmbErbSav16}.
\item \label{r:REW2}The topology $\tau_\mssd$ generated by $\mssd$ is not separable in general. 
\item \label{r:REW3}The topology $\tau_\mssd$ is stronger than $\tau$, see \cite[(4.4)]{AmbErbSav16}.
\item \label{r:REW4}$\mssd$ is $\tau \times \tau$-lower semi-continuous due to \ref{d:EMM2}. For every $u\in \mcB(X, \tau)$, the slope~$|\mathsf D_\mssd u|$ is 
 universally measurable, see \cite[Lem.~2.6]{AmbGigSav14}. 
\end{enumerate}
\end{rem}

{


{
 \subsection{$L^p$-transportation extended distance}
Let $(X, \tau, \mssd)$ be an extended metric topological space in the sense of Definition~\ref{d:EMM}, and suppose that $(X, \mssd)$ is geodesic.  Take any distance $\mssd_{\tau}$ on~$X$ metrising $\tau$ and define $\tau_\mrmu$ to be the uniform topology on~the space~$C\bigl([0,1]; (X, \tau)\bigr)$ of $\tau$-continuous paths $[0,1] \to X$. As $\tau_\mssd$ is stronger than $\tau$, $\mathsf{Geo}(X, \mssd) \subset C\bigl([0,1];(X, \tau)\bigr)$ and we induce the subset topology $\tau_\mrmu$ to $\mathsf{Geo}(X, \mssd)$. Due to the $\tau^{\times 2}$-lower semi-continuity of $\mssd$, the following set is  a $\tau^{\times 2}$-Borel set 
$$D:=\{(x,y) \in X^{\times 2}: \mssd(x,y)<\infty\} \fstop$$
We endow $D$ with the subset topology $\tau^{\times 2}$ from~$X^{\times 2}$. 
\begin{lem}[Measurable selection of geodesics, the proof appears in Appendix] \label{l:MS}
Assume that $(X, \tau, \mssd)$ is an extended metric topological space in the sense of Definition~\ref{d:EMM} with the index set~$I=\mathbb N$. Suppose that $(X, \mssd)$ is geodesic.  
Then, for any $\nu \in \mathcal P(D)$, there exists a $\mathscr B(\tau^{\times 2})^\nu/\mathscr B(\tau_\mrmu)$-measurable map $\mathsf{GeoSel}: D \to \mathsf{Geo}(X, \mssd)$ such that for every~$(x, y) \in D$, $\mathsf{GeoSel}(x, y)$ is a constant speed geodesic connecting~$x$ and~$y$. 
\end{lem}

Let $(X, \tau, \mssd)$ be an extended metric topological space in the sense of Definition~\ref{d:EMM}.  
We endow~$\mcP(X)=\mcP(X, \tau)$ with the weak topology $\tau_\mrmw$, i.e., the topology induced by the duality of $C_b(X, \tau)$.
For $\nu, \sigma \in \mathcal P(X)$, we define {\it the $L^p$-transportation extended distance}~$\mssW_{p, \mssd}$ for $1 \le p<\infty$ as 
\begin{align*} 
\mssW_{p, \mssd}^p(\nu, \sigma):=\inf_{{\sf c} \in {\sf Cpl}(\nu, \sigma)} \int_{X^{\times 2}} \mssd^p(\gamma, \eta) \diff {\sf c}(\gamma, \eta) \comma
\end{align*}
where ${\sf Cpl}(\nu, \sigma)$ is the space of all Borel probability measures on~$(X^{\times 2},\tau^{\times 2})$ satisfying ${\sf c}(\Xi \times X)=\nu(\Xi)$ and ${\sf c}(X \times \Lambda)=\sigma(\Lambda)$ for every $\Xi, \Lambda \in \mathscr B(X, \tau)$. Due to the $\tau^{\times 2}$-lower semi-continuity of $\mssd^p$, for $\nu, \sigma \in \mcP(X)$ with $\mssW_{p, \mssd}(\nu, \sigma)<\infty$, there exists an optimal coupling $\mssc \in {\sf Cpl}(\nu, \sigma)$, i.e., (\cite[Thm.~4.1]{Vil09})
$$\mssW_{p, \mssd}(\nu, \sigma)^p = \int_{X^{\times 2}} \mssd(\gamma, \eta)^p \diff \mssc(\gamma, \eta) \fstop$$
We denote by ${\rm Opt}(\nu, \sigma)$ the set of optimal couplings between $\nu$ and $\sigma$. 
\begin{prop}[the proof appears in Appendix] \label{p:EMW}
Let $1\le p<\infty$. If $(X, \tau, \mssd)$ is an extended geodesic metric-topological space in the sense of Definition~\ref{d:EMM} with the index set~$I=\mathbb N$, then  so is $(\mathcal P(X), \tau_\mrmw, \mssW_{p, \mssd})$. 
\end{prop}
}

\subsection{Cheeger energy and $\RCD$ condition} \label{ss:RCD}
Let $(X, \tau, \mssd, \nu)$ be an extended metric measure space. We say that $v \in L^2(\nu)$ is a {\it relaxed slope (or relaxed gradient) of $u \in L^2(\nu)$} if there exist $\mssd$-Lipschitz functions $u_n \in L^2(\nu) \cap \mcB(\tau)$ such that 
\begin{itemize}
\item $u_n \to u$ in $L^2(\nu)$ and $|\mssD_{\mssd} u_n|$ weakly converges to some $\tilde{v} \in L^2(\nu)$ as $n \to \infty$;
\item  $\tilde{v} \le v$ $\nu$-a.e..
\end{itemize}
We say that $v$ is the {\it minimal relaxed slope (or minimal relaxed gradient) of $u$} if its $L^2(\nu)$-norm is minimal among relaxed slopes. We denote it by $|\nabla_{\mssd, \nu} u|_* \in L^2(\nu)$.
\begin{defs}[{\cite[Dfn.~4.2,  Thm.~4.5]{AmbGigSav14}}] 
The {\it Cheeger energy}~$\Ch^{\mssd, \nu}: L^2(\nu) \to \R_+ \cup\{+\infty\}$
 is defined as 
  \begin{align} \label{d:CH1}
 \Ch^{\mssd, \nu}(u)&:=\frac{1}{2}\int_X |\nabla_{\mssd, \nu} u|_*^2 \diff\nu \fstop
 \end{align}
 We set $ \Ch^{\mssd, \nu}(u)=+\infty$ if there is no relaxed slope of $u$. 
 The domain is denoted by $\dom{\Ch^{\mssd, \nu}}:=\{u \in L^2(\nu): \Ch^{\mssd, \nu}(u)<+\infty\}$. It is convex and $L^2(\nu)$-lower semi-continuous. 
 \end{defs}
 By~\cite[(4.9)]{AmbGigSav14} and \eqref{e:SLC},  we have 
\begin{align}\label{i:MSL}
|\nabla_{\mssd, \nu} u|_* \le |\mathsf D_\mssd u| \le \Lip_{\mssd}(u) \quad \text{$\nu$-a.e.}\cquad u\in \Lip_b(\mssd) \cap \mcB(\tau)\fstop
\end{align}

 \begin{rem}[Variants of definitions] \label{r:CH}
 There are different definitions of the Cheeger energy: one is defined in terms of the minimal weak upper gradient (see \cite[\S5]{AmbGigSav14}),  the other is defined in terms of the relaxation of the asymptotic slopes (see \cite{AmbErbSav16, Sav19}). All these definitions are equivalent in the setting of Dfn.~\ref{d:EMM} when $\nu$ is a finite measure due to~\cite[Thm.~6.2]{AmbGigSav14} and \cite[Cor.~10.39, Thm.~11.7]{Sav19}. 
 \end{rem}

\vspace{-3mm}
\paragraph{RCD condition} 
Let $(X, \tau, \mssd, \nu)$ be an extended geodesic metric measure space and suppose that the $\tau$-support of $\nu$ is the whole space~$X$. Define the relative entropy $\Ent_\nu(\sigma):=\int_{X} \frac{\diff \sigma}{\diff \nu} \log  \frac{\diff \sigma}{\diff \nu} \diff \nu$ for  $\sigma \ll \nu$, otherwise $\Ent_\nu(\sigma):=+\infty$.
 The domain is defined as~$\dom{\Ent_\nu}:=\{\sigma \in \mcP(X): \Ent_\nu(\sigma)<\infty\}$.  The following definition is an adaptation to extended metric measure spaces, see~\cite[Dfn.~9.1]{AmbGigSav14}.
 \begin{defs} \label{d:RCD}We say that $(X, \tau, \mssd, \nu)$ satisfies $\RCD(K,\infty)$ if 
\begin{itemize}
\item For every $\sigma_0, \sigma_1 \in \dom{\Ent_\nu}$ with $\mssW_{2, \mssd}(\sigma_0, \sigma_1)<\infty$, there exists a constant speed ${\mssW}_{2, \mssd}$-geodesic $(\sigma_t)_{t \in [0,1]}\subset \dom{\Ent_\nu}$ along which  the entropy~$\Ent_\nu$ satisfies the convexity inequality: 
\begin{align*}
\Ent_\nu(\sigma_t) \le (1-t)\Ent_{\nu}(\sigma_0) + t \Ent_{\nu}(\sigma_1) -\frac{K}{2}t(1-t)\mssW_{2, \mssd}(\sigma_0, \sigma_1)^2 \scolon
\end{align*}

\item $\Ch^{\mssd, \nu}$ is quadratic (called {\it infinitesimally Hilbertian}): 
$\Ch^{\mssd, \nu}(u+v)+\Ch^{\mssd, \nu}(u-v)= 2\Ch^{\mssd, \nu}(u)+2\Ch^{\mssd, \nu}(v)$
for every~$u, v \in \dom{\Ch^{\mssd, \nu}}$. 
\end{itemize}
\end{defs}
 Due to~\cite[Cor.\ 4.18]{AmbGigSav15}, if $(X, \mssd, \nu)$ is a geodesic (non-extended) metric measure space, it is an $\RCD(K,\infty)$ space  if and only if the following three conditions hold:
\begin{enumerate}[{\rm (i)}]
\item $\Ch^{\mssd, \nu}$ is quadratic;
\item $(\Ch^{\mssd, \nu}, \dom{\Ch^{\mssd, \nu}})$ satisfies the Sobolev-to-Lipschitz property, viz., 
$$u \in \dom{\Ch^{\mssd, \nu}} \quad |\nabla_{\mssd, \nu} u|_* \le 1 \implies \exists \tilde{u} = u \ \text{$\nu$-a.e.} \quad \Lip_\mssd(\tilde{u}) \le 1 \semicolon $$
\item $\Ch^{\mssd, \nu}$ satisfies $\BE_2(K,\infty)$, i.e., 
$|\nabla_{\mssd, \nu} T_t u|^2_* \le e^{-2Kt} T_t|\nabla_{\mssd, \nu}  u|^2_*$ for $u \in  \dom{\Ch^{\mssd, \nu}}$ and  $t \ge 0$,
where $(T_t)_{t \ge 0}$ is the $L^2(\nu)$-symmetric semigroup associated with $\Ch^{\mssd, \nu}$. 
\end{enumerate}
Due to~\cite[\S6]{AmbGigSav14} and ~\cite[(7.1)]{AmbGigMonRaj12}, there exists a heat kernel $p_t(x, \diff y)=p_t(x,y)\diff \nu(y)$, i.e., 
\begin{align} \label{d:HK}
T_t u(x)=\int_{X}u(y)p_t(x, \diff y) \cquad u \in L^2(\nu) \quad t>0\comma
\end{align}
and $\seq{T_t}_{t>0}$ is conservative (also called mass-preserving) due to~\cite[Thm.~4.20]{AmbGigSav14}:
\begin{align} \label{d:MP}
T_t \1= \1 \qquad t \ge 0 \fstop
\end{align}

\subsection{Configuration space and the $\ell^2$-matching extended distance}
Let $(X, \tau)$ be a locally compact Polish space. 
\begin{itemize}
\item A \emph{configuration} on $X$ is an $\overline\N_0$-valued Radon measure~$\gamma$ on~$X$, which is expressed by 
$\gamma = \sum_{i=1}^N \delta_{x_i},  N \in \overline{\N}_0$,
 where $x_i \in X$ for every $i$ and  $\gamma \equiv 0$  when $N=0$.
%
\vspace{1mm}
\item The \emph{configuration space}~$\U=\dUpsilon(X)$ is the space of all configurations on~$X$. 
\end{itemize}
Regarding $\dUpsilon$ as a subspace of the space $\mcM(X)$ of Radon measures on~$X$, $\U$ is endowed with the vague topology~$\tau_\mrmv$, i.e., 
$\gamma_n \xrightarrow{\tau_\mrmv} \gamma \iff \int_{X} f \diff \gamma_n \to  \int_{X} f \diff \gamma$ for $f \in \Cc(X).$
 We write the {\it restriction}~$\gamma_A\eqdef \gamma\mrestr{A}$ for~$A \in \mathscr B(X)$ and the restriction map is denoted by 
\begin{align}\label{eq:ProjUpsilon}
\gamma\longmapsto \pr_A(\gamma):=\gamma_{A}\fstop
\end{align}
The $N$-particle configuration space is denoted by
$\dUpsilon^N(X)\ \eqdef\ \set{\gamma\in \dUpsilon: \gamma(X)=N}$ for $N\in\overline\N_0$. 
Let $\mathfrak S_k$ be the $k$-symmetric group. It can be readily seen that the $k$-particle configuration space~$\U^k$ is isomorphic to the quotient space~$X^{\times k}/\mathfrak S_k$:
$\dUpsilon^k(X)\cong X^{\odot k}:=X^{\times k}/\mathfrak S_k$ for $k \in \N\fstop$
The associated projection map from $X^{\times k}$ to the quotient space~$X^{\times k}/\mathfrak S_k$ is denoted by~$\quot_k$. 
For $\eta \in \U$ and~$E\in \mathscr B(X)$, we define
$\U_E^\eta:=\{\gamma \in \U: \gamma_{E^c}=\eta_{E^c}\}.$

\paragraph{$\ell^2$-matching extended distance} Let $(X, \mssd)$ be a locally compact complete separable metric space. For~$i=1,2$ let~$\proj_i\colon X^{\times 2}\rar X$ denote the projection to the~$i^\text{th}$ coordinate for $i=1,2$. 
For~$\gamma,\eta\in \dUpsilon$, let~$\Cpl(\gamma,\eta)$ be the set of couplings of~$\gamma$ and~$\eta$, i.e., 
$\Cpl(\gamma,\eta)\eqdef \set{\cpl\in \U(X^{\tym{2}}) \colon (\proj_1)_\pfwd \cpl =\gamma \comma (\proj_2)_\pfwd \cpl=\eta}.$
The {\it $\ell^2$-matching extended distance} on~$\dUpsilon(X)$ is defined as
\begin{align*}
\mssd_{\dUpsilon}(\gamma,\eta)\eqdef \inf_{\cpl\in\Cpl(\gamma,\eta)} \paren{\int_{X^{\times 2}} \mssd^2(x,y) \diff\cpl(x,y)}^{1/2}\comma \qquad \inf{\emp}=+\infty \fstop
\end{align*}
We denote by $\mssd_{\U^k}$ the restriction of $\mssd_\U$  in $\U^k \subset \U$. 
\subsection{Conditional absolute continuity} 
\begin{defs}[Conditional Probability {\cite[452E, 452O, 452G(c)]{Fre00}}] \label{d:CP}
Let $\mu$ be a Borel probability measure on $\U=\U(\R^n)$ and $\QP_{B_r^c}:={\pr_{B_r^c}}_\#\QP$, where $\pr_{B_r^c}$ is the projection defined in~\eqref{eq:ProjUpsilon}.
 There exists a family of Borel probability measures~$\{\QP_{r}^\eta: r>0, \eta \in \U(B_r^c)\}$ on $\U$ such that 
\begin{enumerate}[(a)]
\item (disintegration) for  every $\Xi \in \mathscr B(\U)^\QP$, it holds
$\QP(\Xi)= \int_{\U(B_r^c)} \QP_{r}^\eta(\Xi) \diff \QP_{B_r^c}(\eta)\ ;$
\item (strong consistency) for every  $\Xi \in \mathscr B(\U)^\QP$ and $\Lambda \in  \mathscr B(\U(B_r^c))^{\QP_{B_r^c}}$, 
$$\QP(\Xi \cap \Lambda)= \int_{\Lambda} \QP_{r}^\eta(\Xi) \diff \QP_{B_r^c}(\eta) \quad \text{and} \quad \QP_{r}^\eta(\U_r^\eta)=1 \quad \QP_{B_r^c}\text{-a.e.~$\eta$}\fstop$$ 
\end{enumerate}
We call $\{\QP_{r}^\eta: r>0, \eta \in \U(B_r^c)\}$ {\it strongly consistent regular conditional probability measures}. 
\end{defs}
\begin{rem}\label{r:CP}
We may think of $\QP_r^\eta$ as a Borel probability measure on $\U(B_r)$ instead of~$\U$. Indeed, thanks to the strong consistency, the projection~$\pr_{B_r}:\U_r^\eta \to \U(B_r)$ with its inverse $\pr_{B_r}^{-1}: \U(B_r) \to \U_r^\eta$ defined as $\gamma \mapsto \gamma+\eta$ gives a bi-measure preserving bijection map between the two measure spaces 
\begin{align*} 
(\U_r^\eta, \QP_r^\eta) \cong (\U(B_r), {\pr_{B_r}}_\# \QP_r^\eta) \fstop
\end{align*}
Hence, we may identify  $\QP_r^\eta$ with ${\pr_{B_r}}_\# \QP_r^{\eta_{B_r^c}}$ and regard~the regular conditional probability measures as a family $\{\QP_r^\eta: r>0, \eta \in \U\}$ of  probability measures on $\U(B_r)$ indexed by $\eta \in \U$ and $r>0$. We follow this convention unless confusions could occur. 
\end{rem}

For a $\QP$-measurable function $ u\colon \dUpsilon\to \R$ and $\eta \in \dUpsilon$, we define 
 \begin{align} \label{e:SEF}
u_{r}^\eta(\gamma)\eqdef  u(\gamma+\eta_{B_r^\complement})  \qquad \gamma\in \dUpsilon(B_r) \fstop
 \end{align}
By definition of conditional probabilities,  it is straightforward to see
$\int_{\dUpsilon} u \diff\QP = \int_{\dUpsilon} \bigl(\int_{\dUpsilon(B_r)} u_{r}^\eta \diff \QP^\eta_r\bigr) \diff\QP(\eta)$ for every $u \in L^1(\mu)$.
\begin{defs}[Conditional absolute continuity]\label{d:ConditionalAC}
Let $\QP$ be a Borel probability on~$\U=\U(\R^n)$ and $\mssm_r$ be the Lebesgue measure restricted on $B_r=B_r(0)=\{x \in \R^n: |x|<r\}$.
We say that $\QP$
is \emph{conditionally absolutely continuous}  if 
\begin{equation}\tag*{$(\mathsf{CAC})_{\ref{d:ConditionalAC}}$}
\label{ass:CE}
 \QP^{\eta, k}_{r} \ll \mssm_{r}^{\odot k}  \quad r\in \N\comma \ k \in \N_0  \comma \  \QP\text{-a.e.~$\eta$}\comma
\end{equation}
where $\QP^{\eta, k}_{r} := \QP^{\eta}_{r}\mrestr{\U^k(B_r)}$ is the restriction of~$\QP^{\eta}_{r}$ in the $k$-particle configuration space~$\U^k(B_r)$.
\end{defs}
\subsection{Conditional closability}
Let $W^{1,2}_{sym}(\mssm^{\otimes k}_r)$ be the space of $\mssm^{\otimes k}_r$-classes of {\it symmetric}~$(1,2)$-Sobolev functions on the product space $B_r^{\times k}$, i.e., 
$W^{1,2}_{sym}(\mssm^{\otimes k}_r):=\{u \in L^2_{sym}(\mssm^{\otimes k}_{r}): \int_{B_r^{\times k}} |\nabla^{\otimes k} u|^2 \diff \mssm^{\otimes k}_{r} <\infty \}$,
where $\nabla^{\otimes k}$ denotes the distributional derivative on $(\R^n)^{\times k}$: $\nabla^{\otimes k}u:=(\partial_1 u, \ldots, \partial_ku)$.  Due to the Rademacher theorem in the Euclidean space, the distributional derivative coincides with the slope (see \eqref{d:SLP}) in~$\R^k$ for Lipschitz functions:
\begin{align}\label{e:RDT}
|\nabla^{\otimes k}u|(x) = |{\sf D}^{\mssd^{\times k}}u|(x)  \comma \quad \text{$u \in \Lip_b(B_r^{\times k}, \mssd^{\times k})\comma$ $\mssm_r^{\otimes k}$-a.e.} \comma
\end{align}
where $\mssd(x, y):=|x-y|$ is  the Euclidean distance in $\R^n$.
As the space $W^{1,2}_{sym}(\mssm^{\otimes k}_r)$ consists of symmetric functions, the projection $\quot_k: B_r^{\times k} \to \U^k(B_r) \cong B_r^{\times k} /\mathfrak S_k$ acts on~$W^{1,2}_{sym}(\mssm^{\otimes k}_r)$ and the resulting quotient space is denoted by $W^{1,2}(\mssm_r^{\odot k})$. Namely, 
$W^{1,2}(\mssm^{\odot k}_r):=\{u \in L^2(\mssm^{\odot k}_{r}): \int_{\U^k(B_r)} |\nabla^{\odot k} u|^2 \diff \mssm^{\odot k}_{r} <\infty\}$,
where $\nabla^{\odot k}$ is the quotient operator of the distributional gradient operator $\nabla^{\otimes k}$ through the projection $\quot_k$ and $\mssm_r^{\odot k}$ is the symmetric product measure defined as 
$\mssm_r^{\odot k}:=\frac{1}{k!} (\quot_k)_\#\mssm_r^{\otimes k}$.
The equality~\eqref{e:RDT} descends to the quotient space:
\begin{align}\label{e:RDT1}
|\nabla^{\odot k}u|(x) = |{\sf D}^{\mssd_{\U}}u|(x) \qquad \text{ $\mssm_r^{\odot k}$-a.e., \ \ $u \in \Lip_b(\U^k(B_r), \mssd_{\U})$}\fstop
\end{align}
When $k=0$, $\U^0(B_r)$ is a one-point set consisting of $\gamma\equiv 0$ and $\mssm_r^{\odot 0}=\delta_0$ is the Dirac measure on~$\gamma \equiv 0$. We set $\nabla^{\odot 0}u \equiv 0$, so $W^{1,2}(\U^0(B_r), \mssm_r^{\odot 0}) = L^{2}(\U^0(B_r), \mssm_r^{\odot 0}) \cong~\R$.
We define 
\begin{align} \label{d:SFD}
\cdc^{\dUpsilon(B_r)}(u)& := \sum_{k=0}^\infty \Bigl| \nabla^{\odot k} u|_{\U^{k}(B_r)}\Bigr|^2\comma \quad u \in {\rm LIP}_b(\U(B_r), \mssd_\U) \comma
\end{align}
where 
${\rm LIP}_b(\U(B_r), \mssd_\U):=\{u: \U(B_r) \to \R \ \text{bounded}: u|_{\U^k(B_r)} \in \Lip_b(\U^k(B_r), \mssd_\U)\comma k \in \N\}$.
The integral of \eqref{d:SFD} against $\QP_r^\eta$ is denoted by
\begin{align*}
&\EE{\dUpsilon(B_{r})}{\QP^\eta_{r}}(u)\eqdef \frac{1}{2}\int_{\U(B_r)} \cdc^{\dUpsilon(B_r)}(u) \diff \QP_r^\eta \comma
\\
&u \in \mcC_r^\eta:={\rm LIP}_b(\U(B_{r}), \mssd_{\U}) \cap \{u: \U(B_r) \to \R: \EE{\dUpsilon(B_{r})}{\QP^\eta_{r}}(u)<\infty\} \fstop \notag
\end{align*}

\begin{defs}[Conditional closability]\label{ass:ConditionalClosability}
Let~$\QP$ be a Borel probability measure on~$\dUpsilon$ satisfying~\ref{ass:CE}. 
%
We say that~$\QP$ satisfies the \emph{conditional closability}~\ref{ass:ConditionalClos} if 
\begin{align}\tag*{$(\mathsf{CC})_{\ref{ass:ConditionalClosability}}$}\label{ass:ConditionalClos}
\EE{\dUpsilon(B_{r})}{\QP^\eta_{r}}(u)=\frac{1}{2}\int_{\dUpsilon(B_{r})} \cdc^{\dUpsilon(B_{r})}(u) \diff\QP^\eta_{r}\cquad u \in  \mcC_r^\eta \notag
\end{align}
is closable in~$L^2\ttonde{\dUpsilon(B_r),\QP^\eta_{r}}$ for every $r \in \N$ and $\QP$-a.e.~$\eta\in\dUpsilon$. 
The closure is denoted by
$\ttonde{\EE{\dUpsilon(B_r)}{\QP^\eta_{r}},\dom{\EE{\dUpsilon(B_r)}{\QP^\eta_{r}}}}.$
\end{defs}

\subsection{Dirichlet form in $\U$} 
\begin{defs}[Core]\label{d:core}For $r>0$, $\mathcal C_r$ is defined as the space of functions $u$ such that
\begin{enumerate}[$(a)$]
\item\label{i:d:core1} $u\in L^\infty(\U, \mu)$; 
\item\label{i:d:core2} $u_r^\eta \in {\rm LIP}_b(\U(B_r), \mssd_\U)$ for $\QP$-a.e.~$\eta$;
\item\label{i:d:core3}  The following integral is finite:
\begin{align} \label{eq:VariousFormsA} 
\E^{\U, \QP}_r(u):=\int_{\U} {\E^{\U(B_r), \QP_r^\eta}(u_r^\eta)} \diff\QP(\eta) <+\infty \fstop
\end{align} 
\end{enumerate}
\end{defs}

Recall that, for $u: \U \to \R$ and $\eta \in \U$, the function $u_r^\eta: \U(B_r) \to \R$ was defined as $u_r^\eta(\gamma)=u(\gamma+\eta_{B_r})$  in~\eqref{e:SEF}. 
The proof of the following proposition works verbatim as \cite[Prop.~4.7,~4.13]{Suz22b}, see~also \cite[Prop.~3.9]{Suz23b}.
\begin{prop} \label{p:CDLC}
Suppose~\ref{ass:CE} and~\ref{ass:ConditionalClos}. 
Then, $(\E_r^{\U, \QP}, \mcC_r)$ in~\eqref{eq:VariousFormsA} is  closable  in $L^2(\QP)$ and 
the closure $(\E_r^{\U, \QP}, \dom{\E^{\U, \QP}_r})$ has the following integral expression: 
\begin{align}\label{e:TG}
\E_r^{\U, \QP}(u)=\int_{\U}\cdc^\U_r(u) \diff\QP \cquad
\cdc^\U_r(u)(\gamma)=\cdc^{\U(B_r)}(u_r^\gamma)(\gamma) \quad \text{$\QP$-a.e.~$\gamma$} \cquad u \in \dom{\E^{\U, \QP}_r} \fstop 
\end{align}
Furthermore, $r \mapsto \bigl(\vE^{\U, \QP}_r, \dom{\E^{\U, \QP}_r}\bigr)$ is monotone increasing. 
\end{prop}

\begin{defs}[Monotone limit form] \label{d:DFF}
Suppose~\ref{ass:CE} and~\ref{ass:ConditionalClos}. Define
\begin{align*} 
\dom{\vvE^{\U, \QP}}&:=\{u \in \cap_{r>0} \dom{\E^{\U, \QP}_r}: \lim_{r \to \infty}\E^{\U, \QP}_r(u) <+\infty\} \comma
\\
\vvE^{\U, \QP}(u)&:=\lim_{r \to \infty}\E^{\U, \QP}_r(u) \cquad 
\vvE^{\U, \QP}(u,v):=\frac{1}{4}\Bigl(\vvE^{\U, \QP}(u+v)- \vvE^{\U, \QP}(u-v)\Bigr) \notag \fstop
\end{align*} 
The form $(\vvE^{\U, \QP}, \dom{\vvE^{\U, \QP}})$ is a strongly local conservative $\QP$-symmetric Dirichlet form on~$L^2(\U, \QP)$.
The square field~$\cdc^\U$ is defined as the monotone limit of~$\cdc^\U_r$ as well:
\begin{align} \label{d:SF}
\cdc^\U(u):=\lim_{r \to \infty}\cdc^\U_r(u) \cquad \cdc^\U(u, v):= \frac{1}{4}\Bigl( \cdc^\U(u+v)- \cdc^\U(u-v)\Bigr)\cquad u, v \in \dom{\vvE^{\U, \QP}}\fstop 
\end{align}
Furthermore,  $(\vvE^{\U, \QP}, \dom{\vvE^{\U, \QP}})$ satisfies the Rademacher-type property:
\begin{align}\label{p:Rad}
\Lip(\U, {\mssd}_\U, \QP) \subset \dom{\vvE^{\U, \QP}}\comma \quad \cdc^\U(u) \le \Lip_{{\mssd}_\U}(u)^2 \qquad  u \in \Lip(\U, {\mssd}_\U, \QP) \fstop
\end{align}
\end{defs}
\begin{proof}
The proof works verbatim as \cite[Dfn.~4.14, Prop.~4.18]{Suz22b}, see also \cite[Dfn.~3.17]{Suz23b}.
\end{proof}

 \section{Partial matching pseudo distance}
In this section, we introduce a partial optimal matching pseudo-distance on~$\U=\U(\R^n)$ and discuss its properties. In particular, we prove that the partial $\ell^2$-matching pseudo-distance approximates the $\ell^2$-matching distance~$\mssd_\U$, which plays an essential role to show the identification between a Dirichlet form and a Cheeger energy in Section~\ref{s:ICE}.
Recall $B_r:=\{x \in \R^n: |x|<r\}$ and $\bar{B}_r:=\{x \in \R^n: |x| \le r\}$. Let $\mssd$ be the standard Euclidean distance in~$\R^n$. Recall that $\gamma_A:=\gamma\mrestr{A}$ is the restriction of $\gamma \in \U$ on $A \subset \R^n$.
\begin{defs}[Partial $\ell^2$-matching pseudo-distance] {\it The partial $\ell^2$-matching pseudo-
distance subjected to~$B_r$} is defined as:  \label{d:POM}
\begin{align} \label{d:TDD}
 \mssd_{\U}^{(r)}(\gamma, \eta):= \inf_{\alpha, \beta \in \U(\partial B_r)}\mssd_{\U(\bar{B}_r)}(\gamma_{B_r}+\alpha, \eta_{B_r}+\beta) \qquad  \gamma, \eta \in \U \fstop
\end{align}
\end{defs}
We summarise relevant properties of~$\mssd_{\U}^{(r)}$ for later arguments. Let $\sigma_r$ be the $\sigma$-algebra generated by the projection $\pr_{B_r}:\U \to \U(B_r)$ mapping~$\gamma \mapsto \gamma_{B_r}$. We denote by $\mcB(\sigma_r)$ the space of bounded $\sigma_r$-measurable functions in $\U$. In particular, for $u \in \mcB(\sigma_r)$, we have 
\begin{align}\label{e:SLB}
u(\gamma)=u(\gamma_{B_r}) \quad \gamma \in \U \fstop
\end{align}

\begin{prop}\label{p:PDF} The following properties hold:
\begin{enumerate}[{\rm (a)}]
\item \label{p:PDF1} $\mssd_{\U}^{(r)}$ is a pseudo-distance on~$\U$. Furthermore, $\mssd_{\U}^{(r)}(\gamma, \eta)=0$ implies $\gamma_{B_r}=\eta_{B_r}$$;$
\item \label{p:PDF2} $\mssd_{\U}^{(r)}$ is $\tau_\mrmv^{\times 2}$-continuous for every $r>0$ and  
\begin{align*}
\gamma_n \xrightarrow{\tau_\mrmv} \gamma \in \U \quad \iff \quad \mssd_{\U}^{(r)}(\gamma_n, \gamma) \to 0 \qquad \text{for every}\ r>0 \ ;
\end{align*}
\item  \label{p:PDF3} every $\mssd_{\U}^{(r)}$-Lipschitz function belongs to~$\mathcal B(\sigma_r)$. In particular, 
\begin{align}\label{ineq:LAC2}
\Lip(\U, \mssd_{\U}^{(r)}) \subset \mathcal C(\U, \tau_\mrmv) \cap \mathcal B(\U, \sigma_r) \ ;
\end{align}

\item  \label{p:PDF4} $\mssd_{\U}^{(r)}\nearrow \mssd_{\U}$ as $r \nearrow \infty$. In particular, 
\begin{align}\label{ineq:LAC1}
 &\Lip(\U, \mssd_{\U}^{(r)})  \subset \Lip(\U, \mssd_{\U}^{(s)}) \subset  \Lip(\U, \mssd_{\U}) \comma 
 \\
\notag &\Lip_{\mssd_{\U}}(u)\le \Lip_{\mssd_{\U}^{(s)}}(u)\le  \Lip_{\mssd_{\U}^{(r)}}(u) \comma \quad 0<r \le s \fstop
\end{align}

\end{enumerate}
\end{prop}
\begin{proof}  \ref{p:PDF1} We only prove the triangle inequality as the other properties are obvious by definition. Let $\gamma, \eta, \zeta \in \U$. Take $\alpha, \beta, \kappa, \theta \in \U(\partial B_r)$ such that 
$$\mssd_{\U(\bar{B}_r)}(\gamma_{B_r}+\alpha, \eta_{B_r}+\beta) \le \mssd_{\U}^{(r)}(\gamma, \eta)+\epsilon \comma \quad \mssd_{\U(\bar{B}_r)}(\eta_{B_r}+\kappa, \zeta_{B_r}+\theta) \le \mssd_{\U}^{(r)}(\eta, \zeta) +\epsilon \fstop$$
By the triangle inequality of~$\mssd_{\U(\bar{B}_r)}$ and the the contraction~$\mssd_{\U(\bar{B}_r)}(\gamma+\rho, \eta+\rho)\le \mssd_{\U(\bar{B}_r)}(\gamma, \eta)$ for every $\gamma, \eta, \rho \in \U(\bar{B}_r)$,  
\begin{align*}
\mssd_{\U}^{(r)}(\gamma, \zeta) &\le \mssd_{\U(\bar{B}_r)}(\gamma_{B_r}+\alpha+\kappa, \zeta_{B_r}+\beta+\theta) 
\\
&\le  \mssd_{\U(\bar{B}_r)}(\gamma_{B_r}+\alpha+\kappa, \eta_{B_r}+\beta+\kappa)+ \mssd_{\U(\bar{B}_r)}(\eta_{B_r}+\beta+\kappa, \zeta_{B_r}+\beta+\theta)\\
 &\le \mssd_{\U(\bar{B}_r)}(\gamma_{B_r}+\alpha, \eta_{B_r}+\beta)+ \mssd_{\U(\bar{B}_r)}(\eta_{B_r}+\kappa, \zeta_{B_r}+\theta)\\
&\le \mssd_{\U}^{(r)}(\gamma, \eta) + \mssd_{\U}^{(r)}(\eta, \zeta) +2\e   \fstop
\end{align*}
As $\e$ is arbitrary, the triangle inequality is proved. 

 \ref{p:PDF2} {\it The proof of {\rm (LHS) $\implies$ (RHS)}.} Fix arbitrary $r>0$ and take representatives~$\gamma_n\mrestr{B_r}=\sum_{i=1}^{N_n}\delta_{x_i^{(n)}}$, $\gamma\mrestr{B_r}=\sum_{i=1}^{N}\delta_{x_i}$ in terms of Dirac measures.
As $B_r$ is open and $\gamma_n \xrightarrow{\tau_\mrmv}\gamma$, we may assume  $N_n \ge N$ for sufficiently large $n$, and we may also assume that $\sup_{n \in \N}N_n<\infty$.
Due to the hypothesis $\gamma_n \xrightarrow{\tau_\mrmv}\gamma$, up to permutations of the labelling of the points, we may assume that $x_i^{(n)}$ converges to $x_i$ for $i=1, \ldots, N$ and
\begin{align*} 
\mssd_{\U}^{(r)}(\gamma_n, \gamma)^2
&=  \sum_{i=1}^{N} \mssd(x^{(n)}_{i}, x_i)^2 + \sum_{i=N+1}^{N_n}\mssd(x^{(n)}_{i}, \partial B_r)^2 \xrightarrow{n \to \infty}0 \fstop
\end{align*}
As $r>0$ is arbitrary, the proof is complete. 

{\it The proof of {\rm (RHS) $\implies$ (LHS)}.} Suppose $\mssd_{\U}^{(r)}(\gamma_n, \gamma) \to 0$ for every $r>0$. Let $\mssd_{\U}^{(r)}(\gamma_n, \gamma) = \mssd_{\U(\bar{B}_r)}(({\gamma_{n}})_{B_r}+\alpha_n, \gamma_{B_r}+\beta_n)$ with $\alpha_n, \beta_n \in \U(\partial B_r)$. Take a representative~$\gamma_{B_r}=\sum_{i=1}^{N_r}\delta_{x_i}$ and $(\gamma_n)_{B_r}=\sum_{i=1}^{N^{(n)}_r}\delta_{x^{(n)}_i}$. As $\mssd_{\U(\bar{B}_r)}((\gamma_{n})_{B_r}+\alpha_n, \gamma_{B_r}+\beta_n) \xrightarrow{n \to \infty} 0$, after suitably labelling points,  we have that  $x_i^{(n)}  \xrightarrow{n \to \infty}x_i$ for $i=1, 2, \ldots, N_r$ and $x_i^{(n)}$ converges to a point in $\partial B_r$ for $i \ge N_r+1$.  As this holds for every $r>0$, this concludes the sought statement.

\smallskip
 \ref{p:PDF3} This follows immediately by~\ref{p:PDF1}, \ref{p:PDF2} and the Lipschitz inequality $|u(\gamma)-u(\eta)| \le \Lip_{\mssd_{\U}^{(r)}}(u)\mssd_{\U}^{(r)}(\gamma, \eta)$. 

\smallskip
 \ref{p:PDF4} The monotonicity of~$\mssd_{\U}^{(r)} \le \mssd_{\U}^{(s)} \le \mssd_{\U}$ for $r \le s$ readily follows by definition. The formula~\eqref{ineq:LAC1} follows immediately by the monotonicity.
We prove that the monotone limit of $\mssd_{\U}^{(r)}$ is equal to $\mssd_\U$. Let $\gamma, \eta \in \U$ with $\mssd_\U(\gamma, \eta)<\infty$. Let $\gamma_r, \eta_r \in \U(\bar{B}_r)$ such that $\mssd_{\U}^{(r)}(\gamma, \eta)=\mssd_{\U(\bar{B}_r)}(\gamma_r, \eta_r)$. Then, it is immediate to see that $\gamma_r \to \gamma$ and $\eta_r \to \eta$ in $\tau_\mrmv$ as $r \to \infty$. By the~$\tau_\mrmv^{\times 2}$-lower semi-continuity of $\mssd_{\U}$, we have 
\begin{align*}
\mssd_{\U}(\gamma, \eta) 
&\le \liminf_{r \to \infty}\mssd_{\U}(\gamma_r, \eta_r) 
= \liminf_{r \to \infty}\mssd_{\U(\bar{B}_r)}(\gamma_r, \eta_r) 
\\
& = \liminf_{r \to \infty}\mssd_{\U}^{(r)}(\gamma, \eta) 
\le  \limsup_{r \to \infty}\mssd_{\U}^{(r)}(\gamma, \eta) 
 \le \mssd_{\U}(\gamma, \eta) \fstop  \qedhere
\end{align*}
%
\end{proof}

Let $\hat\mssd_r: \bar B_r^{\times 2} \to \R_+$ be the  function given by 
\begin{align} \label{d:sqd}
\hat\mssd_r(x, y):=\mssd(x, y) \wedge\Bigl( \inf_{z \in \partial B_r}\mssd(x, z)+ \inf_{z \in \partial B_r}\mssd(y, z)\Bigr)\comma
\end{align}
and  its symmetric product function on $\bar B_r^{\odot k}$ is given by:
\begin{align} \label{d:sqd2}
\hat\mssd_r^{\odot k}(x, y):=\inf_{\sigma \in \mathfrak S(k)}\hat\mssd_r^{\times k}(x_\sigma, y) \comma
\end{align}
where $x_\sigma=(x_{\sigma(1)}, \ldots,x_{\sigma(k)})$ for $x=(x_1,\ldots, x_k) \in  \bar B_r^{\times k}$. We introduce an equivalence relation  $x \sim_r y$ if $\hat \mssd_r(x,y)=0$. The quotient space $\bar B_r/\sim_r$ is denoted by $\hat B_r$, which is the space where all the boundary points $x \in \partial B_r$ are glued as one point. The function $\hat\mssd_r$ (resp.~$(\hat\mssd_r)^{\odot k}$) is a distance function on $\hat B_r$ (resp.~$\hat B_r^{\odot k}$), called {\it glued distance}.
\begin{prop} \label{p:RHDD}
For every $r>0$ and $k \in \N_0$, 
$$\mssd_\U^{(r)} =\hat\mssd_r^{\odot k} \quad \text{ on} \quad  \U^{k}(B_r) \cong  B_r^{\odot k}\fstop$$
\end{prop}
\begin{proof}
Let $\gamma=\sum_{i=1}^k \delta_{x_i}$, $\eta=\sum_{i=1}^k \delta_{y_i} \in \U^k(B_r)$ and $\alpha=\sum_{i=1}^l \delta_{w_i}, \beta=\sum_{i=1}^l \delta_{z_i} \in \U(\partial B_r)$ with $l \le k$. Then, we have  
\begin{align*}
&\mssd_{\U}^{(r)}(\gamma, \eta)^2= \inf_{\alpha, \beta \in \U(\partial B_r)}\mssd_{\U(\bar{B}_r)}(\gamma_{B_r}+\alpha, \eta_{B_r}+\beta)^2 
\\
&= \inf_{\alpha, \beta \in \U(\partial B_r)}  \inf_{\sigma \in \mathfrak S(k)} \Bigl(\sum_{i=1}^m \mssd(x_{\sigma(i)}, y_i)^2 + \sum_{i=m+1}^{k} \mssd(x_{\sigma(i)}, z_{i-m})^2+ \sum_{i=m+1}^{k} \mssd_r(w_{\sigma(i-m)}, y_i)^2  \Bigr) \comma
\end{align*}
where $m \in \N$ satisfies $k=l+m$. If $m=k$ (no particle in $B_r$ is coupled with configurations at the boundary), the last two terms are zero.
Interchanging two infimums, 
\begin{align*}
&  \inf_{\sigma \in \mathfrak S(k)}  \inf_{\alpha, \beta \in \U(\partial B_r)} \Bigl(\sum_{i=1}^m \mssd(x_{\sigma(i)}, y_i)^2 + \sum_{i=m+1}^{k} \mssd(x_{\sigma(i)}, z_{i-m})^2+ \sum_{i=m+1}^{k} \mssd(w_{\sigma(i-m)}, y_i)^2  \Bigr) 
\\
 &= \inf_{\sigma \in \mathfrak S(k)}\sum_{i=1}^k \Bigl(\mssd(x_{\sigma(i)}, y_i)^2 \wedge\bigl( \inf_{z \in \partial B_r}\mssd(x_{\sigma(i)}, z)+ \inf_{z \in \partial B_r}\mssd(y_i, z)\bigr)^2\Bigr)
 \\
 &= \inf_{\sigma \in \mathfrak S(k)}\sum_{i=1}^k \hat\mssd_r(x_{\sigma(i)}, y_i)^2 
 \\
&=  \hat\mssd_r^{\odot k}(\gamma, \eta)^2 \fstop \qedhere
\end{align*}
\end{proof}

\begin{prop}\label{p:CEEM}
Suppose that $\mu$ is a probability measure on~$(\U, \msB(\tau_\mrmv))$. Then, 
$(\U, \tau_\mrmv, \mssd_{\U}, \QP)$ is a complete geodesic extended metric measure space witnessed by the pseudo-distances~$(\mssd_{\U}^{(r)})_{r>0}$. 
\end{prop}
\begin{proof} The space $(\mcM, \tau_\mrmv)$ of Radon measures in $\R^n$ is Polish. 
Since $\U$ is a closed subspace in $\mcM$,  it is Polish as well. The completeness of $(\U, \mssd_\U)$ is standard. For the geodesic property, see, e.g., \cite[Cor.~2.7]{ErbHue15}.  The family of the pseudo-distances~$(\mssd_{\U}^{(r)})_{r>0}$ witnesses $(\U, \tau_\mrmv, \mssd_{\U}, \QP)$ to be an extended metric measure space thanks to~Prop.~\ref{p:PDF} and \cite[Lem.~4.2]{AmbErbSav16}. 
\end{proof}

Note that the same statement as Prop.~\ref{p:CEEM} holds even with the countable pseudo-distances $(\mssd_{\U}^{(r)})_{r \in \N}$.  Thus, by Prop.~\ref{p:CEEM} and Prop.~\ref{p:EMW}, we have
\begin{cor} \label{p:EMW1}
Let $1\le p<\infty$. The space $(\mathcal P(\U), \tau_\mrmw, \mssW_{p, \mssd_\U})$ is a complete geodesic extended metric topological space.
\end{cor}
\begin{prop} \label{p:PHC}
The algebra ${\mathcal C}:=\cup_{r>0}\Lip_b(\U, \mssd_\U^{(r)})$ is dense in~$L^2(\QP)$, separates points in~$(\U, \tau_\mrmv)$ and $\mcC \subset \dom{\vvE^{\U, \QP}}$.
\end{prop}
\begin{proof}
As $(\U, \tau_\mrmv, \mssd_{\U}, \QP)$ is an extended metric measure space witnessed by $(\mssd_\U^{(r)})_{r>0}$ by Prop.~\ref{p:CEEM}, the algebra $\cup_{r>0} \Lip_b(\U, \mssd_\U^{(r)})$ is dense in $L^2(\QP)$ and separates points by \cite[Lem.~4.5]{AmbErbSav16}. By the definition of $\mcC_r$ in Dfn.~\ref{d:core} and \eqref{ineq:LAC1}, we have $\mcC \subset \mcC_r$ for every $r>0$. For every $u \in \mcC$, there exists $r_0>0$ such that $u(\gamma_{B_r})=u(\gamma)$ for every $r \ge r_0$ due to \eqref{ineq:LAC2} in Prop.~\ref{p:PDF}. Thus, $r \mapsto \E^{\U, \QP}_r(u)$ is constant for $r \ge r_0$, which concludes $\underline{\E}^{\U, \QP}(u)=\lim_{r \to \infty}\E^{\U, \QP}_r(u)<\infty$. Thus, $\mcC \subset  \dom{\vvE^{\U, \QP}}$. 
\end{proof}

\begin{prop} \label{p:CES} Suppose~\ref{ass:CE} and~\ref{ass:ConditionalClos}. 
Then, $(\underline{\E}^{\U, \QP}, \mcC)$ is closable. We denote the closure by  $(\E^{\U, \QP}, \dom{\E^{\U, \QP}})$, which   is a strongly local conservative $\QP$-symmetric Dirichlet form. 
\end{prop}
\begin{proof}
Since $(\vvE^{\U, \QP}, \dom{\vvE^{\U, \QP}})$ is closed and $\mcC \subset  \dom{\vvE^{\U, \QP}}$,  the form~$(\underline{\E}^{\U, \QP}, \mcC)$ is closable. 
The conservativeness is immediate due to $\1 \in \mcC$ and $\E^{\U, \QP}(\1)=0$. 
The other properties inherit from the larger form~$(\vvE^{\U, \QP}, \dom{\vvE^{\U, \QP}})$  since the larger form is an extension of $(\E^{\U, \QP}, \dom{\E^{\U, \QP}})$.
\end{proof}

%
%

\section{Identification of Cheeger energy and Dirichlet forms} \label{s:ICE}
\subsection{Identification Theorem}
In this section, we identify $(\E^{\U, \QP}, \dom{\E^{\U, \QP}})$ and the Cheeger energy $(\Ch^{\mssd_{\U}, \QP}, \dom{\Ch^{\mssd_{\U}, \QP}})$. 
\begin{thm} \label{t:E=Ch}
Suppose~\ref{ass:CE} and~\ref{ass:ConditionalClos}. Then
$$\bigl(\Ch^{\mssd_{\U}, \QP}, \mathscr D(\Ch^{\mssd_{\U}, \QP})\bigr) = \bigl(\vE^{\U, \QP}, \dom{\E^{\U, \QP}}\bigr)\fstop$$
\end{thm}

 We split the proof into several steps. 
 Recall that $\mssd_{\vE^{\U, \QP}}$ is the intrinsic extended distance associated with $(\vE^{\U, \QP}, \dom{\E^{\U, \QP}})$, see \eqref{d:IDS}.
 \begin{lem} \label{l:11}
 $\mssd_{\U} \le \mssd_{\vE^{\U, \QP}}$. 
 \end{lem}
 \begin{proof}
Due to~Prop.~\ref{p:PDF},  $\mssd_{\U}^{(r)}(\cdot, \eta) \wedge c \in \mathcal C_b(\U, \tau_\mrmv)$ for every constant~$c>0$. By the Rademacher-type property~\eqref{p:Rad} and~\eqref{ineq:LAC1}, we have
  $$\vcdc^\U(\mssd_{\U}^{(r)}(\cdot, \eta) \wedge c) \le 1 \comma \quad  \eta \in \U\comma  c>0\fstop$$
Thus, by the definition of the intrinsic distance, we have 
 $$\mssd_{\U}^{(r)}(\gamma, \eta) \wedge c=\mssd_{\U}^{(r)}(\gamma, \eta) \wedge c- \mssd_{\U}^{(r)}(\eta, \eta)\wedge c \le \mssd_{\vE^{\U, \QP}}(\gamma, \eta)\fstop$$ 
 Since $\mssd_{\U}^{(r)}(\gamma, \eta) \wedge c \nearrow \mssd_{\U}(\gamma, \eta) \wedge c$ as $r \nearrow \infty$ by~Prop.~\ref{p:PDF}, we conclude 
 $$\mssd_{\U}(\gamma, \eta) \wedge c \le  \mssd_{\vE^{\U, \QP}}(\gamma, \eta) \comma \quad c>0 \fstop$$ 
 As $c$ is arbitrary, the proof is complete.
 \end{proof}
 
 \begin{lem} \label{l:21}
  $(\vE^{\U, \QP}, \dom{\E^{\U, \QP}}) \le \bigl(\Ch^{\mssd_{\U}, \QP}, \dom{\Ch^{\mssd_{\U}, \QP}}\bigr)$.
 \end{lem}
 \begin{proof}
 Let $(\Ch^{\mssd_{\vE^{\U, \QP}}, \QP}, \dom{\Ch^{\mssd_{\vE^{\U, \QP}}, \QP}})$ be the Cheeger energy associated with the intrinsic extended distance~$\mssd_{\vE^{\U, \QP}}$. By Lem.~\ref{l:11} and the definition of the slope \eqref{d:SLP}, we have the following inequality for the corresponding slopes:
  \begin{align*}
  |{\sf D}_{\mssd_{\vE^{\U, \QP}}} u|  \le  |{\sf D}_{\mssd_\U} u| \fstop
 \end{align*}
 Hence, by the definition of the Cheeger energy, we have
   \begin{align*}
  \bigl(\Ch^{\mssd_{\vE^{\U, \QP}}, \QP}, \dom{\Ch^{\mssd_{\vE^{\U, \QP}}, \QP}}\bigr) \le \bigl(\Ch^{\mssd_{\U}, \QP}, \dom{\Ch^{\mssd_{\U}, \QP}}\bigr) \fstop
 \end{align*}
 By~\cite[the first statement of Thm.~12.5]{AmbErbSav16}, 
    \begin{align*}
  (\vE^{\U, \QP}, \dom{\E^{\U, \QP}}) \le \bigl(\Ch^{\mssd_{\vE^{\U, \QP}}, \QP}, \dom{\Ch^{\mssd_{\vE^{\U, \QP}}, \QP}}\bigr)  \fstop
 \end{align*}
Hence, we have  
 $$ (\vE^{\U, \QP}, \dom{\E^{\U, \QP}}) \le \bigl(\Ch^{\mssd_{\U}, \QP}, \dom{\Ch^{\mssd_{\U}, \QP}}\bigr)\fstop \qedhere$$
 \end{proof}
 
 \begin{proof}[Proof of Thm.~\ref{t:E=Ch}]
 Thanks to~Lem.~\ref{l:21} and \eqref{i:MSL}, we only need to prove the following inequality: 
 \begin{align*}
 |{\sf D}_{\mssd_\U} u|^2 \le \vcdc^{\U}(u)\comma \quad  u \in \vC \quad \text{$\QP$-a.e.~.}
 \end{align*}
 As $u \in \vC=\cup_{r>0} \Lip_b(\mssd_{\U}^{(r)})$, due to \ref{p:PDF3} in Prop.~\ref{p:PDF}, there exists $r_0>0$ such that 
 \begin{align}\label{e:LS}
 u(\gamma)=u(\gamma_{B_r})\comma \quad  \gamma \in \U\comma r \ge r_0 \fstop
 \end{align}
  Thus, in the following argument, we may regard $u$ as a function on~$\U(B_r)$, if necessary.  
For $\zeta \in \U$, choose $r=r(\zeta) >r_0$ such that $\zeta(\partial B_r)=0$. Then, 
 \begin{align} \label{e:CDG}
 |{\sf D}_{\mssd_\U} u|(\zeta)&=\limsup_{\gamma \to \zeta}\frac{|u(\gamma)-u(\zeta)|}{\mssd_{\U}(\gamma, \zeta)} 
 =  \limsup_{\gamma \to \zeta}\frac{|u(\gamma_{B_r})-u(\zeta_{B_r})|}{\mssd_{\U}(\gamma, \zeta)}
 \\
& \le  \limsup_{\gamma \to \zeta}\frac{|u(\gamma_{B_r})-u(\zeta_{B_r})|}{\mssd_{\U}(\gamma_{B_r}, \zeta_{B_r})} \comma \notag
 \end{align}
 where the convergence $\gamma \to \zeta$ is in terms of the topology~$\tau_{\mssd_\U}$ metrised by $\mssd_{\U}$. Note that the last inequality in~\eqref{e:CDG} follows from the following argument: as $B_r$ is open and $\zeta(\partial B_r)=0$, we have that $(\gamma_n)_{B_r}\xrightarrow{\mssd_{\U}} \zeta_{B_r}$ as $n \to \infty$ in $\U(B_r)$ whenever $\gamma_n \xrightarrow{\mssd_{\U}} \zeta$, where $(\gamma_n)_{B_r}:=\gamma_n\mrestr{B_r}$. Thus, as long as $\gamma_n \xrightarrow{\mssd_{\U}} \zeta$, we have that, for sufficiently large~$n$, 
 $$\mssd_{\U}(\gamma_n, \zeta)^2 = \mssd_{\U}((\gamma_n)_{B_r}, \zeta_{B_r})^2 + \mssd_{\U}((\gamma_n)_{B_r^c}, \zeta_{B_r^c})^2 \ge  \mssd_{\U}((\gamma_n)_{B_r}, \zeta_{B_r})^2\comma$$
 which concludes the last inequality in~\eqref{e:CDG}. 
%
Thus, we have 
 $$ |{\sf D}_{\mssd_\U} u|(\zeta) \le  \limsup_{\gamma \to \zeta}\frac{|u(\gamma_{B_r})-u(\zeta_{B_r})|}{\mssd_{\U}(\gamma_{B_r}, \zeta_{B_r})} =  |{\sf D}_{\mssd_{\U(B_r)}} u|(\zeta_{B_r}) \comma$$
 where,  in the RHS, we regard $u$ as a function on~$\U(B_r)$ and $\mssd_{\U(B_r)}$ is the restriction of $\mssd_\U$ on $\U(B_r)$. 
Due to~\eqref{e:RDT1} and \ref{ass:CE}, we have 
 \begin{align*}
|{\sf D}_{\mssd_{\U(B_r)}} u|^2 = \cdc^{\U(B_r)}(u)  \qquad \text{$\QP_r^\eta$-a.e.~for $\QP$-a.e.~$\eta$} \quad r \ge r_0 \fstop
 \end{align*}
 Thus,  for $\QP$-a.e.~$\zeta$, 
  \begin{align*}
  |{\sf D}_{\mssd_\U} u|^2(\zeta)\le  |{\sf D}_{\mssd_{\U(B_r)}} u|^2(\zeta_{B_r}) = \cdc^{\U(B_r)}(u)(\zeta_{B_r}) = \vcdc^{\U}_r(u)(\zeta) = \vcdc^{\U}(u)(\zeta) 
 \end{align*}
where the second last equality follows from the equality~\eqref{e:TG} and regarding $u$ as a function on $\U(B_r)$ due to \eqref{e:LS};  the last equality follows from~\eqref{d:SF}  and~\eqref{e:LS}.
%
 \end{proof}
We obtain the following as a byproduct:
 \begin{cor}Suppose~\ref{ass:CE} and~\ref{ass:ConditionalClos}. 
Then, 
\begin{align} \label{e:SC}
\vcdc^{\U}(u) = |{\sf D}_{\mssd_\U} u|^2 \cquad u \in \vC \fstop
\end{align}
Furthermore, the Rademacher-type property holds: 
\begin{align} \label{in:RAMT}
\Lip_b(\U, \mssd_\U, \QP) \subset \dom{\E^{\U, \QP}} \cquad \cdc^{\U}(u) \le \Lip_{\mssd_\U}(u)^2 \quad u \in \Lip_b(\U, \mssd_\U, \QP) \fstop
\end{align}
 \end{cor}
 \begin{proof}
 The first statement has been seen in the proof of Thm.~\ref{t:E=Ch}. The second statement follows \eqref{i:MSL} and Thm.~\ref{t:E=Ch}. 
 \end{proof}

One of the significant consequences of Thm.~\ref{t:E=Ch} is that we obtain the quasi-regularity with respect to the vague topology~$\tau_\mrmv$, which implies the existence of the associated diffusion process in $\U$ whose transition probability is a $\tau_\mrmv$-continuous $\QP$-representative of the $L^2$-semigroup $T^{\U, \QP}_t$, see \cite[Thm.~3.5 in Chapter IV, Thm.~1.5 in Chapter V]{MaRoe90}. 
The $\tau_\mrmv$-quasi-regularity will be also used to prove the $p$-Bakry--\'Emery gradient estimate via Savar\'e's self-improvement argument in Cor.~\ref{p:PBE}. 
The idea of the following proof is essentially from \cite[Cor.~3.4]{RoeSch99}. We give the proof for the sake of completeness. 
  \begin{cor}[$\tau_\mrmv$-quasi-regularity]\label{c:QR}
  Suppose~\ref{ass:CE} and~\ref{ass:ConditionalClos}. 
 Then, the form $(\E^{\U, \QP}, \dom{\E^{\U, \QP}})$   is $\tau_\mrmv$-quasi-regular.
 \end{cor}
  \begin{proof}
 \ref{QR2} is obvious  as $\mathcal C\subset \mathcal C(\U, \tau_\mrmv)$ and $\dom{\E^{\U, \QP}}$ is the closure of $\mathcal C$. 
 
\ref{QR3} follows by~Prop.~\ref{p:PHC}. 

 \ref{QR1}  By~\eqref{d:TC}, it is equivalent to show the tightness of the capacity $\Cap_{\E^{\U, \QP}}$. As $\QP$ is a Borel probability in a Polish space~$(\U, \tau_\mrmv)$, it is a Radon measure. Thus, there exists a sequence $(K_n)_{n \in \N}$ of $\tau_\mrmv$-compact sets such that $\lim_{n \to \infty}\QP(K_n)=1$. Without loss of generality, we may assume that $K_n \subset K_{n+1}$ for $n \in \N$. Define $\rho_{K_n}(\gamma):=\inf_{\eta \in K_n}\mssd_\U(\gamma, \eta)$. The function $\gamma \mapsto \rho_{K_n}(\gamma)$ is $\tau_\mrmv$-lower semi-continuous (\cite[(vii) Lem.~4.1]{RoeSch99}). In particular, $\rho_{K_n}\wedge c \in \Lip_b(\U, \mssd_\U, \QP)$ for $c \in \R_+$. By Thm.~\ref{t:E=Ch} and~\eqref{i:MSL}, 
 $$\sup_{n \in \N} \E^{\U, \QP}_1(\rho_{K_n} \wedge 2) <+\infty \comma$$
 where $\E^{\U, \QP}_1(u):= \E^{\U, \QP}(u)+\|u\|_{L^2}^2$. Thus, by, e.g.,~\cite[Lem.~I.2.12 on p.21]{MaRoe90}, there exists a  subsequence $k_n$ such that $u_n:=\frac{1}{n} \sum_{j=1}^n (\rho_{K_{k_j}} \wedge 2)$ converges to $u \in \dom{\E^{\U, \QP}}$ in terms of the norm~$(\E^{\U, \QP}_1)^{1/2}$. By the monotonicity~$\rho_{K_n} \ge \rho_{K_{n+1}}$ for $n \in \N$ and $\rho_{K_n}\equiv0$ on $K_n$, we have that $u \equiv 0$ on the set~$\cup_{n \in \N} K_{n}$ of full measure. Furthermore,  $\{\rho_{K_{k_n}} > 1/2\} \subset \{u_n > 1/2\}$ by construction. Noting that the sublevel set $\{\rho_{K_n} \le 1/2\}$ is $\tau_\mrmv$-compact (\cite[(vii) Lem.~4.1]{RoeSch99}), the following computation completes the proof
  $$\Cap_{\E^{\U, \QP}}\Bigl( \{\rho_{K_n}>1/2\}\Bigr) \le \Cap_{\E^{\U, \QP}}\Bigl( \{u_n>1/2\}\Bigr) \le \E_1^{\U, \QP}(2u_n) \xrightarrow{n \to \infty}0 \fstop \qedhere$$

 \end{proof}
 
 \subsection{Identification with upper Dirichlet form}
 In this section, we work on the configuration space~$\U=\U(\R)$ over the one-dimensional space~$\R$, and  we identify $(\E^{\U, \QP}, \dom{(\E^{\U, \QP}})$ with the Dirichlet form constructed in \cite[Theorem 1]{Osa96}, called {\it upper Dirichlet form} (see~\cite{KawOsaTan21} for this terminology).
For a function $u:\U \to \R$, $\eta \in \U$, $r>0$ and $k \in \N_0$, recall that  a function $u_r^{k, \eta}: \U^k(B_r) \to \R$ was defined in~\eqref{e:SEF} as 
$$u_r^{k, \eta}(\gamma):=u(\gamma+\eta_{B_r}) \quad \gamma \in \U^k(B_r) \fstop$$
Let $\mathbf u_r^{k, \eta}: B_r^{\times k} \to \R$ be a unique symmetric function such that $\mathbf u_r^{k, \eta}=u_r^{k, \eta}\circ \mssP_k$, where $\mssP_k$ is the canonical projection from $B_r^{\times k} \to B_r^{\times k}/\mathfrak S_k$. We call $\{\mathbf u_r^{k, \eta}: r>0, k \in \N_0, \eta \in \U\}$ {\it labelled representative of $u$}. 
Recall that a function $u: \U \to \R$ is said to be {\it local} if there exists $R>0$ such that $u(\gamma)=u(\gamma_{B_R})$ for every $\gamma \in \U$. 
\begin{defs}[Smooth local function {\cite[(1.2)]{Osa96}}] \ \label{d:SLF}
\begin{itemize}
\item A function $u: \U \to \R$ is {\it  smooth} if $\mathbf u_r^{k, \eta}$ is smooth in $ B_r^{\times k}$ for every $r>0$, $k \in \N_0$ and $\eta \in \U$. 
We denote by $\mcD_{loc}^\infty$ the space of all smooth local functions. 
\item Under~\ref{ass:CE} and~\ref{ass:ConditionalClos}, we define  
$$\mcD_{loc, \QP}^\infty:=\{u\in \mcD_{loc}^\infty: \underline{\E}^{\U, \QP}(u)+\|u\|_{L^2(\QP)}<\infty\} \fstop$$
{According to~Dfn.~\ref{d:core}, $\mcD_{loc, \QP}^\infty \subset \cup_{r>0} \mcC_r \subset \dom{\underline{\E}^{\U, \QP}}$}, thus $\underline{\E}^{\U, \QP}(u)$ is well-defined for $u \in \mcD_{loc, \QP}^\infty$ and $(\underline{\E}^{\U, \QP}, \mcD_{loc, \QP}^\infty)$ is closable since  $(\underline{\E}^{\U, \QP}, \dom{\underline{\E}^{\U, \QP}})$ is a closed extension of it. The closure is denoted by $(\overline{\E}^{\U, \QP}, \dom{\overline{\E}^{\U, \QP}})$. 
Due to the locality of $u$, there exists $R>0$ such that $\cdc^{\U}_r(u)=\cdc^{\U}(u)$ for every $r \ge R$ and 
\begin{align}\label{e:LCU}
 \overline{\E}^{\U, \QP}(u)=\frac{1}{2}\int_\U \cdc^{\U}(u) \diff \QP = \frac{1}{2}\int_\U \cdc_r^{\U}(u) \diff \QP \quad r \ge R\fstop
 \end{align}
\end{itemize}
\end{defs}  \label{r:LNL}
\begin{rem}\ Let $u$ be a smooth local function. 
\begin{enumerate}[(a)]
\item \label{r:LNL1} Its labelled representative $\mathbf u_r^{k, \eta}$ can be regarded as an element of the space~$\mcC^\infty_{sym}(\hat B_r^{\times k})$ of smooth symmetric functions for sufficiently large $r>0$, where $\hat B_r = \R/r\mbbZ$.  Indeed, as $u(\gamma)=u(\gamma_{B_r})$ for every $\gamma \in \U$ for sufficiently large $r>0$, the function~$u$ does not depend on configurations on the boundary $\partial B_r$ (as well as the complement $B_r^c$). Thus, its labelled representative $\mbfu_r^{k, \eta}$ does not depend on $\eta$, so it is constant on the boundary $\partial B_r^{\times k}$. Thus, we may regard $\mbfu_r^{k, \eta}$ as a function defined in~$\hat B_r^{\times k}$ for every $k$. 
\item  \label{r:LNL2} Although~$u|_{\U^k(B_r)} \in \Lip_b(\U^k(B_r), \mssd_\U)$ for each $k$, in general $u \not\in\Lip_b(\U, \mssd_\U)$, see~\cite[Example 4.35]{LzDSSuz21}, where $\Lip_{\mssd_{\U}}(u|_{\U^k(B_r)})$ could diverge as $k \to \infty$.
\end{enumerate} 
\end{rem}
\paragraph{Parallel extension}  
Here, we introduce a way to extend smooth symmetric functions in $\hat B_r^{\times k} \cong (\R/r\mbbZ)^{\times k}$ to smooth local functions in $\U(B_r)$.  
 Let $\hat\mssd_r^{\times k}$ be the product distance of the glued distance~$\hat \mssd_r$ given in~\eqref{d:sqd}, which is the geodesic distance on the $k$-dimensional flat torus~$\hat B_r^{\times k}\cong (\R/r\mbbZ)^{\times k}$.  Define the diagonal set $D_k=\{(x_1, \ldots, x_k) \in B_r^{\times k}: x_i = x_j \ \text{for some $i \neq j$}\}.$ For $\mbfx_k=(x_1,\ldots, x_k) \in \hat B_r^{\times k}$, define~
\begin{itemize}
\item $d(\mbfx_k):=\argmin\{\hat\mssd_r^{\times k}(\mbfx_k, \mbfy_k): \mbfy_k \in D_k\}$, i.e., the closest point  in the diagonal set from $\mbfx_k$; 
\item $R(d(\mbfx_k), \mbfx_k)$ to be  the ray starting from $d(\mbfx_k)$ passing through $\mbfx_k$;
\item  $b(\mbfx_k)=(b(x_1),\ldots, b(x_k))$ to be the element in the intersection $R(d(\mbfx_k), \mbfx_k) \cap \partial \hat B_r^{\times k-1}$.
\item $a(\mbfx_k):=\argmin\{\mssd(x, \partial \hat B_r): x \in \{x_1,\ldots, x_k\}\}$. If there are more than one such elements,  we choose the one with the largest index $i$ among such $x_i$'s; 
\item Let $\pi_{k, k+1}: \hat B_r^{\times k+1} \to  \hat B_r^{\times k}$ be the projection defined by removing $a(\mbfx_{k+1})$: suppose $x_l=a(\mbfx_{k+1})$ for some $l$ and define
$$(x_1, \ldots, x_{l-1},  a(\mbfx_{k+1}), x_{l+1}, \ldots, x_{k+1}) \mapsto (x_1, x_2, \ldots, x_{l-1},  x_{l+1}, \ldots, x_{k+1}) \fstop$$
\end{itemize}

For $u^{(k)} \in \Lip_{sym, b}(\hat B_r^{\times k}, \hat\mssd_r^{\times k})$, we define $u^{(k,1)}: \hat B_r^{\times k+1} \to \R$ as 
$$u^{(k,1)}(\mbfx_{k+1})=u^{(k)}\circ \pi_{k, k+1}(b(\mbfx_{k+1})) \fstop$$
\begin{center}
\begin{tikzpicture}[scale=1]
  \def\r{2.5}

  \draw[-] (-\r-0.6,0) -- (\r+0.6,0) ;
  \draw[-] (0,-\r-0.6) -- (0,\r+0.6) ;

  \draw[thick] (-\r,-\r) rectangle (\r,\r);

   \node[below] at (-\r-0.3,0) {$-r$};
  \node[below] at (\r+0.3,0) {$r$};
  \node[left] at (0,-\r-0.3) {$-r$};
  \node[left] at (0,\r+0.3) {$r$};
   \draw[dashed] (-\r,-\r) -- (\r,\r);
  \node[circle,fill,inner sep=1.2pt,label=left:{$\mathbf{x}_k$}] (pt) at (2.0,0.8) {};
  \node[circle,fill,inner sep=1.2pt,label=left:{$d(\mathbf{x}_k)$}] (pt) at (1.4,1.4) {};
    \node[circle,fill,inner sep=1.2pt,label=right:{$b(\mathbf{x}_k)$}] (pt) at (2.5,0.30) {};


 \draw[dotted]  (2.0,0.8) -- (1.4, 1.4);
 \draw[dotted]  (2.0,0.8) -- (2.5,0.30) ;
\end{tikzpicture}
\\
\smallskip 
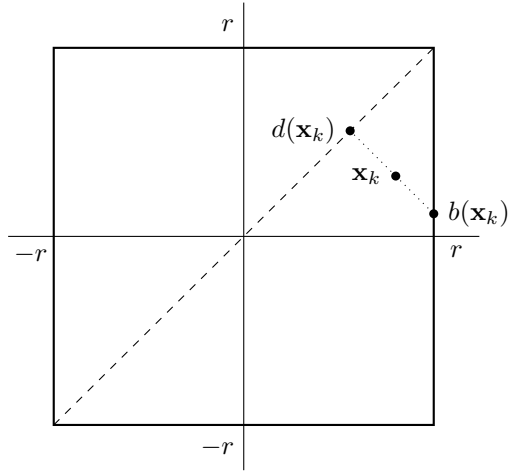
\noindent{The picture for $k=2$}
  \label{fig:square-projection}
\end{center}
\bigskip
By construction, $u^{(k,1)}$ is constant along each ray $R(d(\mbfx_{k+1}), \mbfx_{k+1})$ for every $\mbfx_k$:
$$u^{(k,1)} \equiv  u^{(k)}\circ \pi_{k, k+1}\bigl(b(\mbfx_{k+1})\bigr) \quad \text{on} \quad R(d(\mbfx_{k+1}), \mbfx_{k+1})\fstop$$
The function $u^{(k,1)}$ is well-defined in~$\hat B_r^{\times k+1}$. To see this, as we already know that $u^{(k)}$ is well-defined in $\hat B_r^{\times k}$, it suffices to see that $u^{(k,1)}$ does not depend on the particles at the boundary in the $(k+1)$-coordinate. Indeed, take two points in the boundary $\partial B_r^{\times k+1}$: $\mbfx_{k+1}=(x_1,x_2,\ldots, x_{k}, r)$ and $\mbfx_{k+1}'=(x_1,x_2,\ldots, x_{k}, -r)$.
Then,
$$b(\mbfx_{k+1})=\mbfx_{k+1} \cquad b(\mbfx_{k+1}')=\mbfx_{k+1}' \cquad a(\mbfx_{k+1})=x_{k+1}=r \cquad a(\mbfx_{k+1})=x_{k+1}=-r \fstop$$
Hence, the two points $b(\mbfx_{k+1})$ and $b(\mbfx_{k+1}')$ are sent to the same point $\mbfx_k$ by the projection:
$$\pi_{k, k+1}(b(\mbfx_{k+1})) = (x_1,x_2,\ldots, x_{k}) \cquad \pi_{k, k+1}(b(\mbfx_{k+1}')) = (x_1,x_2,\ldots, x_{k}) \fstop$$
Therefore, $u^{(k,1)}(\mbfx_{k+1})=u^{(k,1)}(\mbfx_{k+1}')=u^{(k)}(\mbfx_k)$.   

It is easy to see that $u^{(k,1)}$ is symmetric and that  the Lipschitz constant is bounded as 
$$\Lip_{\mssd_r^{\times k+1}}(u^{(k,1)}) \le \sqrt{2}\Lip_{\mssd_r^{\times k}}(u^{(k)}) \fstop$$
If $u^{(k)} \in \mcC_{sym}^\infty(\hat B_r^{\times k})$, then $u^{(k, 1)} \in  \mcC_{sym}^\infty(\hat B_r^{\times k+1})$ as well by construction. 
By induction, we define $u^{(k,l)} \in  \Lip_{sym, b}(\hat B_r^{\times k+l}, \hat\mssd_r^{\times k+l})$ as 
$$u^{(k,l)}(\mbfx_{k+l})=u^{(k, l-1)}\circ \pi_{k+l-1, k+l}\bigl(b(\mbfx_{k+l})\bigr) \quad l \in \N \fstop$$ 
We similarly have the Lipschitz estimate:
\begin{align}\label{i:LIR}
\Lip_{\hat\mssd_r^{\times k+l}}(u^{(k,l)}) \le {2}^{l/2}\Lip_{\hat\mssd_r^{\times k}}(u^{(k)}) \qquad l \in \N \fstop
\end{align}
Furthermore,  for non-positive $l$, we define 
 $$u^{(k,l)}(\mbfx_{k+l}):=u^{(k)}(x_1, \ldots, x_{k+l}, r, r, \ldots, r) \cquad  l =0, -1, \ldots, -k  \fstop$$
 Obviously, 
 \begin{align}\label{e:LEC}
 \Lip_{\hat\mssd_r^{\times k+l}}(u^{(k,l)}) \le \Lip_{\hat\mssd_r^{\times k}}(u^{(k)}) \cquad l =0, -1, \ldots, -k  \fstop
 \end{align}
  If $u^{(k)} \in \mcC_{sym}^\infty(\hat B_r^{\times k})$, then, for every~$l \in \{-k, -k+1, \ldots, 0\} \cup \N$, 
  \begin{align} \label{e:SLFN}
  u^{(k, l)} \in  \mcC_{sym}^\infty(\hat B_r^{\times k+l})  \fstop
  \end{align}
\begin{defs}[Parallel extension] For $u^{(k)} \in \Lip_{sym, b}(\hat B_r^{\times k}, \hat\mssd_r^{\times k})$, we define
\begin{align*}
\Phi(u^{(k)}):= \sum_{l=-k}^\infty u^{(k, l)} \1_{\hat B_r^{\times k+l}} \fstop
\end{align*}
\end{defs}
We summarise key properties of $\Phi(u^{(k)})$ below.
\begin{itemize}
\item $\Phi(u^{(k)})$ is symmetric. Thus, we can regard it as a function in $\U(B_r)$ by defining
\begin{align} \label{e:ESLS}
\Phi(u^{(k)})(\gamma):=\Phi(u^{(k)})(x_1, \ldots, x_l) \cquad \gamma=\sum_{i=1}^l \delta_{x_i} \in \U(B_r) \quad l=\gamma(B_r) \fstop
\end{align}
In the following argument,  if necessary, we use the same notation $\Phi(u^{(k)})$ to indicate the function in $\U(B_r)$. 
\item $\|\Phi(u^{(k)})\|_{L^\infty}=\|u^{(k)}\|_{L^\infty}$. 
\item If $u^{(k)} \in \mcC_{sym}^\infty(\hat B_r^{\times k})$, then $\Phi(u^{(k)})$ is a smooth local function by \eqref{e:SLFN}. 
\item Thanks to~\eqref{i:LIR}, \eqref{e:LEC} and Prop.~\ref{p:RHDD} with the symmetry of $u^{(k)}$, we have 
\begin{align}\label{e:LIAL}
\Lip_{\mssd_{\U}^{(r)}}\bigl(\Phi(u^{(k)})|_{\U^{k+l}(B_r)}\bigr) \le 2^{\frac l2}\Lip_{\hat\mssd_r^{\times k}}(u^{(k)}) \cquad l \in \N \comma
\end{align}
and 
\begin{align}\label{e:LIAL5}
\Lip_{\mssd_{\U}^{(r)}}\bigl(\Phi(u^{(k)})|_{\U^{k+l}(B_r)}\bigr) \le \Lip_{\hat\mssd_r^{\times k}}(u^{(k)}) \cquad l=0, -1, \ldots, -k  \fstop
\end{align}
\item For any sequence $(u_n^{(k)})_{n \in \N} \subset\Lip_{sym, b}(\hat B_r^{\times k}, \hat\mssd_r^{\times k})$,
\begin{align}\label{e:LIAL1}
u_n^{(k)}(x) \xrightarrow{n \to \infty} u^{(k)}(x) \implies \Phi(u_n^{(k)})(x) \xrightarrow{n \to \infty} \Phi(u^{(k)})(x) \qquad x \in \sqcup_{k \in \N_0} \hat B_r^{\times k} \fstop
\end{align}
\item Let $u: \sqcup_{k=0}^\infty(\hat B_r)^{\times k} \to \R$ such that $u^{(k)}:=u|_{\hat B_r^{\times k}}  \in \Lip_{sym, b}(\hat B_r^{\times k}, \hat\mssd_r^{\times k})$ for $k \in \N_0$. If, for every $m, k \in \N_0$ and $l\in \{-k, -k+1, \ldots, 0\}$ with $m=k+l$,  
\begin{align} \label{e:CSC}
u^{(m)}=u^{(k, l)}
\end{align}
then  
\begin{align}\label{e:LIAL2}
\Phi(u^{(k)})(x) \xrightarrow{k \to \infty} u(x) \fstop \qquad x \in \sqcup_{k \in \N_0} \hat B_r^{\times k} \fstop
\end{align}
Indeed, if $x \in  \hat B_r^{\times m}$ for some $m$, then take $k \ge m$ and $l$ such that $m=k+l$. Then,  
$$\Phi(u^{(k)})(x)=u^{(k, l)}(x)=u^{(m)}(x)=u|_{\hat B_r^{\times m}}(x)=u(x) \fstop$$
\end{itemize}

 \begin{thm} \label{t:IUL} Suppose~$\U=\U(\R)$, \ref{ass:CE}, \ref{ass:ConditionalClos} and that for every $r>0$ there exists a constant $c_r>0$ such that 
 \begin{align} \label{e:AFD} 
 \sum_{l=0}^\infty 2^{\frac{l}{2}}\mu\bigl(\gamma \in \U: \gamma(B_r)=l\bigr) <c_r  \fstop
 \end{align}
Then, 
 $$(\overline{\E}^{\U, \QP}, \dom{\overline{\E}^{\U, \QP}})= \bigl(\Ch^{\mssd_{\U}, \QP}, \mathscr D(\Ch^{\mssd_{\U}, \QP})\bigr) = \bigl(\vE^{\U, \QP}, \dom{\E^{\U, \QP}}\bigr)\fstop$$
 \end{thm}
 \begin{proof}As the second equality has been proven in Thm.~\ref{t:E=Ch}, it suffices to prove 
 \begin{align} \label{p:S1}
 \bigl(\Ch^{\mssd_{\U}, \QP}, \mathscr D(\Ch^{\mssd_{\U}, \QP})\bigr)  \le (\overline{\E}^{\U, \QP}, \dom{\overline{\E}^{\U, \QP}})
 \end{align}
 and 
  \begin{align} \label{p:S2}
  (\overline{\E}^{\U, \QP}, \dom{\overline{\E}^{\U, \QP}}) \le \bigl(\vE^{\U, \QP}, \dom{\E^{\U, \QP}}\bigr) \fstop
   \end{align}
 
 \paragraph{Proof of \eqref{p:S1}} Let $u \in \mcD_{loc, \QP}^\infty$ and take $r>0$ such that $u(\gamma)=u(\gamma_{B_r})$ for every $\gamma \in \U$. By the exactly same proof as Thm.~\ref{t:E=Ch}, 
  \begin{align}\label{e:CC3}
 |{\sf D}_{\mssd_\U} u|^2 \le \vcdc^{\U}(u)\comma \quad \text{$\QP$-a.e.~.}
 \end{align}
Thus, it suffices to show that $u \in \dom{\Ch^{\mssd_\U, \QP}}$. Note that $u \not\in\Lip_b(\U, \mssd_\U)$ in general as remarked in~\ref{r:LNL2} in Rem.~\ref{r:LNL}, however $u$ is Lipschitz along every absolutely continuous curve in $\U$. Indeed, 
 let $(\gamma_t)_{t \in [0,1]} \subset \U$ be a $\mssd_\U$-absolutely continuous curve. 
 As $\mssd_\U$ metrises a stronger topology than the vague topology~$\tau_\mrmv$, the curve $t \mapsto \gamma_t$ is $\tau_\mrmv$-continuous. In particular, there exists $k \in \N$ such that $\sup_{t \in [0,1]}\gamma_t(B_r) \le k$. Since $u|_{\U^k(B_r)} \in \Lip_b(\U^k(B_r), \mssd_\U)$, $u$ is Lipschitz along $(\gamma_t)_{t \in [0,1]}$. Hence, the curve $t \mapsto u\circ \gamma_t$ is absolutely continuous, which implies that the slope $|{\sf D}_{\mssd_\U} u|$ is an upper gradient of $u$ along every  $\mssd_\U$-absolutely continuous curve (see \cite[Remark 2.8]{AmbGigSav14}). In particular, the slope~$|{\sf D}_{\mssd_\U} u|$ is a weak upper gradient of $u$ in the sense of \cite[Dfn.~5.4]{AmbGigSav14}. Furthermore, $|{\sf D}_{\mssd_\U} u| \in L^2(\U, \QP)$ due to~\eqref{e:CC3}.  Using the equivalence~\cite[Thm.~6.2]{AmbGigSav14} between the Cheeger energy given in~\eqref{d:CH1} and the one in terms of the minimal weak upper gradient (\cite[Dfn.~5.11]{AmbGigSav14}), we conclude $u \in \dom{\Ch^{\mssd_\U, \QP}}$.

\paragraph{Proof of \eqref{p:S2}} 
It suffices to show 
 \begin{align} \label{ie:DI}
 \dom{\E^{\U, \QP}} \subset \dom{\overline\E^{\U, \QP}} \fstop
 \end{align}
Indeed, the Dirichlet form~$(\underline{\E}^{\U, \QP}, \dom{\underline{\E}^{\U, \QP}})$ given in Dfn.~\ref{d:DFF} is a closed extension of both $(\overline{\E}^{\U, \QP}, \dom{\overline{\E}^{\U, \QP}})$ and $({\E}^{\U, \QP}, \dom{{\E}^{\U, \QP}})$, thus \eqref{ie:DI}
 implies $\E^{\U, \QP}=\overline{\E}^{\U, \QP}$ on $\dom{\E^{\U, \QP}}$, which concludes the inequality~\eqref{p:S2}.  
 As $\mcC=\cup_{r>0}\Lip_b(\U, \mssd_\U^{(r)})$ is dense in $\dom{\E^{\U, \QP}}$,
 it suffices to show $\Lip_b(\U, \mssd_\U^{(r)}) \subset \dom{\overline{\E}^{\U, \QP}}$.  Let $u \in \Lip_b(\U, \mssd_\U^{(r)})$ and $u^{(k)}:=u|_{\U^k(B_r)}$. As $u$ does not depend on configurations on the boundary $\partial B_r$ due to~\ref{p:PDF3} in Prop.~\ref{p:PDF}, we may regard $u^{(k)}$ as a symmetric $\hat \mssd_r^{\times k}$-Lipschitz functions in $\hat B_r^{\times k}$ such that \eqref{e:CSC} holds. Furthermore, by  Prop.~\ref{p:RHDD} with the symmetry of $u^{(k)}$, 
 \begin{align} \label{e:ELCD}
 \Lip_{\hat \mssd^{\times k}_r}(u^{(k)}) = \Lip_{\mssd_\U^{(r)}}(u) \fstop
 \end{align}
 Take $u^{(k)}_n \in C_{sym}^\infty(\hat B_r^{\times k})$ such that 
 $$\Lip_{\hat \mssd^{\times k}_r}(u^{(k)}_n) \le \Lip_{\hat \mssd^{\times k}_r}(u^{(k)}) \quad \text{and} \quad u^{(k)}_n \to u^{(k)} \quad  \text{uniformly}\fstop$$
  We can take such $u_n^{(k)}$ by e.g., $u_n^{(k)}:=T_{1/n}u^{(k)}$, where $(T_t)_{t \ge 0}$ is the heat semigroup on the $k$-dimensional flat torus $\hat B_r^{\times k}= (\R/r\Z)^{\times k}$. Indeed, as the Ricci curvature of the flat torus is $0$,  we can apply the Bakry--\'Emery gradient estimate~$\BE(0,\infty)$ to get the  gradient bound, hence,  the sought Lipschitz bound of $u_n^{(k)}$ as well.  
By~\eqref{e:LIAL}, \eqref{e:LIAL5} and \eqref{e:ELCD},  we have 
\begin{align}\label{e:ULEN1}
\Lip_{\mssd_\U^{(r)}}(\Phi(u^{(k)}_n)|_{\U^{k+l}(B_r)}) \le 2^{\frac l2}\Lip_{\mssd_\U^{(r)}}(u) \quad l \in \N \comma
\end{align}
and 
\begin{align}\label{e:ULEN2}
 \Lip_{\mssd_\U^{(r)}}(\Phi(u^{(k)}_n)|_{\U^{k+l}(B_r)}) \le \Lip_{\mssd_\U^{(r)}}(u)  \quad  l=0, -1, \ldots, -k \fstop
 \end{align}
Furthermore, by~\eqref{e:LIAL1}
 \begin{align} \label{e:ULEN3}
 \Phi(u^{(k)}_n) \xrightarrow{n \to \infty} \Phi(u^{(k)}) \quad \text{everywhere} \fstop
 \end{align}
 By the locality~\eqref{e:LCU} of $u$, the hypothesis~\eqref{e:AFD}, the Rademacher-type property~\eqref{p:Rad} and~\eqref{ineq:LAC1},  we have
\begin{align} \label{e:UBT}
&\overline\E^{\U, \QP}(\Phi(u_n^{(k)})) \notag
= \overline\E_r^{\U, \QP}(\Phi(u_n^{(k)})) \notag
\\
&\le  \Lip_{\mssd_\U^{(r)}}(u)^2 \sum_{l=0}^k \mu\bigl(\gamma \in \U: \gamma(B_r)=l\bigr) + \Lip_{\mssd_\U^{(r)}}(u)^2\sum_{l=k}^\infty 2^{\frac l2}\mu\bigl(\gamma \in \U: \gamma(B_r)=l\bigr) \notag 
\\
&\le  \Lip_{\mssd_\U^{(r)}}(u)^2(1+ c_r) \fstop  
\end{align}
Thus, $\Phi(u_n^{(k)}) \in \mcD_{loc, \QP}^\infty$. By \eqref{e:ULEN1}--\eqref{e:ULEN3}, the same Lipschitz bound inherits to the limit~$\Phi(u^{(k)})$:
  $$ \Lip_{\mssd_\U^{(r)}}\bigl(\Phi(u^{(k)})|_{\U^{k+l}(B_r)}\bigr) \le 2^{\frac l2}\Lip_{\mssd_\U^{(r)}}(u) \quad  l \in \N \comma$$
  and 
   $$ \Lip_{\mssd_\U^{(r)}}\bigl(\Phi(u^{(k)})|_{\U^{k+l}(B_r)}\bigr) \le \Lip_{\mssd_\U^{(r)}}(u) \quad l=0, -1, \ldots, -k \fstop$$
By the same estimate as in~\eqref{e:UBT}, 
\begin{align} \label{e:UBT1}
\overline\E^{\U, \QP}(\Phi(u^{(k)})) \le  \Lip_{\mssd_\U^{(r)}}(u)^2(1+ c_r) \fstop
\end{align}
%
Thus, $\Phi(u^{(k)}) \in  \dom{\overline{\E}^{\U, \QP}}$. Letting $k \to \infty$ with~\eqref{e:LIAL2} and the uniform bound~\eqref{e:UBT1} in $k$, we obtain~$\Phi(u^{(k)}) \to u$  weakly in $\dom{\overline{\E}^{\U, \QP}}$. In particular, $u \in \dom{\overline{\E}^{\U, \QP}}$. 
%
 \end{proof}

\begin{rem} \label{r:TCO} 
Let $\mu$ be the law of either 
\begin{enumerate}[(a)]
\item
 a determinantal point process on $\mathbb R$ with correlation kernel $K$.  Namely, it is the point process whose $n$-point correlation density $f_\QP^{(n)}$ (see Dfn.~\ref{d:DPP}) with respect to the Lebesgue measure is given by 
 $$f_\QP^{(n)}(x_1,\ldots, x_n)=\det\bigl(K(x_i, x_j)_{1 \le i, j \le n}\bigr) \fstop$$
 Let $K_{B_r}(x,y) = K_r(x, y ):=\1_{B_r}(x)K(x, y)\1_{B_r}(y)$ and 
 define $K_{B_r}f(x)=K_rf(x):=\int_{B_r}K_r(x,y)f(y)\diff y$ for $f \in L^2(B_r, \diff x)$.  
Suppose that $K_r$ is trace class in $L^2(B_r, \diff x)$ for every $r>0$.
%

\item a Pfaffian point process  with the real $2\times 2$-matrix correlation kernel $K(x, y)$ having the skew-symmetry $K(x,y)=-K(x,y)^T$. 
 Namely, it is the point process whose $n$-point correlation density $f_\QP^{(n)}$ is given by 
 $$f_\QP^{(n)}(x_1,\ldots, x_n)={\rm Pf}\bigl(K(x_i, x_j)_{1 \le i, j \le n}\bigr) \comma$$
where $\operatorname{Pf}( K(x_i, x_j)_{1 \le i,j \le n})$ is the Pfaffian of the real $2n\times 2n$ skew-symmetric matrix with
the $2\times 2$ block~$K(x_i, x_j )$. Let $K_r(x, y ):=\1_{B_r}(x)K(x, y)\1_{B_r}(y)$ and 
 define $K_rf(x):=\int_{B_r}K_r(x,y)f(y)\diff y$ for $f \in L^2(B_r \to \R^2, \diff x)$.  
Suppose that $K_r$ is trace class in $L^2(B_r \to \R^2, \diff x)$ for every $r>0$.
\end{enumerate}
 Then, Condition~\eqref{e:AFD} holds. 
This includes many point processes, e.g., the Poisson point process (e.g., with constant intensity), $\sine_\beta$ and $\Airy_\beta$ $(\beta=1,2,4)$. 
\end{rem}
\begin{proof}  We discuss the Pfaffian case first.  Let $N_r(\gamma):=\gamma(B_r)$ and $G_r(t):=\int t^{N_r(\gamma)} \diff \QP(\gamma)$ be the generating function.   Noting that 
$$ G_r(\sqrt{2})=\sum_{l=0}^\infty 2^{\frac{l}{2}}\mu\bigl(\gamma \in \U: \gamma(B_r)=l\bigr)\comma$$
it suffices to show $G_r(\sqrt{2})<\infty$. 
Due to the trace-class condition of $K_r$, the Fredholm Pfaffian series is well-posed
\begin{align} \label{f:PFF}
\operatorname{Pf}\big(J + (t-1) K_r\big)
:=
\sum_{n=0}^{\infty} \frac{(t-1)^n}{n!} 
\int_{B_r^{\times n}} 
\operatorname{Pf}\Big(  K_r(x_i, x_j )_{1 \le i,j \le n} \Big)
\, \diff x_1 \cdots \diff x_n \comma
\end{align}
where $J$ is the canonical skew symmetric operator acting on $\mbbR^2$
\[
J = \begin{pmatrix}
0 & -1 \\
1 & 0
\end{pmatrix}
 \fstop
\]
Noting that for a bounded measurable function $f: B_r \to \R$, 
$$\prod_{x \in \gamma}\bigl(1+f(x)\bigr)=\sum_{n=0}^\infty \sum_{\substack{\eta \subset \gamma \\ |\eta|=n}}\prod_{x \in \eta} f(x) \comma$$
we have 
\begin{align}  \label{f:PFF2}
\int_{\U} \prod_{x \in \gamma}\bigl(1+f(x)\bigr) \diff \QP(\gamma) = \sum_{n=0}^{\infty} \frac{1}{n!} 
\int_{\R^{\times n}} 
f_{\QP}(x_1, \ldots, x_k) \prod_{i=1}^n f(x_i)
\, \diff x_1 \cdots \diff x_n \fstop 
\end{align}
Having the equality $t^{N_r(\gamma)}= \prod_{x \in \gamma}\bigl(1+(t-1)\1_{B_r}(x)\bigr)$, plugging $f(x)=(t-1)\1_{B_r}(x)$ into~\eqref{f:PFF2} and comparing it with \eqref{f:PFF} yield  
$$G_r(t) = {\rm Pf}(J + (t-1)K_r) \fstop$$
Using the relation between the Fredholm determinant and the Fredholm Pfaffian
$$ {\rm Pf}(J + (t-1)K_r)^2 =  {\rm det}(I + (t-1)JK_r) \comma$$
the trace class condition leads to
$${\rm Pf}(J + (\sqrt{2}-1)K_r)^2 =  {\rm det}(I + (\sqrt{2}-1)JK_r) \le e^{(\sqrt{2}-1)\|K_r\|_1}<\infty \comma$$
where $\|K_r\|_1$ is the trace norm of $K_r$.
Thus, 
$$ G_r(\sqrt{2}) \le e^{\frac{\sqrt{2}-1}{2}\|K_r\|_1}<\infty \fstop $$

A similar argument also applies to the determinantal case: with $K_r$ trace class, we have
$G_r(t)=\det(I+(t-1)K_r)$ (Fredholm determinant). Applying the  trace-norm bound $|\det(I+(t-1)K_r)|\le e^{(t-1)\|K_r\|_1}$, 
we obtain $G_r(\sqrt2)<\infty$. 
\end{proof}

\section{$\mssW_p$-contraction and log-Harnack inequality}
In this section, we discuss the $p$-Wasserstein contraction property, Wang's dimension-free Harnack inequality, and the $p$-Bakry--\'Emery gradient estimate for the extended metric measure space~$(\U, \tau_\mrmv, \mssd_{\U}, \QP)$. 
Recall that $\pr_{B_r}: \U \to \U(B_r)$ was the projection defined as $\gamma \mapsto \gamma_r:=\gamma_{B_r}$. 
Let $(T_t^{\U, \QP})_{t>0}$ be the $L^2(\U, \QP)$-semigroup associated with the Dirichlet form~$(\E^{\U, \QP}, \dom{\E^{\U, \QP}})$ constructed in Prop.~\ref{p:CES}. Due to~\eqref{e:con1}, $(T_t^{\U, \QP})_{t \ge 0}$ can be extended to the $L^p(\U, \QP)$-semigroup for every $1 \le p \le \infty$, which is denoted by the same symbol. Let $(\mathcal P(\U), \tau_\mrmw)$ be the space of Borel probability measures on~$(\U, \tau_\mrmv)$ endowed with the weak topology $\tau_\mrmw$, i.e., the topology induced by the duality of $C_b(\U, \tau_\mrmv)$. Let $\mcP_\QP(\U)$ be the subset of $\mcP(\U)$ consisting of $\QP$-absolutely continuous elements. 
We recall that 
\begin{itemize}
\item The {\it relative entropy} $\Ent_\QP: \mathcal P_\QP(\U) \to \R\cup\{+\infty\}$ is given by 
 $$\Ent_\QP(\nu):=\int_{\U} \rho \log \rho \diff \QP \comma \quad \nu=\rho\cdot \mu \fstop$$
The domain of $\Ent_\QP$ is denoted by $\dom{\Ent_\QP}:=\{\nu \in \mathcal P_\QP(\U): \Ent_\QP(\nu)<+\infty\}$. 
\item The {\it Fisher information} ${\sf F}_\QP:  \mathcal P_\QP(\U) \to \R\cup\{+\infty\}$ is given by 
 $${\sf F}_\QP(\nu):=4\E^{\U, \QP}(\sqrt{\rho}) \comma \quad \nu=\rho\cdot \mu \fstop$$
The domain of ${\sf F}_\QP$ is denoted by $\dom{{\sf F}_\QP}:=\{\nu \in \mathcal P_\QP(\U): {\sf F}_\QP(\nu)<+\infty\}$.
\end{itemize}
\subsection{The density of the entropy domain}
We define 
\begin{align}\label{e:EP}
\mcP_{e, p}(\U):=\{\nu \in \mcP(\U): \mssW_{p, \mssd_\U}(\nu, \sigma)<+\infty,\ \exists \sigma \in \dom{\Ent_\QP}\} \fstop
\end{align}
We simply write $\mcP_e(\U)$ when $p=2$. 
The following proposition is a generalisation of \cite[Lem.~5.1]{ErbHue15}, which discussed the case of the Poisson measure~$\QP$ and $p=2$.
\begin{prop}\label{p:DE} 
Let $1 \le p<\infty$. Suppose that $\QP$ is a probability measure on $(\U, \msB(\tau_\mrmv))$ satisfying~\ref{ass:CE}. 
Then, 
$$\overline{\dom{\Ent_\QP}}^{\mssW_{p, \mssd_\U}}=\mcP_{e,p}(\U) \comma$$
where the LHS denotes the $\mssW_{p, \mssd_\U}$-completion of $\dom{\Ent_\QP}$.
\end{prop}
 \begin{proof}
 The inclusion $\subset$ is obvious. We prove the other inclusion $\supset$. Let $\sigma \in \mcP_{e, p}(\U)$ and $\nu \in \dom{\Ent_\QP}$ such that $\mssW_{p, \mssd_\U}(\sigma, \nu)<+\infty$. Take an optimal coupling~$q \in {\rm Opt}(\sigma, \nu)$. 
By the Kuratowski--Ryll-Nardzewski measurable selection theorem, we can construct a $\msB(\tau_\mrmv)^{\otimes 2}/\msB(\tau_\mrmv^{\times 2})$-measurable map $\U^{\times 2} \to \U(\R^n \times \R^n)$ assigning $(\gamma, \omega) \in {\rm supp}[q]$ to an optimal matching $\eta$ for $\mssd_\U(\gamma, \omega)<\infty$, see \cite[Lem.~6.1]{ErbHue15}. 
Let $\bar B_r=\{x\in \R^n: |x| \le r\}$. We simply write $B=\bar B_r$ unless the dependency of $r$ is relevant. Define a map $(\gamma, \omega) \mapsto \xi =\xi_r(\gamma, \omega) \in \U$ as 
\begin{align}\label{e:EP0}
\xi:={\rm proj}_1(\1_{B \times B}\eta) + {\rm proj}_2(\1_{B^c \times B^c}\eta) +{\rm proj}_2(\1_{B^c \times B}\eta) + {\rm proj}_2(\1_{B \times B^c}\eta) \comma 
\end{align}
which is $\msB(\tau_\mrmv)^{\otimes 2}/\msB(\tau_\mrmv)$-measurable. By construction, 
\begin{align}\label{e:EP1}
\xi_{B^c}=\omega_{B^c} \cquad \xi(B)=\omega(B) \fstop
\end{align}
Take $\alpha:=\frac{1}{r^{1/p}\xi(B)^{1/2}}$. For $x \in {\rm supp}[\xi_B]$, define 
\begin{align}
\chi(x)=
\begin{cases}
x \quad & \text{if $\inf_{y \in B^c}|x-y| >\alpha$}
\\
x_\alpha & \text{otherwise}\comma
\end{cases}
\end{align}
where $x_\alpha$ is the point on the geodesic from $x$ to $0$ such that $|x-x_\alpha|=\alpha$. Note that $B_\alpha(\chi(x)):=\{y \in \R^n: |y-\chi(x)| <\alpha\} \subset B$. 
For $(\gamma, \omega) \in {\rm supp}[q]$, define $\mcU_{\gamma, \omega}^r \in \mcP(\U)$ as follows: Let $\xi_B=\sum_{i=1}^k \delta_{x_i}$ and $\xi_{B^c}=\sum_{i=k+1}^N \delta_{x_i}$ with $N=N(\omega) \in \N_0 \cup \{+\infty\}$ and $k=k(\omega)$  depending on~$\omega$. Given $(y_1,\ldots, y_k) \in (\R^n)^{\times k}$,  we define 
$$\tilde{\xi}(y_1,\ldots, y_k)=\sum_{i=1}^k \delta_{y_i} + \sum_{i=k+1}^N \delta_{x_i} \in \U \fstop$$
We define
\begin{align} \label{e:EP3}
\mcU_{\gamma, \omega}^{r}:=\int_{\U^k(B)} \delta_{\tilde{\xi}(y_1,\ldots, y_k)} h_B^{{\omega}, k}(y_1,\ldots, y_k)  \prod_{i=1}^k \1_{B_\alpha(\chi(x_i))}(y_i) \diff \mssm_B^{\odot k}(y_1, \ldots , y_k) \comma
\end{align}
where $h_B^{\omega, k}=\frac{\diff \QP_B^{\omega, k}}{\diff \mssm_B^{\odot k}}$  is the Radon--Nikodym density  due to \ref{ass:CE}, $\QP_B^{\omega, k}:=\QP_r^{\omega, k}$ is the projected conditional probability (Dfn.~\ref{d:CP}), and $\mssm_B:=\mssm\mrestr{B}$ is the Lebesgue measure restricted in~$B=\bar B_r$. 
 Using the map $T^{(r)}: (\gamma, \omega) \mapsto \mcU_{\gamma, \omega}^{r}$, we  define  a probability measure in $(\U, \msB(\tau_\mrmv))$ as follows: for $A \in \msB(\tau_\mrmv)$, 
\begin{align}\label{e:DS}
A\mapsto \sigma^{(r)}(A):=\frac{1}{Z_r}\int_{\U^{\times 2}} T^{(r)}(\gamma, \omega)(A)q(\diff \gamma, \diff \omega)  \cquad Z_r:=\int_{\U^{\times 2}} T^{(r)}(\gamma, \omega)(\U)q(\diff \gamma, \diff \omega) \fstop
\end{align}
It suffices to prove the following claims: 
\begin{enumerate}[{\rm Claim} 1]
\item \label{c:EP1} $\mssW_{p, \mssd_\U}(\sigma^{(r)}, \sigma) \to 0$ as $r \to \infty$.
\vspace{2mm}
\item \label{c:EP2} $\Ent_\QP(\sigma^{(r)})<+\infty$ for $r \in \N$.
\end{enumerate}
{\it Proof of \ref{c:EP1}.} 
By construction, for $(\gamma, \omega) \in \supp[q]$ and $\xi=\xi_r(\gamma, \omega)$ as given in \eqref{e:EP0}, we have 
\begin{align} \label{e:TWE}
&\mssW_{p, \mssd_\U}\Bigl(\delta_\xi, \frac{\mcU_{\gamma, \omega}^r}{Z_r}\Bigr)^p\le \bigl(\alpha^2 \xi(B))^{p/2}= \frac{1}{r}
\\
&\mssW_{p, \mssd_\U}(\delta_\gamma, \delta_\xi)^p=\mssd_\U(\gamma, \xi_r(\gamma, \omega))^p \searrow 0 \quad \text{as $r \to \infty$}\fstop
\end{align}
Taking a coupling $Q:=({\rm proj}_1, \frac{T^{(r)}}{Z_r})_\#q$ between $\sigma$ and $\sigma^{(r)}$, using the triangle inequality and $(x+y)^p \le 2^{p-1}(x^p+y^p)$, we have 
\begin{align*}
\mssW_{p, \mssd_\U}(\sigma, \sigma^{(r)})^p &\le \int \mssW_{p, \mssd_\U}\Bigl(\delta_\gamma, \frac{\mcU_{\gamma, \omega}^r}{Z_r}\Bigr)^p \diff q(\gamma, \omega)
\\
& \le \frac{2^{p-1}}{r} + 2^{p-1} \int_{\U^{\times 2}}\mssd_\U(\gamma, \xi_r(\gamma, \eta))^p\diff q(\gamma, \omega) 
\\
&\xrightarrow{r \to \infty}0 \comma
\end{align*}
which completes the proof of \ref{c:EP1}.

{\it Proof of \ref{c:EP2}.}  
Let $q_{\omega}(\diff \gamma)$ be the disintegration of $q(\diff \gamma, \diff \omega)$ with respect to $\nu$ for the variable~$\omega$. 
Let $(\sigma^{(r)})_{B}^\zeta:=(\sigma^{(r)})_{r}^\zeta, \ T^{(r)}(\gamma, \omega)_B^\zeta$ and $\nu_B^\zeta \in \mcP(\U(B))$ be the projected conditional probabilities of $\sigma^{(r)}$, $T^{(r)}(\gamma, \omega)$ and  $\nu$ with conditioning $\zeta$ outside the ball $B=\bar B_r$ respectively.  By construction, $T(\gamma, \omega_B+\omega_{B^c})_B^\zeta =0$ (zero measure) if $\zeta_{B^c} \neq \omega_{B^c}$.
Thus, due to~\eqref{e:DS}, for $A \in \msB(\U(B))$ and $\zeta \in \U$, 
\begin{align} \label{e:FSN}
(\sigma^{(r)})_{B}^{{\zeta}}(A) = \frac{1}{Z_r}\int_{\U\times\U(B)} T^{(r)}(\gamma, \omega_{B}+\zeta_{B^c})_B^{{\zeta}}(A) q_{\omega_B+\zeta_{B^c}}(\diff \gamma) \nu_B^\zeta(\diff \omega_B) \fstop
\end{align}
By construction, we have $\sigma^{(r)}_{B^c}:=({\rm pr}_{B^c})_\# \sigma^{(r)}=\nu_{B^c}$. Thus,  the chain rule of the relative entropy yields
\begin{align*}
\Ent_{\QP}(\sigma^{(r)}) &= \int_{\U} \Ent_{\QP_B^\zeta}((\sigma^{(r)})_{B}^\zeta) \diff \sigma^{(r)}_{B^c}(\zeta) + \Ent_{\QP_{B^c}}(\sigma^{(r)}_{B^c}) 
\\
&= \int_{\U} \Ent_{\QP_B^\zeta}((\sigma^{(r)})_{B}^\zeta) \diff \nu_{B^c}(\zeta) + \Ent_{\QP_{B^c}}(\nu_{B^c}) \fstop
\end{align*}
By the contraction property of the entropy under the push-forward, we have $\Ent_{\QP_{B^c}}(\nu_{B^c}) \le \Ent_\QP(\nu)<\infty$. So, it suffices to show  $\zeta \mapsto \Ent_{\QP_B^\zeta}((\sigma^{(r)})_{B}^\zeta) \in L^1(\nu_{B^c})$.
Applying Jensen's inequality to \eqref{e:FSN}, 
\begin{align*}
 &\Ent_{\QP_B^\zeta}((\sigma^{(r)})_{B}^\zeta) =  \int_{\U} \frac{\diff (\sigma^{(r)})_B^\zeta}{\diff \QP_B^\zeta} \log \frac{\diff (\sigma^{(r)})_B^\zeta}{\diff \QP_B^\zeta} \diff \QP_B^\zeta
 \\
 & \le  \int_{\U \times \U(B)} \Ent_{\QP_B^\zeta} \Bigl(\frac{1}{Z_r}T(\gamma, \omega_B+\zeta_{B^c})_B^\zeta\Bigr) q_{\omega_B+\zeta_{B^c}}(\diff \gamma) \nu_{B}^\zeta(\omega_B) \fstop
 \end{align*}
By construction, $T(\gamma, \omega_B+\zeta_{B^c})_B^\zeta =0$ (zero measure) on $\U^{l}(B)$ if $l \neq k=\omega(B)$. 
 Having this  and further expansion of the integrand yield 
 \begin{align} \label{e:cm}
 & \Ent_{\QP_B^\zeta} \Bigl(\frac{1}{Z_r}T(\gamma, \omega_B+\zeta_{B^c})_B^{\zeta}\Bigr) 
 \\
 &=   \int_{\U^k(B)}  \frac{1}{Z_r}  \prod_{i=1}^k \1_{B_\alpha(\chi(x_i))}(y_i) \log \Bigl(\frac{1}{Z_r}\prod_{i=1}^k \1_{B_\alpha(\chi(x_i))}(y_i) \Bigr) h_B^{\zeta, k}(y_1,\ldots, y_k) \diff \mssm_B^{\odot k}(y_1, \ldots, y_k) \notag
  \\
 &\le   \int_{\U^k(B)}  \Bigl| \frac{1}{Z_r}  \log  \frac{1}{Z_r} \Bigr| h_B^{\zeta, k}(y_1,\ldots, y_k) \diff \mssm_B^{\odot k}(y_1, \ldots, y_k) \notag
   \\
 &\le  \Bigl|\frac{1}{Z_r}  \log  \frac{1}{Z_r}  \Bigr| \int_{\U^k(B)}     h_B^{\zeta, k}(y_1,\ldots, y_k) \diff \mssm_B^{\odot k}(y_1, \ldots, y_k) \notag
    \\
 &= \Bigl|\frac{1}{Z_r}  \log  \frac{1}{Z_r}\Bigr| \QP_B^\zeta(\U^k(B)) \notag \fstop
 \end{align}
 Thus, noting $\sum_{k=0}^\infty \QP_B^\zeta(\U^k(B)) =1$, 
 \begin{align*} \int_{\U} \Ent_{\QP_B^\zeta}((\sigma^{(r)})_{B}^\zeta) \diff \nu_{B^c}(\zeta) 
 &\le  \Bigl| \frac{1}{Z_r}  \log  \frac{1}{Z_r} \Bigr|\int_{\U^{\times 2}}  \sum_{k=0}^\infty \QP_B^\zeta\bigl(\U^k(B)\bigr) q_{\omega_B+\zeta_{B^c}}(\diff \gamma) \nu_{B}^\zeta(\omega_B)  \diff \nu_{B^c}(\zeta_{B^c})
 \\
 &=  \Bigl|\frac{1}{Z_r}  \log  \frac{1}{Z_r}\Bigr|  \comma
 \end{align*}
 which completes the proof. 
 \end{proof}

\subsection{The main assumption}
The main assumption in this section is the following:
 \begin{ass} \label{a:AP}
 Let $\QP$ be a probability measure in~$(\U, \msB(\tau_\mrmv))$ fully supported in~$(\U, \tau_\mrmv)$. 
Suppose that there exists a $\QP$-measurable set~$\Theta \subset \U$ with $\QP(\Theta)=1$ such that 
\begin{enumerate}[(a)]
\item \label{AP-11} $\{\gamma \in \U: \mssd_\U(\gamma, \eta)<+\infty \comma \eta \in \Theta\} = \Theta$; 
\item \label{AP-2} for every $\gamma \in \Theta$ and $t>0$, there exists a Borel probability measure $p^{\U, \QP}_t(\gamma, \diff \eta)$ on $\U$ such that 
\begin{align}\label{e:IKP}
\tilde{T}_t^{\U, \QP}u(\gamma):=\int_{\U}u(\eta) p^{\U, \QP}_t(\gamma, \diff \eta) \cquad u \in \mathcal B_b(\U)
\end{align}
is a $\QP$-representative of $T_t^{\U, \QP}u$;
\item \label{AP-4} there exists a subset~$\{k\} \subset \N$ of infinite cardinality and a probability measure~$\mu^k$ on $\U^k$ for each $k$ such that the metric measure space~$(\U^k, \mssd_{\U^k}, \mu^k)$ is an $\RCD(K,\infty)$ space for every $k$, and  satisfies the following property: for every $\gamma, \zeta \in \Theta$ with $\mssd_\U(\gamma, \zeta)<\infty$, there exist $\gamma^{(k)}, \zeta^{(k)} \in \U^k$ such that, as $k \to \infty$ 
\begin{align*}
&p^{\U^{k}, \QP^{k}}_t(\gamma^{(k)}, \diff \eta) \xrightarrow{\tau_\mrmw} p^{\U, \QP}_t(\gamma, \diff \eta) 
\\
&p^{\U^{k}, \QP^{k}}_t(\zeta^{(k)}, \diff \eta) \xrightarrow{\tau_\mrmw} p^{\U, \QP}_t(\zeta, \diff \eta)  
\\
& \limsup_{k \to \infty} \mssd_\U(\gamma^{(k)}, \zeta^{(k)}) \le \mssd_\U(\gamma, \zeta)  \comma
\end{align*}
for every $t>0$, where $p^{\U^{k}, \QP^{k}}_t(\gamma, \diff \eta)$ is the heat kernel on $(\U^k, \mssd_{\U^k}, \mu^k)$, see~\eqref{d:HK}. 
\end{enumerate}
Note that, for the weak convergence in \ref{AP-4}, we regard $p^{\U^{k}, \QP^{k}}_t(\gamma, \diff \eta)$ as a probability measure in $\U$ through the inclusion~$\U^k \subset \U$. 
We write 
$$\mcT^{\U^k, \QP^k}_t \delta_{\gamma^{(k)}}:=p^{\U^{k}, \QP^{k}}_t(\gamma^{(k)}, \diff \eta) \cquad \mcT^{\U, \QP}_t \delta_{\gamma}:=p^{\U, \QP}_t(\gamma, \diff \eta) \fstop$$
\end{ass}

\begin{prop} Suppose \ref{ass:CE}, ~\ref{ass:ConditionalClos} and  Assumption~\ref{a:AP}.
The semigroup $(\tilde{T}^{\U, \QP}_t)_{t \ge 0}$ can be uniquely extended  to $L^1(\U, \QP)$ for every $t>0$ with the expression 
\begin{align}\label{e:EIK}
\tilde{T}^{\U, \QP}_tu(\gamma) := \int_{\U} \tilde{u}(\eta)p_t^{\U, \QP}(\gamma, \diff \eta) \cquad \gamma \in \tilde{\Theta} \cquad u \in L^1(\U, \QP)\comma
\end{align}
 where  $\tilde{u}$ is any $\msB(\tau_\mrmv)$-representative of $u$ and the integral~\eqref{e:EIK} does not depend on the choice of representatives, and $\tilde{\Theta} \subset \Theta$ is a set of $\QP$-full measure. 
Furthermore, 
\begin{align} \label{e:VER}
\text{$\tilde{T}_t^{\U, \QP}u$ is a $\QP$-representative of $T_t^{\U, \QP}u$ for every $u \in L^1(\U, \QP)$}.
\end{align}
In the rest of this section, therefore, we do not distinguish  $\tilde{T}_t^{\U, \QP}u$ and  ${T}_t^{\U, \QP}u$ for $u \in L^1(\U, \QP)$. 
\end{prop}
\begin{proof}
By \ref{AP-2} in Assumption~\ref{a:AP}, it follows that $\tilde{T}^{\U, \QP}_t$ is $\QP$-symmetric, i.e., 
\begin{align*}
(u, \tilde{T}^{\U, \QP}_t v)_{L^2} = (\tilde{T}^{\U, \QP}_tu,  v)_{L^2} \cquad u, v \in \mathcal B_b(\U) \cquad t>0\fstop
\end{align*}
Thus, by the Cauchy--Schwarz inequality, $(\tilde{T}^{\U, \QP}_t v)^2 \le  (\tilde{T}^{\U, \QP}_t \1)(\tilde{T}^{\U, \QP}_t v^2)$. Using $\tilde{T}^{\U, \QP}_t \1=\1$ due to the conservativeness (Prop.~\ref{p:CES}) and the $\QP$-invariance of $T_t^{\U, \QP}$ in \eqref{e:inv1}, we have 
 $$\int_\U (\tilde{T}^{\U, \QP}_t v)^2 \diff \QP \le  \int_\U v^2 \diff \QP \cquad v \in \mathcal B_b(\U) \cquad t>0\fstop$$
This showing that $\tilde{T}^{\U, \QP}_t$ is $L^2$-contractive, the operator $\tilde{T}^{\U, \QP}_t$ extends to the whole~$L^2(\QP)$.
By~\eqref{e:con1}, $(\tilde{T}^{\U, \QP}_t)_{t \ge 0}$ can be uniquely extended to $L^1(\U, \QP)$ as well. Finally the expression~\eqref{e:EIK} can be proven by the following argument: For $u \in L^1(\U, \QP)$ with $u \ge 0$, take any $\msB(\tau_\mrmv)$-measurable representative $\tilde{u}$ and define $\tilde{u}_c:=\tilde{u} \wedge c$.  
Due to~\ref{AP-2}, there exists $\tilde{\Theta} \subset \Theta$ with $\QP(\tilde{\Theta})=1$ such that
\begin{align*}
T^{\U, \QP}_t u(\gamma) =T^{\U, \QP}_t \tilde{u}(\gamma) = \lim_{c \to \infty}T^{\U, \QP}_t \tilde{u}_c(\gamma)  = \lim_{c \to \infty} \tilde{T}_t^{\U, \QP}\tilde{u}_c(\gamma)  \qquad \gamma \in \tilde{\Theta} \fstop
\end{align*}
By using \eqref{e:IKP} and the Monotone Convergence, we have
\begin{align*}
 \lim_{c \to \infty} \tilde{T}_t^{\U, \QP}\tilde{u}_c(\gamma)
 &=\lim_{c \to \infty} \int_{\U} \tilde{u}_c(\eta)p_t^{\U, \QP}(\gamma, \diff \eta) 
\\
&= \int_{\U} \tilde{u}(\eta)p_t^{\U, \QP}(\gamma, \diff \eta) \quad \gamma \in \tilde{\Theta} \fstop
\end{align*}
For a general $u \in L^1(\U, \QP)$, we decompose $u=u^+-u^-$ and apply the same argument to each $u^+$ and $u^-$, which proves  \eqref{e:VER}.
\end{proof}
\subsection{Wasserstein contraction} 
{Let $\Theta$ be the set in Assumption~\ref{a:AP} and 
$$\mcP_\Theta=\{\nu \in \mcP(\U): \nu(\Theta)=1\} \fstop$$
Let $\mathcal P_\QP(\U)$ be the subspace of~$\mathcal P(\U)$ consisting of probability measures absolutely continuous with respect to $\QP$. We write $\nu=\rho \cdot \mu$ for $\nu \in \mcP_\QP(\U)$ if $\rho=\frac{\diff \nu}{\diff \QP}$. 
Note that
 $$\mcP_\QP(\U) \subset \mcP_\Theta \quad \text{as well as}\quad \delta_{\gamma} \in \mcP_\Theta \quad \gamma \in \Theta \fstop$$
 For $\nu \in \mcP_\Theta$,  the dual semigroup ${\mathcal T}^{\U, \QP}_t \nu$ is defined as 
   \begin{align} \label{d:DS2}
 {\mathcal T}^{\U, \QP}_t \nu(A) := \int_{\Theta} \mathcal{T}_t^{\U, \QP}\delta_{\gamma}(A) \diff \nu(\gamma) \cquad A \in \msB(\U)  \fstop
 \end{align}
 If $\nu=\rho\cdot \QP \in \mcP_\QP(\U)$, it is easy to see
 \begin{align} \label{d:DS}
 {\mathcal T}^{\U, \QP}_t \nu= \bigl( {T}_t^{\U, \QP}\rho) \cdot \QP \cquad \nu=\rho\cdot\QP \in \mcP_\QP(\U)  \fstop
 \end{align}
 Indeed,  if $\rho \in L^2(\U, \QP)$, by the $L^2$-symmetry of $T_t^{\U, \QP}$ and $\QP(\Theta)=1$ we have 
    \begin{align*}
 {\mathcal T}^{\U, \QP}_t \nu(A) 
 &= \int_{\U} \mathcal{T}_t^{\U, \QP}\delta_{\gamma}(A) \rho(\gamma) \diff \QP(\gamma) 
 =  \int_{\U} {T}_t^{\U, \QP}\1_A(\gamma) \rho(\gamma) \diff \QP(\gamma) 
 \\
 &= \int_{\U} \1_A(\gamma) {T}_t^{\U, \QP} \rho(\gamma) \diff \QP(\gamma) = \bigl(\bigl( {T}_t^{\U, \QP}\rho) \cdot \QP\bigr)(A) \fstop
 \end{align*}
 The general case $\rho \in L^1(\U, \QP)$ follows by a standard truncation argument by constant cut-off with the $L^1$-contraction $\|T_t^{\U, \QP} \rho\|_{L^1} \le \|\rho\|_{L^1}$ due to~\eqref{e:con1}.
  Thanks to the conservativeness (mass-preservation) in Prop.~\ref{p:CES}, we have 
 $T_t^{\U, \QP}\1=\1$ for every $t>0$, hence, the dual semigroup maps $\mathcal T^{\U, \QP}_t: \mcP_\Theta \to \mcP(\U)$ to probability measures for every $t>0$.

\begin{thm}[Wasserstein contraction] \label{t:WC}
Suppose \ref{ass:CE}, ~\ref{ass:ConditionalClos} and  Assumption~\ref{a:AP}.  Then, for $1 \le p <\infty$, 
\begin{align}\label{e:WC}
\mssW_{p, \mssd_\U}\bigl(\mathcal T_t^{\U, \QP}\nu, \mathcal T_t^{\U, \QP}\sigma\bigr) \le e^{-Kt}\mssW_{p, \mssd_\U}(\nu, \sigma)  \cquad  t>0 \cquad  \nu, \sigma \in \mcP_\Theta \fstop
\end{align}
In particular, $\mcT_t^{\U, \QP}$ can be extended to the $\mssW_{p, \mssd_\U}$-completion and 
\begin{align}\label{e:WC1}
\mssW_{p, \mssd_\U}\bigl(\mathcal T_t^{\U, \QP}\nu, \mathcal T_t^{\U, \QP}\sigma\bigr) \le e^{-Kt}\mssW_{p, \mssd_\U}(\nu, \sigma)  \cquad  t>0 \cquad  \nu, \sigma \in \overline{\mcP_\Theta}^{\mssW_{p, \mssd_\U}} \fstop
\end{align}
\end{thm}


\begin{proof}
Take $\gamma, \eta \in \Theta$ with $\mssd_\U(\gamma, \eta)<\infty$ and $\gamma^{(k)}, \eta^{(k)} \in \U^k$ witnessed by \ref{AP-4} in Assumption~\ref{a:AP}.
Due to~Cor.~\ref{p:EMW1}, $\mssW_{p, \mssd_\U}$ is $\tau_\mrmv^{\times 2}$-lower semi-continuous. Thus,   
\begin{align}\label{WCP}
\mssW_{p, \mssd_\U}\bigl(\mathcal T_t^{\U, \QP}\delta_\gamma, \mathcal T_t^{\U, \QP}\delta_{\eta}\bigr) 
&\le \liminf_{k \to \infty} \mssW_{p, \mssd_\U}\bigl(\mathcal T_t^{\U^k, \QP^k}\delta_{\gamma^{(k)}}, \mathcal T_t^{\U^k, \QP^k}\delta_{\eta^{(k)}}\bigr) 
\\
&=\liminf_{k \to \infty} \mssW_{p, \mssd_{\U^k}}\bigl(\mathcal T_t^{\U^k, \QP^k}\delta_{\gamma^{(k)}}, \mathcal T_t^{\U^k, \QP^k}\delta_{\eta^{(k)}}\bigr) \notag
\\
&\le  e^{-Kt} \liminf_{k \to \infty} \mssd_{\U}\bigl({\gamma^{(k)}}, {\eta^{(k)}}\bigr) \notag
\\
&\le  e^{-Kt}\mssd_{\U}\bigl({\gamma}, {\eta}\bigr) \cquad \gamma, \eta \in \Theta \comma \notag
\end{align}
where the second line follows by the fact that the measure $\mathcal T_t^{\U^k, \QP^k}\delta_{\gamma_r}$ and $\mathcal T_t^{\U^k, \QP^k}\delta_{\eta_r}$ are supported in $\U^k$ due to the conservativeness~\eqref{d:MP}, thus, $\mssW_{p, \mssd_\U}$ coincides with $\mssW_{p, \mssd_{\U^k}}$ for these measures; the third line follows by the $p$-Wasserstein contraction for the heat flow on $(\U^k, \mssd_{\U^k}, \QP^k)$ due to~\ref{AP-4} in Assumption~\ref{a:AP} (and see \cite[Thm.~4.4]{Sav14} for the implication from the $\RCD$ condition to the $p$-Wasserstein contraction).
Let $\mathbf c$ be an optimal coupling of $\nu$ and $\sigma$, and $\mathbf c_{\gamma, \eta}$ be an optimal coupling for $\mathcal T_t^{\U, \QP}\delta_\gamma$ and $\mathcal T_t^{\U, \QP}\delta_{\eta}$. 
Set $\mathbf c_0:=\int_{\Theta \times \Theta} \mathbf c_{\gamma, \eta} \diff \mathbf c(\gamma, \eta) \in \mathcal P(\U\times \U)$, which is an admissible coupling of $\mathcal T_t^{\U, \QP}\nu$ and $\mathcal T_t^{\U, \QP}\sigma$, i.e., the marginals of $\mathbf c_0$ are $\mathcal T_t^{\U, \QP}\nu$ and $\mathcal T_t^{\U, \QP}\sigma$. The inequality~\eqref{WCP} leads 
\begin{align*}
\mssW_{p, \mssd_\U}\bigl(\mathcal T_t^{\U, \QP}\nu, \mathcal T_t^{\U, \QP}\sigma\bigr)^p
&\le \int_{\U \times \U} \mssd_{\U}(\gamma, \eta)^p \diff \mathbf c_0(\gamma, \eta) 
\\
&= \int_{\Theta \times \Theta} \mssW_{p, \mssd_\U}\bigl(\mathcal T_t^{\U, \QP}\delta_\gamma, \mathcal T_t^{\U, \QP}\delta_{\eta}\bigr)^p\diff \mathbf c(\gamma, \eta)
\\
& \le e^{-pKt}\ \int_{\Theta \times \Theta} \mssd_{\U}\bigl({\gamma}, {\eta}\bigr)^p \diff \mathbf c(\gamma, \eta)
=e^{-pKt} \mssW_{p, \mssd_\U}(\nu, \sigma)^p \fstop 
\end{align*}
The latter statement follows as the contraction property extends to the completion.
\end{proof}
}

\subsection{log-Harnack inequality}

\begin{thm}\label{t:DFH}
Suppose \ref{ass:CE}, ~\ref{ass:ConditionalClos} and  Assumption~\ref{a:AP}.  Then
\begin{enumerate}[{\rm (a)}]
\item \label{DFH-1} $(${\bf  log-Harnack inequality}$)$ for every non-negative $u \in L^1(\U, \QP)$ and $t>0$, there exists $\Omega \subset \U$ with $\QP(\Omega)=1$ such that
\begin{align} \label{i:LH1}
T^{\U, \QP}_t\log u(\gamma) \le \log (T^{\U, \QP}_tu(\eta)) + \frac{{\mssd}_\U(\gamma, \eta)^2}{4I_{2K}(t)}\comma \quad \text{$\gamma, \eta \in \Omega$} \comma
\end{align}
where $I_K(t) :=\int_0^t e^{Kr}\diff r=\frac{e^{Kt}-1}{K}$. 
Furthermore, for every non-negative $u \in \mathcal B_b(\U)$ and $t>0$, 
\begin{align} \label{i:LH2}
T^{\U, \QP}_t\log u(\gamma) \le \log (T^{\U, \QP}_tu(\eta)) + \frac{{\mssd}_\U(\gamma, \eta)^2}{4I_{2K}(t)}\comma \quad \text{$\gamma, \eta \in \Theta$} \comma
\end{align}
where $\Theta$ was given in Assumption~\ref{a:AP};
\item  \label{DFH-2}  $(${\bf  dimension-free Harnack inequality}$)$ for every non-negative $u \in L^\infty(\U, \QP)$, $t>0$ and $\alpha>1$, there exists $\Omega \subset \U$ with $\QP(\Omega)=1$ such that 
$$(T^{\U, \QP}_tu)^\alpha(\gamma)\le T^{\U, \QP}_tu^\alpha(\eta) \exp\Bigl\{ \frac{\alpha}{4(\alpha-1)I_{2K}}{\mssd}_\U(\gamma, \eta)^2\Bigr\} \comma \quad \gamma, \eta \in \Omega \fstop$$
\end{enumerate}
\end{thm}
\begin{proof}
\ref{DFH-1} We first prove the following $\e$-truncated version:
\begin{align} \label{i:LH11}
T^{\U, \QP}_t\log (u+\e)(\gamma) \le \log (T^{\U, \QP}_tu(\eta)+\e) + \frac{{\mssd}_\U(\gamma, \eta)^2}{4I_{2K}(t)} \cquad \e>0 \fstop
\end{align}
Take $\gamma, \eta \in \Theta$ with $\mssd_\U(\gamma, \eta)<\infty$ and $\gamma^{(k)}, \eta^{(k)} \in \U^k$ witnessed by \ref{AP-4} in Assumption~\ref{a:AP}.
By the $\RCD$ condition~in~\ref{AP-4}  in Assumption~\ref{a:AP}, we have
\begin{align*}
T^{\U^k, \QP^k}_t\log (u+\e)(\gamma^{(k)}) \le \log (T^{\U^k, \QP^k}_tu(\eta^{(k)})+\e) + \frac{{\mssd}_{\U^k}(\gamma^{(k)}, \eta^{(k)})^2}{4I_{2K}(t)}\comma \quad u \in \Cb(\U) \ ;
\end{align*}
for every $\e>0$  (see \cite[Lem.~4.6]{AmbGigSav15}), where we regard $u \in \Cb(\U)$ as an element in $\Cb(\U^k)$ by the restriction $u|_{\U^k}$. Furthermore, by~\ref{AP-4}  in Assumption~\ref{a:AP}, for every $u \in \Cb(\U)$
\begin{align*}
&T^{\U^k, \QP^k}_t\log (u+\e)(\gamma^{(k)}) \xrightarrow{k \to \infty} T^{\U, \QP}_t\log (u+\e)(\gamma) \comma
\\
& \log (T^{\U^k, \QP^k}_tu(\eta^{(k)})+\e) \xrightarrow{k \to \infty} \log (T^{\U, \QP}_tu(\eta)+\e) \comma
\\
& \limsup_{k \to \infty} \mssd_\U(\gamma^{(k)}, \eta^{(k)}) \le  \mssd_\U(\gamma, \eta)  \fstop
\end{align*}
Thus, the sought inequality \eqref{i:LH11} was proven  for $u \in \Cb(\U)$ and $\Omega=\Theta$. The case for $u \in L^1(\U, \QP)$ in \eqref{i:LH1}  can be readily proven by the density of $\Cb(\U)$ in $L^1(\U, \QP)$ and the $L^1(\U, \QP)$-contraction property~\eqref{e:con1} of $T^{\U, \QP}_t$.  The case $u \in \mathcal B_b(\U)$ in \eqref{i:LH2} follows by the expression~\eqref{e:IKP} and the fact that $\mathcal B_b(\U)$ can be approximated by $\Cb(\U)$ in $L^1\bigl(\U, p^{\U, \QP}_t(\gamma, \diff \eta)\bigr)$ for every $\gamma \in \Theta$ and $t>0$ because $p^{\U, \QP}_t(\gamma, \diff \eta)$ is a Borel probability measure on $\U$, $\mcC_b(\U)$ is dense in $L^1\bigl(\U, p^{\U, \QP}_t(\gamma, \diff \eta)\bigr)$ and $\mcB_b(\U) \subset L^1\bigl(\U, p^{\U, \QP}_t(\gamma, \diff \eta)\bigr)$. Finally, we get the case $\e=0$ by passing to the limit $\e \downarrow 0$ with the integral expression~\eqref{e:IKP} for \eqref{i:LH2} and the expression~\eqref{e:EIK} for \eqref{i:LH1} respectively.

\ref{DFH-2}  Due to~\cite[Thm.~3.1]{Li15}, the dimension-free Harnack inequality holds for $\RCD(K,\infty)$ spaces. Thus, for each $k$, it holds for $(\U^k, \mssd_{\U^k}, \QP^k)$ with the semigroup~$T_t^{\U^k, \QP^k}$. Letting $k \to \infty $ by a similar argument to~\ref{DFH-1}, the proof is complete. 
\end{proof}

As a corollary of Thm.~\ref{t:DFH}, we have the following regularisation estimate. The proof is an adaptation of \cite[Lem.~5.3, 5.4]{ErbHue15}, which discussed the case of the Poisson measure~$\QP$. 
\begin{cor}[$L\log L$-regularisation] \ \label{c:LLL}
\begin{enumerate}[(a)]
\item \label{c:LLL1} For every~$\nu \in \mathcal P_e(\U)$  and $\sigma \in \dom{\Ent_{\QP}}$,  we have $\mcT_t^{\U, \QP}\nu \in \dom{\Ent_\QP}$ and 
\begin{align} \label{e:LLOL}
\Ent_{\QP}(\mathcal T_t^{\U, \QP}\nu) \le \Ent_{\QP}(\sigma) +\frac{1}{4I_{2K}(t)}{\mssW}_{2, \mssd_{\U}}(\nu, \sigma)^2 \cquad t>0\fstop
\end{align} 
\item \label{c:LLL2} For every~$\nu \in \mathcal P_e(\U)$ and $t>0$, 
\begin{align} \label{c:LLL3} 
\mssW_{2, \mssd_\U}(\mcT_t^{\U, \QP}\nu, \nu)<+\infty \cquad \mssW_{2, \mssd_\U}(\mcT_t^{\U, \QP}\nu, \nu) \xrightarrow{t \to 0}0 \fstop
\end{align}
\end{enumerate}
\end{cor}
\begin{proof} 
\ref{c:LLL1} We first prove the case $\nu=f_1 \cdot \mu \in \dom{\Ent_\QP}$ and $\sigma=f_2\cdot\mu \in \dom{\Ent_\QP}$. We may assume $\mssW_{2, \mssd_\U}(\nu, \sigma)<\infty$, otherwise there is nothing to prove. 
Applying \eqref{i:LH1} with $u=T_t^{\U, \QP}f_1$ and then integrating against an optimal coupling $q \in {\rm Opt}(\nu, \sigma)$, we have 
\begin{align*}
\Ent_\QP(T_t^{\U, \QP} \nu)&=\int_\U T^{\U, \QP}_tf_1\log T_t^{\U, \QP}f_1(\gamma)\diff \mu(\gamma)
\\
&= \int_\U \bigl(T^{\U, \QP}_t\log T_t^{\U, \QP}f_1(\gamma)\bigr) f_1\diff \mu(\gamma) 
\\
&\le \int_\U  \log (T^{\U, \QP}_{2t}f_1(\eta)) \diff \sigma(\eta)+ \frac{\mssW_{2,{\mssd}_\U}(\nu,  \sigma)^2}{4I_{2K}(t)}\comma
\end{align*}
where the second equality follows by the $L^2$-symmetry of $T_t^{\U, \QP}$. By Jensen's inequality and the fact that $\int_\U T_{2t}^{\U, \QP}f_1\diff \QP=1$, 
\begin{align*}
\int_\U \log T_{2t}^{\U, \QP}f_1 \diff \sigma &= \int_\U \log \frac{T_{2t}^{\U, \QP}f_1}{f_2} \diff \sigma + \int_\U \log f_2 \diff \sigma
\\
& \le \log\biggl( \int_\U \frac{T_{2t}^{\U, \QP}f_1}{f_2}  \diff \sigma \biggr) + \int_\U f_2\log f_2 \diff \mu
\\
& =\log\biggl( \int_\U T_{2t}^{\U, \QP}f_1 \diff \mu \biggr) + \Ent_\QP(\sigma)= \Ent_\QP(\sigma) \comma
\end{align*}
which completes the proof.

For the general case $\nu \in \mcP_e(\U)$, thanks to Prop.~\ref{p:DE}, we can take an approximation $\nu^{\e} \in \dom{\Ent_\QP}$ converging to $\nu$ in $\mssW_{2, \mssd_\U}$ as $\e \to 0$. Let $\nu_t^\e:=\mcT_t^{\U, \QP}\nu^\e \in \dom{\Ent_\QP}$. 
Then, as we proved above, we have 
\begin{align*}
\Ent_{\QP}(\nu_t^\e) \le \Ent_{\QP}(\sigma) +\frac{1}{4I_{2K}(t)}{\mssW}_{2, \mssd_{\U}}(\nu^\e, \sigma)^2 \cquad t>0 \quad \sigma \in \dom{\Ent_\QP} \fstop
\end{align*}
As $\nu \in \mcP_e(\U)$, we can take $\sigma\in \dom{\Ent_\QP}$ such that $\mssW_{2, \mssd_\U}(\nu, \sigma)<\infty$. As $\mssW_{2, \mssd_\U}(\nu^\e, \nu) \to 0$, we have that $\lim_{\e \to 0} \mssW_{2, \mssd_\U}(\nu^\e, \sigma)^2=\mssW_{2, \mssd_\U}(\nu, \sigma)^2<+\infty$. 
Furthermore, by the Wasserstein contraction~\eqref{e:WC} and the fact that the $\mssW_{2, \mssd_\U}$-topology is stronger than the weak topology~$\tau_\mrmw$ due to Cor.~\ref{p:EMW1} and \ref{r:REW3} in Rem.~\ref{r:REW}, we have $\nu_t^\e \xrightarrow{\e \to 0} \nu_t$ weakly. By the $\tau_\mrmw$-lower semi-continuity of $\Ent_\QP$, 
\begin{align*}
\Ent_\QP(\nu_t) & \le \liminf_{\e \to 0} \Ent_\QP(\nu_t^\e)  \le \Ent_{\QP}(\sigma) +\frac{1}{4I_{2K}(t)}{\mssW}_{2, \mssd_{\U}}(\nu, \sigma)^2<+\infty \fstop
\end{align*}
Thus, $\nu_t = \mcT_t^{\U, \QP}\nu \in \dom{\Ent_\QP}$ and we obtain the inequality~\eqref{e:LLOL}.

\ref{c:LLL2} We first prove the case $\nu=f\cdot \mu \in \dom{\Ent_\QP}$. Let $\nu_t:=\mcT_t^{\U, \QP}\nu$. By~\eqref{e:IKP}, Jensen's inequality and the invariance $\int_\U T_t^{\U, \QP}f \diff \QP=\int_\U f \diff \QP$ (see \eqref{e:inv1}), the dual semigroup~$\mcT_t^{\U, \QP}$ contracts $\Ent_\QP$:
\begin{align} \label{e:CEH}
\Ent_\QP(\nu_t) &\le \int_\U \int_\U f(\eta) \log f(\eta) p_t^{\U, \QP}(\gamma, \diff \eta) \diff \QP =   \int_\U T_t^{\U, \QP} (f\log f) \diff \QP 
\\
&= \int_\U f \log f \diff \QP = \Ent_\QP(\nu) \fstop \notag
\end{align}
Due to~\cite[Lem.~6.1]{AmbGigSav14} and Thm.~\ref{t:E=Ch}, we have $|\dot{\nu}_t|^2 \le \Fis_\QP(\nu_t)$, where $|\dot{\nu}_t|$ is the metric speed of $t \mapsto \nu_t \in \mcP(\U)$ with respect to $\mssW_{2, \mssd_\U}$, see~\eqref{d:MD}. Thus, 
\begin{align} \label{e:WIE}
\mssW_{2, \mssd_\U}(\nu_t, \nu) \le \int_0^t |\dot{\nu}_s| \diff s \le \sqrt{t} \biggl(\int_0^t \Fis_\QP(\nu_s) \diff s\biggr)^{1/2} \le 2\sqrt{t} \Ent_\QP(\nu)^{1/2} \comma
\end{align}
where the last inequality follows by e.g.,~\cite[Lem.~5.2]{ErbHue15}. 
Thus, $\mssW_{2, \mssd_\U}(\nu_t, \nu) \xrightarrow{t \to 0}0$.

For the general case $\nu \in \mcP_e(\U)$, take an approximation $\nu^{\e} \in \dom{\Ent_\QP}$ converging to $\nu$ in $\mssW_{2, \mssd_\U}$ as $\e \to 0$. Let $\nu_t^\e:=\mcT_t^{\U, \QP}\nu^\e \in \dom{\Ent_\QP}$. 
By the Wasserstein contraction~\eqref{e:WC}, 
\begin{align*}
\mssW_{2, \mssd_\U}(\nu_t, \nu) &\le \mssW_{2, \mssd_\U}(\nu_t, \nu_t^\e)+\mssW_{2, \mssd_\U}(\nu_t^\e, \nu^\e)+ \mssW_{2, \mssd_\U}(\nu^\e, \nu)
\\
& \le (1+e^{-Kt})\mssW_{2, \mssd_\U}(\nu, \nu^\e)+\mssW_{2, \mssd_\U}(\nu_t^\e, \nu^\e) \fstop
\end{align*}
Taking $\e$ small enough to make the first term small, and then taking $t \to 0$ with the estimate~\eqref{e:WIE}, the RHS can be arbitrarily small, which concludes the statement. 
\end{proof}

{
\begin{rem}[Finiteness of  $\mssW_{p, \mssd_\U}$] Since $\mssW_{p, \mssd_\U}$ is an extended metric, it is important to know when $\mssW_{p, \mssd_\U}$ is finite. As a consequence of our results, 
 we have the finiteness of $\mssW_{p, \mssd_\U}$ along the dual semigroup in the following cases:
\begin{enumerate}[(a)] 
\item By~Cor.~\ref{c:LLL}, for every $\nu \in \mcP_e(\U)$; 
$$\mssW_{2, \mssd_\U}(\mcT_t^{\U, \QP}\nu, \nu)<+\infty \quad \ t \ge 0 \fstop$$
\item  Let $1 \le p <\infty$ and let $\Theta$ be  the set~in Assumption~\ref{a:AP}. By the Wasserstein contraction~\eqref{e:WC1}, for $\nu, \sigma \in \overline{\mcP_\Theta}^{\mssW_{p, \mssd_\U}}$, 
\begin{align*}
\mssW_{p, \mssd_\U}\bigl(\mathcal T_t^{\U, \QP}\nu, \mathcal T_t^{\U, \QP}\sigma\bigr) \le e^{-Kt}\mssW_{p, \mssd_\U}(\nu, \sigma) <+\infty  \cquad  t>0 \fstop
\end{align*}
\end{enumerate}
\end{rem}
}
\begin{cor}[Lipschitz regularity of the semigroup]\label{c:LIP}
 Suppose \ref{ass:CE}, ~\ref{ass:ConditionalClos} and  Assumption~\ref{a:AP}.  
  Then,  for $u \in \Lip_b({\mssd}_\U, \QP)$ and $t>0$, 
$T_t^{\U, \QP}u$ has a ${\mssd}_\U$-Lipschitz $\QP$-representative~$\widehat{T_t^{\U, \QP}u}$
and the following estimate holds:
$$\Lip_{{\mssd}_\U}(\widehat{T_t^{\U, \QP} u}) \le e^{-Kt}\Lip_{{\mssd}_\U}(u) \fstop$$
\end{cor}
\begin{proof}
By \eqref{WCP} and the expression $T^{\U, \QP}_t u(\gamma)=\int_{\U} u \diff \mathcal T_t^{\U, \QP}\delta_{\gamma}$ for $\gamma \in \Theta$, 
\begin{align*}
|T^{\U, \QP}_t u(\gamma)- T^{\U, \QP}_t u(\eta)| \le \Lip_{\mssd_\U}(u)\mssW_{2, \mssd_\U}( \mathcal T_t^{\U, \QP}\delta_{\gamma},  \mathcal T_t^{\U, \QP}\delta_{\eta}) \le e^{-Kt}\Lip_{\mssd_\U}(u)\mssd_\U(\gamma, \eta)  \fstop
\end{align*}
Noting $\QP(\Theta)=1$, the McShane extension (for extended metric, see~\cite[Lem.~2.1]{LzDSSuz20}) provides the sought $\mssd_\U$-Lipschitz $\QP$-representative. 
%
%
%
\end{proof}

\subsection{$p$-Bakry--\'Emery gradient estimate}
\begin{thm}[$p$-Bakry-\'Emery estimate] \label{p:PBE}
Suppose \ref{ass:CE}, ~\ref{ass:ConditionalClos} and  Assumption~\ref{a:AP}.  
 The form $(\E^{\U, \QP}, \dom{\E^{\U, \QP}})$ satisfies $\BE_p(K,\infty)$ for every $1 \le p <\infty$:
 $$\cdc^{\U}(T_t^{\U, \QP}u)^{\frac{p}{2}} \le e^{-pKt}T_t^{\U, \QP}\bigl(\cdc^{\U}(u)^{\frac{p}{2}}\bigr) \quad \text{$\QP$-a.e.~} \quad u \in \dom{\E^{\U, \QP}} \fstop$$
 \end{thm}
 \begin{proof}
 Let $u \in \Lip_b(\mssd_\U, \QP)$, $\gamma_0, \gamma_1 \in \Theta$, and $(\gamma_s)_{s \in [0,1]} \in \AC^2([0,1], \U, \mssd_\U)$ be any $2$-absolutely continuous curve connecting $\gamma_0$ and $\gamma_1$, see \eqref{d:ABC} for the definition.  Due to \ref{AP-11} in Assumption~\ref{a:AP}, $(\gamma_s)_{s \in [0,1]} \subset \Theta$, therefore, $\nu_s:=\mathcal T_t^{\U, \QP}\delta_{\gamma_s}$ is well-defined for every $s \in [0,1]$.  Thanks to the contraction \eqref{WCP}, we have $(\nu_s)_{s \in [0,1]} \subset \AC^2([0,1], \mathcal P(\U), \mssW_{2, \mssd_\U})$. Then, by \cite[Thm.~3.1]{Lis16}, there exists $\sigma \in \mathcal P\bigl(\mathcal C([0,1], \U, \tau_\mrmv) \bigr)$ such that 
 \begin{enumerate}[(a)]
  \item $\sigma$ is concentrated in $\AC^2([0,1], \U, \mssd_\U)$;
  \item $(e_s)_\# \sigma = \nu_s$ for every $s \in [0,1]$, where $e_s: \mathcal C([0,1], \U, \tau_\mrmv) \to \U$ is $w \mapsto e_s(w)=w_s$;
  \item for $\sigma$-a.e.~$w \in \mathcal C([0,1], \U, \tau_\mrmv)$, the metric derivative $|\dot{w}|_s$ with respect to $\mssd_\U$ exists for a.e.~$s \in [0,1]$ and 
  $$|\dot{\nu}|_s= \| |\dot{w}|_s\|_{L^2(\mathcal C([0,1], \U), \sigma)} \quad \text{a.e.~$s \in [0,1]$} \comma$$
  where $|\dot{\nu}|_s$ is the metric derivative with respect to $\mssW_{2, \mssd_\U}$.
  \end{enumerate}
 Thus, using the upper gradient property of $|\mathsf D_{\mssd_\U} u|$, we have  
  \begin{align*}
 |T_t^{\U, \QP}u(\gamma_1)-T_t^{\U, \QP}u(\gamma_0)| 
 &= \biggl|\int_\U u \diff \mathcal T_t^{\U, \QP}\delta_{\gamma_1}-\int_\U u \diff \mathcal T_t^{\U, \QP}\delta_{\gamma_0}\biggl| 
 \\
 &= \biggl|\int_\U (u \circ e_1(w)- u\circ e_0(w)) \diff \sigma(w)\biggl|
   \\
  &\le \int_\U \biggl(\int_0^1|\mathsf D_{\mssd_\U} u|(w_s)|\dot{w}|_s \diff s\biggr) \diff \sigma(w) 
  \\
  &= \int_0^1\biggl(\int_\U|\mathsf D_{\mssd_\U} u|(w_s)|\dot{w}|_s \diff \sigma(w)\biggr) \diff s
      \\
  &\le \int_0^1\biggl(\int_\U|\mathsf D_{\mssd_\U} u|(w_s)^2  \diff \sigma(w)\biggr)^{1/2} \biggl(\int_\U|\dot{w}|_s^2 \diff \sigma(w)\biggr)^{1/2} \diff s
        \\
  &= \int_0^1\biggl(\int_\U|\mathsf D_{\mssd_\U} u|^2  \diff \nu_s\biggr)^{1/2} |\dot{\nu}|_s \diff s \fstop
 \end{align*}
 Combined with the contraction \eqref{WCP} and the definition of the metric derivative~\eqref{d:MD}, it holds
 \begin{align*}
 |T_t^{\U, \QP}u(\gamma_1)-T_t^{\U, \QP}u(\gamma_0)| 
 & \le  \int_0^1 \biggl(\int_\U |\mathsf D_{\mssd_\U} u|^2 \diff \nu_s\biggr)^{1/2} |\dot{\nu}_s|\diff s
 \\
& \le  e^{-Kt}\int_0^1 \biggl(\int_\U |\mathsf D_{\mssd_\U} u|^2 \diff \nu_s\biggr)^{1/2} |\dot{\gamma}_s|\diff s
   \\
 & =  e^{-Kt}\int_0^1 \bigl(T_t^{\U, \QP}|\mathsf D_{\mssd_\U} u|^2(\gamma_s)\bigr)^{1/2} |\dot{\gamma}_s|\diff s \fstop
 \end{align*}
In view of \ref{AP-11} in Assumption~\ref{a:AP} and the fact that $(\U, \mssd_\U)$ is a length space due to Prop.~\ref{p:CEEM},  $\Theta$ is a length space as well with respect to the extended distance~$(\mssd_\U)|_{\Theta\times \Theta}$ restricted on $\Theta \times \Theta$. Thus, by taking the infimum of $\int_0^1|\dot{\gamma}|_s \diff s$ over all $(\gamma_s)_{s \in [0,1]} \subset \Theta$ achieving $\mssd_\U(\gamma_0, \gamma_1)$, we obtain 
 \begin{align} \label{i:BEI}
 & |T_t^{\U, \QP}u(\gamma_1)-T_t^{\U, \QP}u(\gamma_0)|
  \\
  & \le e^{-Kt}\mssd_\U(\gamma_0, \gamma_1)\sup\Big\{\bigl(T_t^{\U, \QP}|\mathsf D_{\mssd_\U}  u|^2\bigr)^{1/2}(\zeta): \mssd_\U(\zeta, \gamma_0) \le 2\mssd_\U(\gamma_0, \gamma_1), \zeta \in \Theta\Big\} \fstop \notag
 \end{align}
Thanks to the log-Harnack inequality~\eqref{i:LH2} in Thm.~\ref{t:DFH} and \cite[Prop.~3.1]{WanYua11}, $T_t^{\U, \QP}$ has the $\mathcal B_b(\U)$-to-$\Cb(\Theta, \tau_{\mssd_\U})$-regularisation. 
As $|\mathsf D_{\mssd_\U} u| \le \Lip_{\mssd_\U}(u)<\infty$ by~\eqref{e:SLC}, this particularly implies that the function $T_t^{\U, \QP}|\mathsf D_{\mssd_\U} u|^2$  is  bounded and $\tau_{\mssd_\U}$-continuous on $\Theta$.  Combined with~\eqref{i:BEI}, it holds that
  \begin{align*} 
  |\mathsf D_{\mssd_\U} T_t^{\U, \QP}u|^2 \le e^{-2Kt}   T_t^{\U, \QP} |\mathsf D_{\mssd_\U}  u|^2 \quad \text{on} \quad \Theta\fstop
 \end{align*}
To conclude the sought statement for $u \in \dom{\E^{\U, \QP}}$, we just need to take a sequence $u_n \in \Lip_b(\mssd_\U, \QP)$ converging to $u$ in $\dom{\E^{\U, \QP}}$ such that $|\mathsf D_{\mssd_\U} u_n| \to \cdc^{\U, \QP}(u)^{1/2}$ strongly in $L^2(\QP)$. This follows by the identification of $\E^{\U, \QP}$ with the Cheeger energy $\Ch^{\mssd_\U, \QP}$ due to Thm.~\ref{t:E=Ch} and a general fact \cite[(c) Lem.~4.3]{AmbGigSav14} regarding Cheeger energies. 

 The case of $p=1$ follows from the case of $p=2$ and the $\tau_\mrmv$-quasi-regularity in Cor.~\ref{c:QR}, combined with the self-improvement result~\cite[Cor.~3.5]{Sav14}.
The case of $p >1$ follows by the case of $p=1$ and the Jensen's inequality 
 $$\Bigl(T_t^{\U, \QP}\bigl(\cdc^{\U}(u)^{\frac{1}{2}}\bigr)\Bigr)^{p} \le T_t^{\U, \QP}\bigl(\cdc^{\U}(u)^{\frac{p}{2}} \bigr) \fstop \qedhere$$ 
 \end{proof}
 \begin{rem}\ \label{r:TREM}
 
 \begin{enumerate}[(a)]
 \item 
The proof given in~\cite[Thm.~6.2]{AmbGigSav14b} in the framework of metric measure spaces does not directly apply to extended metric measure spaces. 
In our framework, we do not know whether $\mcT_t^{\U, \QP}\delta_\gamma$ can be defined for every $\gamma \in \U$. In usual metric measure spaces, this is possible because Dirac measures can be approximated by elements in $\mcP_\QP(\U)$ in terms of $\mssW_{p, \mssd_\U}$, by which we can extend $\mcT_t^{\U, \QP}$ from~$\mcP_\QP(\U)$ to Dirac measures, but this approximation property is not known in our extended metric measure setting.  Furthermore, since the regularisation property~$T_t^{\U, \QP}: L^\infty(\QP) \to \mcC_b(\U, \tau_\mrmv)$ is not available in our setting, functions relevant to the proof are not defined everywhere, rather only up to  a set of full measure. This prevents from applying the same proof to our setting because such a set of full measure could be in general very wild as a metric space and we lose e.g., the length property. Furthermore, a curve~$(\gamma_s)_{s \in [0,1]} \in \AC^2([0,1], \U, \mssd_\U)$ could stay in a set of measure zero and relevant functions  cannot be defined along such a curve. 
In the proof of Thm.~\ref{p:PBE}, these issues are settled by taking a good set~$\Theta$ in Assumption~\ref{a:AP}.
 \item The $2$-Bakry--\'Emery estimate was proven in \cite{Suz22b} for the {\it lower Dirichlet form} when $\QP=\sine_\beta$ for $\beta>0$ based on the DLR equation of $\QP$. However, the same proof does not apply to the upper Dirichlet form we used in this paper. As Thm.~\ref{p:PBE} needs no DLR equation,  we can apply our result to, e.g., $\QP=\Airy_2$, where the DLR equation remains unknown.  
 \end{enumerate}
 \end{rem}

 We say that $(\E^{\U, \QP}, \dom{\E^{\U, \QP}})$ is {\it irreducible} if $\E^{\U, \QP}(u)=0 \implies u \equiv c$ $\QP$-a.e.~with some constant $c$. 
 The following is a standard corollary of $\BE_1(K,\infty)$ with $K>0$.  
\begin{cor}  \label{c:LST} Suppose \ref{ass:CE}, ~\ref{ass:ConditionalClos} and  Assumption~\ref{a:AP} with $K > 0$. If $(\E^{\U, \QP}, \dom{\E^{\U, \QP}})$ is irreducible, then the following hold:
\begin{enumerate}[{\rm (a)}]
\item $(${\bf log-Sobolev inequality}$)$   \label{c:6-1}$$\Ent_\QP(\nu) \le \frac{2\Fis_\QP(\nu)}{K}  \cquad \nu \in \dom{\Ent_\QP} \cap \dom{\Fis_\QP} \fstop$$
\item \label{c:4-1}  $(${\bf Talagrand inequality}$)$  
$$\mssW_{2, \mssd_{\U}}^2(\nu, \QP) \le \frac{2}{K}\Ent_\QP(\nu) \cquad \nu \in \dom{\Ent_\QP} \fstop$$
%
\end{enumerate}
\end{cor}
\begin{proof}
See, e.g.,  \cite[Cor.~11.5]{AmbErbSav16} or \cite[Prop.~5.7.1]{BakGenLed14} for the log-Sobolev inequality. The Talagrand inequality with respect to~the variational distance~$\mssW_{\E}$ was proven in \cite[Lem.~10.7]{AmbErbSav16}. Thanks to the identification of $\Ch^{\U, \QP}$ and $\E^{\U, \QP}$ in Thm.~\ref{t:E=Ch}, we can apply \cite[(a) Prop.~7.4]{AmbErbSav16} to get the inequality~$\mssW_{\E} \ge \mssW_{2, \mssd_{\U}}$, which concludes the statement.
\end{proof}

\begin{rem}[Irreducibility]
 The irreducibility of $(\E^{\U, \QP}, \dom{\E^{\U,\QP}})$  holds under the number rigidity and the tail triviality of $\QP$, e.g., for $\QP=\sine_2, \Airy_2, \Bessel_{\alpha, 2}\ (\alpha \ge 1), \Ginibre$. See \cite[Theorem I and \S 6 Examples]{Suz23b}.
\end{rem}

\section{EVI gradient flows and dynamical number rigidity/tail triviality}
As an application of the previous sections,  we prove that the evolution variational inequality for the dual semigroup associated with~$(\E^{\U, \QP}, \dom{\E^{\U, \QP}})$. As a consequence, the dual semigroup is identified with  the unique $\EVI$ gradient flow of the relative entropy. 

\subsection{Evolution Variational Inequality} 
 \begin{thm}[$\EVI$ gradient flow] \label{t:main} Suppose~\ref{ass:CE}, ~\ref{ass:ConditionalClos} and Assumption~\ref{a:AP}. 
 Then, the dual semigroup~$(\mathcal T_t^{\U, \QP})_{t >0}$ satisfies $\EVI_K$
$:$ for every $\nu\in \dom{\Ent_\QP}$ and $\sigma \in \mcP(\U)$ with $\mssW_{2, \mssd_{\U}}(\nu, \sigma)<\infty$, it holds that $\mathcal T^{\U, \QP}_t \sigma \in \dom{\Ent_{\QP}}$, ${\mssW}_{2, \mssd_{\U}}\bigl({\mathcal T^{\U, \QP}_t \sigma}, \nu \bigr)<\infty$ and the following inequality holds:
 \begin{align} \tag*{$(\EVI_K)$} \label{EVI} 
\frac{1}{2}\frac{\diff^+}{\diff t}{\mssW}_{2, \mssd_{\U}}\bigl({\mathcal T^{\U, \QP}_t \sigma}, \nu \bigr)^2  + \frac{K}{2}{\mssW}_{2, \mssd_{\U}}\bigl({\mathcal T^{\U, \QP}_t \sigma}, \nu \bigr)^2 \le \Ent_{\mu}({\nu}) - \Ent_{\mu}({\mathcal T^{\U, \QP}_t \sigma})\comma 
\end{align}
where $\frac{\diff^+}{\diff t}$ denotes the upper right Dini derivative.
\end{thm}
\begin{proof}
Having all ingredients $\Ch^{\mssd_\U, \QP}=\E^{\U, \QP}$ in Thm.~\ref{t:E=Ch}, $\mcP_e=\overline{\dom{\Ent_\QP}}^{\mssW_{2, \mssd_\U}}$ in Prop.~\ref{p:DE}, the Bakry--\'Emery gradient estimate in Cor.~\ref{p:PBE}, the log-Harnack inequality in Thm.~\ref{t:DFH}, the $L\log L$-regularisation in Cor.~\ref{c:LLL} and the Wasserstein contraction in Thm.~\ref{t:WC},  the proof works verbatim as \cite[Thm.~5.10]{ErbHue15}, which is an adaption of~\cite[Thm.~4.17]{AmbGigSav15} to an extended metric space~$(\U, \mssd_\U)$.
%
\end{proof}

\begin{cor} \label{c:GC}Suppose~\ref{ass:CE}, ~\ref{ass:ConditionalClos} and Assumption~\ref{a:AP}.  Then, the following hold: 
\begin{enumerate}[{\rm (a)}]
\item \label{c:1} The space~$(\dom{\Ent_\QP}, {\mssW}_{2, \mssd_\U})$ is an extended geodesic metric space: for every pair $\nu, \sigma \in \dom{\Ent_\QP}$ with ${\mssW}_{2, \mssd_\U}(\nu, \sigma)<\infty$, there exists a constant speed ${\mssW}_{2, \mssd_\U}$-geodesic $(\nu_t)_{t \in [0,1]}$ connecting $\nu$ and $\sigma$:
\begin{align} \label{e:GD}
\nu_0=\nu \cquad \nu_1=\sigma \cquad {\mssW}_{2, \mssd_\U}(\nu_t, \nu_s) =|t-s|{\mssW}_{2, \mssd_\U}(\nu, \sigma) \cquad s, t \in [0, 1]\fstop
\end{align}
\item \label{c:2} $(${\bf $K$-geodesical convexity}$)$ The entropy~$\Ent_\QP$ is $K$-geodesically convex along every constant speed ${\mssW}_{2, \mssd_\U}$-geodesic $(\nu_t)_{t \in [0,1]}$:
\begin{align*}
\Ent_\QP(\nu_t) \le (1-t)\Ent_{\QP}(\nu_0) + t \Ent_{\QP}(\nu_1) -\frac{K}{2}t(1-t)\mssW_{2, \mssd_{\U}}(\nu_0, \nu_1)\cquad t \in [0,1] \fstop
\end{align*}
\item \label{c:8} $(${\bf $\RCD$ property}$)$ \label{c:8} The extended metric measure space $(\U, \tau_\mrmv, \mssd_\U, \QP)$ is an $\RCD(K,\infty)$ space. 
\item \label{c:3}  The descendent ${\mssW}_{2, \mssd_\U}$-slope of $\Ent_\QP$ coincides with the Fisher information:
\begin{align*}
|{\sf D}_{{\mssW}_{2, \mssd_\U}}^- \Ent_\QP(\nu)|^2 = 4{\sf F}_\QP(\nu) \cquad \nu \in \dom{\Fis_\QP}\ ;
\end{align*}

%

\item  \label{c:6}  $(${\bf Distorted Brunn--Minkowski inequality}$)$ Let $A_0, A_1$ be Borel sets with $\QP(A_0)\QP(A_1)>0$ and define the $t$-barycentre $(0 \le t \le 1)$ as 
$$[A_0, A_1]_t:=\{\eta \in \U: \eta=\gamma_t,\ \gamma \in {\rm Geo}(\U, \mssd_\U),\ \gamma_0 \in A_0,\ \gamma_1 \in A_1\} \fstop$$
Define $\nu_0:=\frac{\1_{A_0}}{\QP(A_0)} \cdot \mu$ and $\nu_1:=\frac{\1_{A_1}}{\QP(A_1)} \cdot \mu$. If $\mssW_{2, \mssd_{\U}}(\nu_0, \nu_1)<\infty$, 
\begin{align*}
\log \frac{1}{\mu\bigl([A_0, A_1]_t \bigr)} \le (1-t)	\log \frac{1}{\mu(A_0)} + t \log \frac{1}{\mu(A_1)} - \frac{K}{2}(1-t)t\mssW_{2, \mssd_\U}(\nu_0, \nu_1) \fstop
\end{align*}
\item  \label{c:7}  $(${\bf HWI inequality}$)$ For~$\nu_0 \in \dom{\Ent_\QP}, \nu_1 \in \mcP(\U)$ with  $\mssW_{2, \mssd_\U}(\nu_0, \nu_1)<\infty$, 
$$\Ent_{\QP}(\nu_0) \le \Ent_{\QP}(\nu_1)+2\mssW_{2, \mssd_{\U}}(\nu_0, \nu_1)\Fis_{\QP}(\nu_0)^{1/2} -\frac{K}{2}\mssW_{2, \mssd_{\U}}(\nu_0, \nu_1)^2 \fstop$$

\end{enumerate}
\end{cor}
\begin{proof} 
\ref{c:1} Thanks to~Cor.~\ref{p:EMW1}, for $\nu, \sigma \in \dom{\Ent_\QP}$ with $\mssW_{2, \mssd_\U}(\nu, \sigma)<\infty$, there exists a constat speed $\mssW_{2, \mssd_\U}$-geodesic $(\nu_t)_{t \in [0,1]} \subset \mcP(\U)$ with $\nu_0=\nu$ and $\nu_1=\sigma$. Due to  Prop.~\ref{p:DE}, we have $(\nu_t)_{t \in [0,1]} \subset \mcP_e(\U)$.  It suffices to prove $(\nu_t)_{t \in [0,1]} \subset \dom{\Ent_\QP}$. Let $\nu_t^\e:=\mcT_\e^{\U, \QP}\nu_t$. Due to Cor.~\ref{c:LLL}, we have $\mcT_\e^{\U, \QP}\nu_t \in \dom{\Ent_\QP}$ and $\mssW_{2, \mssd_\U}(\mcT_\e^{\U, \QP}\nu_t, \nu)<\infty$ for every $\e>0$, by which we may think of $\e \mapsto \mcT_\e^{\U, \QP}\nu_t$ as an $\EVI$-flow in {\it non-extended} metric space. Thus, applying~the result~\cite[Thm.~3.2]{DanSav08} in metric spaces, we have 
\begin{align} \label{e:CD}
\Ent_\QP(\mcT_\e^{\U, \QP} \nu_t) \le (1-t)\Ent_\QP(\nu)+t\Ent_\QP(\sigma)-\frac{K}{2}t(1-t)\mssW_{2, \mssd_\U}^2(\nu, \sigma) \fstop 
\end{align}
Since $\mcT_\e^{\U, \QP} \nu_t \xrightarrow{\tau_\mrmw} \nu_t$ as $\e \to 0$ due to \eqref{c:LLL3}, the $\tau_\mrmw$-lower semi-continuity of the relative entropy yields $(\nu_t)_{t \in [0,1]} \subset \dom{\Ent_\QP}$.

\ref{c:2} immediately follows by~the formula~\eqref{e:CD} and the lower semi-continuity $\Ent_\QP(\nu_t) \le \liminf_{\e \to 0}\Ent_\QP(\mcT_\e^{\U, \QP} \nu_t)$.
 
 \ref{c:8} follows by \ref{c:2} and Thm.~\ref{t:E=Ch}. 
 
 \ref{c:3} follows by~\ref{c:2} and \cite[Thm.~9.3]{AmbGigSav14}.

 \ref{c:6} The same proof as~\cite[Thm.~18.5, $N=\infty$]{Vil09} works even for extended metric measure spaces. For the sake of clarity, we give the proof below. We apply~\ref{c:2} with~$\nu_0=\frac{\1_{A_0}}{\QP(A_0)} \cdot \mu$ and~$\nu_1:=\frac{\1_{A_1}}{\QP(A_1)} \cdot \mu$.   By~\ref{c:1} and \ref{c:2}, there exists a $\mssW_{2, \mssd_\U}$-constant geodesic $(\nu_t)_{t \in [0,1]}$ connecting $\nu_0$ and $\nu_1$ such that 
 \begin{align*}
\Ent_\QP(\nu_t) \le (1-t) \log\frac{1}{\nu(A_0)} + t \log\frac{1}{\nu(A_1)}-\frac{K}{2}t(1-t)\mssW_{2, \mssd_{\U}}(\nu_0, \nu_1)\cquad t \in [0,1] \fstop
\end{align*}
Let $\nu_t=f_t\cdot\mu$. Due to the result~\cite[Thm.~4.2]{Lis16} in extended metric measure spaces,  we have $\supp[\nu_t] \subset [A_0, A_1]_t$.  So, $\nu_t( [A_0, A_1]_t)=1$.  By Jensen's inequality, 
\begin{align*}
\Ent_\QP(\nu_t) &= \mu([A_0, A_1]_t)\int_{[A_0, A_1]_t} (f_t \log f_t) \frac{1}{\mu([A_0, A_1]_t)} \diff \QP 
\\
&\ge\mu([A_0, A_1]_t)\int_{[A_0, A_1]_t} f_t \frac{1}{\mu([A_0, A_1]_t)} \diff \QP \log \int_{[A_0, A_1]_t} f_t \frac{1}{\mu([A_0, A_1]_t)} \diff \QP 
\\
&=\nu_t([A_0, A_1]_t)\log \frac{\nu_t([A_0, A_1]_t)}{\mu([A_0, A_1]_t)} = \log \frac{1}{\mu([A_0, A_1]_t)} \comma
\end{align*}
which completes the proof. 
 
\ref{c:7}  Let $\mssW_{t}:=\mssW_{2, \mssd_\U}(\mcT_t^{\U, \QP}\nu_0, \nu_1)$. We may assume $\nu_0 \in \dom{\Fis_\QP}$, otherwise, we have nothing to prove. By Thm.~\ref{t:main}, ${\mssW}_t<+\infty$ and the following inequality holds:
 \begin{align}  \label{e: EVIT} 
\frac{1}{2}\frac{\diff^+}{\diff t}{\mssW}_t^2  + \frac{K}{2}{\mssW}_{t}^2 \le \Ent_{\mu}({\nu_1}) - \Ent_{\mu}({\mathcal T^{\U, \QP}_t \nu_0})\fstop 
\end{align}
By \ref{c:LLL2} in Cor.~\ref{c:LLL}, $\mssW_t \xrightarrow{t \to 0}\mssW_0=\mssW_{2, \mssd_\U}(\nu_0, \nu_1)$.  Furthermore, $\Ent_\QP({\mathcal T^{\U, \QP}_t \nu_0}) \xrightarrow{t \to 0} \Ent_\QP(\nu_0)$, which can be easily seen by the $\tau_\mrmw$-lower semi-continuity of $\Ent_\QP$ and the contraction~\eqref{e:CEH}: 
$$\Ent_\QP(\nu_0) \ge \limsup_{t \to 0} \Ent_\QP(\mathcal T^{\U, \QP}_t \nu_0) \ge  \liminf_{t \to 0} \Ent_\QP(\mathcal T^{\U, \QP}_t \nu_0)  \ge  \Ent_\QP(\nu_0) \fstop$$
Thus, applying~\eqref{e: EVIT} with $t=0$, we have 
$$\limsup_{t \to 0^+}\frac{\mssW_t^2-\mssW_0^2}{2t} + \frac{K}{2}\mssW_0^2 \le \Ent_\QP(\nu_1)-\Ent_\QP(\nu_0) \fstop$$
In particular, 
\begin{align} \label{e:KI1}
\limsup_{t \to 0^+}\frac{\mssW_t-\mssW_0}{t} 
&= \limsup_{t \to 0^+}\frac{\mssW_t^2-\mssW_0^2}{t(\mssW_t+\mssW_0)} =  \limsup_{t \to 0^+}\frac{\mssW_t^2-\mssW_0^2}{2t\mssW_0} 
\\
&\le \frac{ \Ent_\QP(\nu_1)-\Ent_\QP(\nu_0)}{\mssW_0} -\frac{K}{2}\mssW_0 \fstop \notag
\end{align}
Here we may assume $\mssW_0>0$ (otherwise there is nothing to prove), and we used $\lim_{t \to 0}\mssW_t=W_0$.
We have $-(\mssW_t-\mssW_0) \le \mssW_{2, \mssd_\U}(\mcT_t^{\U, \QP} \nu_0, \nu_0)$ by triangle inequality. By e.g.,~\cite[(3.17)]{MurSav20}, it holds that 
$$\lim_{t \to 0^+}\frac{\mssW_{2, \mssd_\U}(\mcT_t^{\U, \QP} \nu_0, \nu_0)}{t}= |{\sf D}_{{\mssW}_{2, \mssd_\U}}^- \Ent_\QP(\nu_0)| \fstop$$ 
Thus, applying~\ref{c:3}, 
$$\liminf_{t \to 0^+}\frac{-(\mssW_t-\mssW_0)}{t}=-\limsup_{t \to 0^+}\frac{\mssW_t-\mssW_0}{t} \le 2 \sqrt{{\sf F}_\QP(\nu_0)} \fstop$$
Combined with~\eqref{e:KI1}, the conclusion follows by~
$$\Ent_\QP(\nu_0) \le \Ent_\QP(\nu_1) +  2 \mssW_{2, \mssd_\U}(\nu_0, \nu_1)\sqrt{{\sf F}_\QP(\nu_0)}  - \frac{K}{2}\mssW_{2, \mssd_\U}(\nu_0, \nu_1)^2 \fstop \qedhere$$  
\end{proof}

\subsection{Gradient flow} In this section, we show that the dual semigroup $\bigl(\mathcal T_t^{\U, \QP}\nu_0\bigr)_{t \ge 0}$ is the unique $\mssW_{2, \mssd_{\U}}$-gradient flow in the sense of curve of maximal slope as well as the minimising movement. 
For $\nu \in \mcP(\U)$, we define 
$$\mcP_{[\nu]}(\U):=\{\sigma \in \mcP(\U): \mssW_{2, \mssd_\U}(\nu, \sigma)<+\infty\} \fstop$$
\begin{cor}[Gradient flow: curves of maximal slope] \label{c:GF}Suppose~\ref{ass:CE}, ~\ref{ass:ConditionalClos} and Assumption~\ref{a:AP}. 
Then, the dual semigroup~$\bigl(\mathcal T_t^{\U, \QP}\nu_0\bigr)_{t \ge 0}$ is the unique solution to the $\mssW_{2, \mssd_{\U}}$-gradient flow of $\Ent_{\QP}$ starting at $\nu_0$ in the sense of curves of maximal slope. Namely, for every $\nu_0 \in \mcP_e(\U)$, the curve $(0, \infty) \ni t \mapsto \nu_t=\mathcal T_t^{\U, \QP} \nu_0\in \dom{\Ent_\QP}$ is the unique solution to the energy equality:
\begin{align*}
\frac{\diff}{\diff t} \Ent_\QP({\nu_t}) = -|\dot\nu_t|^2 = -|{\sf D}^-_{\mssW_{2, \mssd_{\U}}} \Ent_\QP|^2(\nu_t) \quad\text{a.e.~$t>0$} \fstop
\end{align*}
\end{cor}
%
\begin{proof}
This is a direct application of Thm.~\ref{t:main} and \cite[Thm.~3.5]{MurSav20} with ~\cite[Thm.~4.2]{MurSav20} for the uniqueness. Although \cite{MurSav20} proves only in metric spaces, we can apply the result to our setting as follows: Noting that $\mssW_{2, \mssd_\U}(\nu_t, \nu_s)<\infty$ for every $s, t \in [0,\infty)$ by Thm.~\ref{t:main} with $\sigma=\nu_0$, we can regard $(\nu_t)$ as a curve in a (non-extended) metric space $\mcP_{[\nu_0]}(\U)$ to which we can apply the results in \cite{MurSav20}. 
\end{proof}


\begin{defs}[Minimising Movement {\cite[Dfn.~5.2 with $\eta=0$]{MurSav20}}] \ \label{d:MMS}
\begin{itemize}
\item Let $\tau>0$ and $U_\tau^0 \in \dom{\Ent_\QP}$. {\it A $\tau$-discrete minimising movement} starting at $U_\tau^0$ is any sequence $(U_\tau^n)_{n \in \N} \subset \mcP_{[U_\tau^0]}(\U)^{\times \infty}$ 
such that 
$$\frac{1}{2\tau}\mssW_{2, \mssd_\U}(U_\tau^n, U_\tau^{n-1})^2+\Ent_\QP(U_\tau^n) \le \frac{1}{2\tau}\mssW_{2, \mssd_\U}(V, U_\tau^{n-1})^2+\Ent_\QP(V)$$
for every $V \in \dom{\Ent_\QP}$ and $n \in \N$. We denote by $\MM_\tau(U_\tau^0) \subset \mcP_{[U_\tau^0]}(\U)^{\times \infty}$ the collection of all $\tau$-discrete minimising movements starting at~$U_\tau^0$;
\item A piecewise constant interpolation $t \mapsto \overline{U}_\tau(t)$ of $(U_\tau^n)_{n \in \N} \in \MM_\tau(U_\tau^0)$ is defined as 
$$\overline{U}_\tau(t):={U}^n_\tau \quad t \in ((n-1)\tau, n\tau) \cquad \overline{U}_\tau(0):=U^0_\tau \scolon$$
\item We say that a curve $u: [0,+\infty) \to \mcP_{[u_0]}(\U)$ is a {\it minimising movement (a.k.a.~JKO scheme)}, denoted by $u \in \MM(\mcP_{[u_0]}(\U), \mssW_{2, \mssd_\U}, \Ent_\QP)$,  if $u(0)=u_0 \in \dom{\Ent_\QP}$ and there exists an element~$(U_\tau^n)_{n \in \N} \in \MM_\tau(u_0)$ for sufficiently small $\tau$ such that its piecewise constant interpolation~$\overline{U}_\tau$ satisfies
$$\mssW_{2, \mssd_\U}\text{-}\lim_{\tau \downarrow 0}\overline{U}_\tau(t)=u(t) \quad t \ge 0  \fstop$$
\end{itemize}
\end{defs}

The following statement shows that the dual semigroup~$\bigl(\mathcal T_t^{\U, \QP}\nu_0\bigr)_{t >0}$ is the unique $\mssW_{2, \mssd_\U}$-gradient flow in the sense of minimising movement (a.k.a.~JKO scheme) as well. 
\begin{cor}[Gradient flow: minimising movement] \label{c:MM} Suppose~\ref{ass:CE}, ~\ref{ass:ConditionalClos} and Assumption~\ref{a:AP}. Then,  the dual semigroup $t \mapsto \mathcal T_t^{\U, \QP}\nu_0$ is the unique element of $\MM(\mcP_{[\nu_0]}(\U), \mssW_{2, \mssd_\U}, \Ent_\QP)$ for every $\nu_0 \in \dom{\Ent_\QP}$. 
\end{cor}
\begin{proof}
Thanks to Cor.~\ref{c:LLL}, $\mssW_{2, \mssd_\U}(\nu_0, \mcT_t^{\U, \QP}\nu_0)<+\infty$ for every $t \ge 0$. So we can think of $t \mapsto \mcT_t^{\U, \QP}\nu$ as a curve in a non-extended metric space~$\mcP_{[\nu_0]}(\U)=\{\sigma \in \mcP(\U): \mssW_{2, \mssd_\U}(\nu_0, \sigma)<+\infty\}$. To complete the proof, we apply~\cite[Cor..~5.7]{MurSav20} with $X=\mcP_{[\nu_0]}$, $\mssd=\mssW_{2, \mssd_\U}$ and $\varphi=\Ent_\QP$. We note that the separability of the topology induced by distance is not assumed there, see~\cite[(3.1)]{MurSav20}.  The domain~$\dom{\Ent_\QP}$ is geodesic due to~Cor.~\ref{c:GC}, thus in particular, it is an approximate length subset in the sense of \cite[Dfn.~2.5]{MurSav20}. Now it suffices to show that the sublevel set~$E_c:=\{\Ent_\QP \le c\}$ is $\mssW_{2, \mssd_\U}$-complete for every $c \in \R_+$. Let $(\nu_n)_{n \in \N}\subset E_c$ be a $\mssW_{2, \mssd_\U}$-Cauchy sequence. Due to the compactness of $E_c$ with respect to the weak topology~$\tau_\mrmw$, we can take a $\tau_\mrmw$-converging (non-relabelled) subsequence with limit $\nu$. By the $\tau_\mrmw$-lower semi-continuity of $\mssW_{2, \mssd_\U}$, we have 
$$\mssW_{2, \mssd_\U}(\nu_n, \nu) \le \mssW_{2, \mssd_\U}(\nu_n, \nu_m)+\mssW_{2, \mssd_\U}(\nu_m, \nu) \le \mssW_{2, \mssd_\U}(\nu_n, \nu_m)+\liminf_{k \to \infty}\mssW_{2, \mssd_\U}(\nu_m, \nu_k)\comma$$
the RHS of which can be arbitrarily small. The proof is complete. 
\end{proof}

 \subsection{Dynamical number rigidity/tail triviality}
Due to the quasi-regularity in Cor.~\ref{c:QR}, there exists $\U_0 \subset \U$ such that $\U \setminus \U_0$ is a set of capacity zero and for every $\gamma \in \U_0$, there exists a diffusion process $(\mathsf X_t, \mathbb P_\gamma)$ such that $P_t^{\U, \QP}f(\gamma):=\mathbb E_\gamma[f(\mathsf X_t)]=T_t^{\U, \QP}f(\gamma)$ $\QP$-a.e.~$\gamma$ for $f \in \mcB_b(\U)$, see \cite[Thm.~3.5 in Chapter IV, Thm.~1.5 in Chapter V]{MaRoe90}. Furthermore, $\gamma \mapsto P_t^{\U, \QP}f(\gamma)$ is quasi-continuous, see \cite[Thm~2.13 and Prop.~2.15 in Chapter V]{MaRoe90}. In the following argument, we identify $P_t^{\U, \QP}$ and $T_t^{\U, \QP}$ and regard $T_t^{\U, \QP}f$ as being quasi-continuous. We denote the  transition probability $p_t^{\U, \QP}(\gamma, \diff \eta)$ of $(\mathsf X_t, \mathbb P_\gamma)$ by $\mcT_t^{\U, \QP}\delta_{\gamma}(\diff \eta) \in \mcP(\U)$. 
\begin{cor} \label{t:NRF} Suppose~\ref{ass:CE}, ~\ref{ass:ConditionalClos}.  
\begin{enumerate}[(i)]
\item \label{i:NRF} If  $\QP$ satisfies~\ref{ass:NR}, then, for every fixed $t>0$,  there exists a set $\Sigma \subset \U$ of $\E^{\U, \QP}$-capacity zero such that   $\mathcal T_t^{\U, \QP} \delta_\gamma$  satisfies~\ref{ass:NR} for every $\gamma \in \Sigma$;
\item \label{iii:NRF} If Assumption~\ref{a:AP} holds and $\QP$ satisfies~\ref{ass:NR} $($resp.~\ref{ass:T}$)$, then~$\mathcal T_t^{\U, \QP}\nu$  satisfies \ref{ass:NR} $($resp.~\ref{ass:T}$)$ for every $t>0$ and~$\nu \in \mcP_e(\U)$.
\end{enumerate}
\end{cor}

\begin{proof}
\ref{i:NRF}: Observe that,  to verify the number rigidity, it suffices to verify the sought property on {\it every countable ball} $B_r=\{x \in \R^n: |x|<r\}$ $(r \in \N)$ instead of {\it every bounded set $B$}. Indeed, suppose that the sought property holds on every $B_r$ ($r \in \N$). For a bounded set $B$, take a sufficiently large ball $B_r$ containing $B$. Let $\Phi_{B_r}: \U(B_r^c) \to \N$ be a function witnessing the sought property in $B_r$, i.e., $\gamma(B_r)=\Phi_{B_r}(\gamma_{B_r^c})$. Define $\Phi_B:\U(B^c) \to \N$ as $\Phi_B(\eta):=\Phi_{B_r}(\eta_{B_r^c})-\eta({B_r \setminus B})$. Then, for $\QP$-a.e.~$\gamma$, we have $\gamma(B)=\gamma(B_r) - \gamma(B_r \setminus B) = \Phi_{B_r}(\gamma_{B_r^c})- \gamma(B_r \setminus B)=\Phi_B(\gamma_{B^c})$, which witnesses the sought property in $B$.

Now take $B_r=\{x \in \R^n: |x|<r\}$ for $r \in \N$.  
 By the hypotheis~\ref{ass:NR}, there exists~$\Omega_r  \subset \U$ with $\QP(\Omega_r)=1$ and a Borel function~$\Phi_{B_r}: \U(B^c_r) \to \N_0$ such that $\Phi_{B_r}(\gamma_{B^c_r})=\gamma(B_r)$ for every $\gamma \in \Omega_r$. 
Due to the conservativeness~$T_t^{\U, \QP}\1=\1$ and the $L^2(\QP)$-symmetry $(T_t^{\U, \QP}u, v)_{L^2}=(u, T_t^{\U, \QP}v)_{L^2}$, we have the invariance~$\int_{\U} T_t^{\U, \QP}u \diff \QP = \int_{\U}u \diff \QP$ for every $u \in L^2(\QP)$. Thus,
$$1=\QP(\Omega_r)= \int_{\U} T_t^{\U, \QP}\1_{\Omega_r} \diff \QP  =  \int_{\U}  \mathcal T_t^{\U, \QP}\delta_\gamma(\Omega_r) \diff \QP(\gamma) \fstop$$
Hence,  $T_t^{\U, \QP}\1_{\Omega_r}(\gamma)=\mathcal T_t^{\U, \QP}\delta_\gamma(\Omega_r)=1$ for every $t>0$ and for $\QP$-a.e.~$\gamma$. 
Due to the quasi-regularity in Cor.~\ref{c:QR},  $T_t^{\U, \QP}\1_{\Omega_r}$ is quasi-continuous. Thus, there exists a $\tau_\mrmv$-closed nest $(K_n)_{n \in \N}$ such that $K_n \nearrow \cup_{n \in \N}K_n$, $\U \setminus \cup_{n \in \N}K_n$ is exceptional  and~$T_t^{\U, \QP}\1_{\Omega_r} \equiv 1$ on $K_n \cap \supp_{\tau_\mrmv}[\QP]$ for every $n \in \N$, where $\supp_{\tau_\mrmv}[\QP]$ is the support of $\QP$ with respect to the topology~$\tau_\mrmv$. Thus, $\mathcal T_t^{\U, \QP} \delta_\gamma(\Omega_r)= 1$ for q.e.~$\gamma$, i.e., it holds for every $\gamma \in \Sigma_r$ with some $\Sigma_r \subset \U$ such that $\Cap_{\E^{\U, \QP}}(\Sigma^c_r)=0$.  
Take $\Omega=\cap_{r \in \N} \Omega_r$ and $\Sigma=\cap_{r \in \N}\Sigma_r$. Then, $\Cap_{\E^{\U, \QP}}(\Sigma^c)=0$, $\mcT_t^{\U, \QP}\delta_\gamma(\Omega)=1$ for every $\gamma \in \Sigma$.  For each $r \in \N$, we have $\gamma(B_r)=\Phi_{B_r}(\gamma_{B_r^c})$ for every $\gamma \in \Omega$.  This concludes that, for every $\gamma \in \Sigma$, the probability measure $\mcT_t^{\U, \QP}\delta_\gamma$ is number rigid. 

 

\smallskip
\ref{iii:NRF}: This is a direct consequence from the regularisation property~$\mathcal T_t^{\U, \QP}: \mcP_e(\U) \to \dom{\Ent_\QP}$ by Cor.~\ref{c:LLL}, yielding $\mathcal T_t^{\U, \QP}\nu \ll \QP$ for positive $t>0$. Due to the inheritance of the number rigidity and the tail triviality to absolutely continuous probability measures, the proof is complete. 
\end{proof}

\begin{ese} \label{exa: R} \ 
\begin{itemize}
\item The number rigidity has been verified for a variety of point processes: $\Ginibre_2$ and GAF (\cite{GhoPer17}), the two-dimensional one-component plasma (\cite{Leb24}), $\Ginibre_\beta$ $\mathrm{sine}_\beta$ (\cite[Thm.\ 4.2]{Gho15}, \cite{ChhNaj18}, \cite{DerHarLebMai20}), $\mathrm{Airy}$, ~$\mathrm{Bessel}$, and $\mathrm{Gamma}$ (\cite{Buf16}), and Pfaffian (\cite{BufNikQiu19}) point processes. We refer the readers also to the survey \cite{GhoLeb17}. 
\item The tail triviality has been verified in the following cases: 
\begin{enumerate}[(i)]
\item (Determinantal point processes) Let $X$ be a locally compact Polish space. 
Then, all determinantal point processes whose kernel are locally trace-class positive contraction satisfy the tail triviality (see \cite[Theorem 2.1]{Lyo18} and \cite{BufQiuSha21,  OsaOsa18, ShiTak03b}). In particular, $\mathrm{sine}_2$, $\mathrm{Riesz}_{2}$, $\mathrm{Bessel}_{\alpha,2}$, $\mathrm{Airy}_2$ and Ginibre point processes are tail trivial. 
\item (Extremal Gibbs measure)  A canonical Gibbs measure $\mu$ is tail  trivial iff $\mu$ is extremal (see \cite[Cor.\ 7.4]{Geo11}). In particular, Gibbs measures of the Ruelle type with sufficiently small activity constants are extremal (see \cite[Thm.\ 5.7]{Rue70}).
\end{enumerate}

\end{itemize}
\end{ese}

 We show that there are elements in  $\mcP_e(\U)=\{\nu \in \mcP(\U): \mssW_{2, \mssd_\U}(\nu, \sigma)<+\infty,\ \exists \sigma \in \dom{\Ent_\QP}\}$ such that they are neither absolutely continuous with respect to $\QP$ nor tail trivial. 
\begin{prop} \label{p:ENT}
Let $\QP=\sine_2$ or $\QP=\Airy_2$. Then, $\mcP_e(\U)$ contains an element $\sigma$ such that it is neither absolutely continuous with respect to $\QP$ nor tail trivial.  
\end{prop}

\begin{proof} 
As explained in Example~\ref{exa: R}, both measures $\QP=\sine_2$ and $\QP=\Airy_2$ possess the tail triviality. Since any absolutely continuous probability measure with respect to $\QP$ inherits the tail triviality, 
we only need to construct a non tail trivial measure~$\sigma$,  such $\sigma$ has to be non-absolutely continuous.

\smallskip
Let $\Upsilon \subset \U$ be the space of configurations in $\R$ without multiplicity. Noting $\sine_2(\Upsilon)=\Airy_2(\Upsilon)=1$, we restrict the space to $\Upsilon$ in the following argument. 
Define the shell~$I_k=[-k-\frac{1}{k}, -k)$ for $k \in \N$. Note that $(I_k)_{k \in \N}$ are disjoint sets in the non-negative real $\R_-$ and the total length of the shells is $\sum_{k \in \N} |I_k|= \sum_{k=1}^\infty 1/k = +\infty$. We first show that infinitely many shells contain at least one point from $\QP$-a.e.~$\gamma$, namely, we prove the following claim:

\smallskip
{\it Claim 1.}  $\QP\Bigl(\{\gamma: \# \{k: \gamma(I_k) \ge 1\}=\infty\}\Bigr)=1$.  
\begin{proof}[Proof of Claim 1]
Let $A_k:=\{\gamma: \gamma(I_k) \ge 1\}$ and $p_k:=\QP(A_k)$. Let $f^{(1)}=f^{(1)}_\QP$ be the intensity measure of $\QP$ (see  Dfn.~\ref{d:DPP}). Since $\QP$ is the law of a determinantal point process with local trace class kernel (say $K$), we have 
$$\QP\bigl( \{\gamma(I_k)=0\}\bigr) = \det(I-K_{I_k}) \le \exp(-{\rm tr} K_{I_k}) = \exp\biggl( - \int_{I_k} f^{(1)}(x) \diff x\biggr) \comma$$
where $K_{I_k}$ is the restriction of the operator $K$ in $I_k$ and $\det$ is the Fredholm determinant, see Rem.~\ref{r:TCO}. Here, the inequality follows by the fact that spectra $\{\lambda_i\}$ of $K_{I_k}$ lie in $[0,1]$ and $\prod (1-\lambda_i) \le e^{-\sum_i \lambda_i}$. Hence, 
\begin{align} \label{e:DPEE}
p_k \ge 1- \exp\biggl( - \int_{I_k} f^{(1)}(x) \diff x\biggr) \fstop
\end{align}
Since $f^{(1)}$ is constant for $\QP=\sine_2$ and $f^{(1)} \sim (1/\pi) \sqrt{-x}$ as $x \to -\infty$ for $\QP=\Airy_2$, we can take $c>0$ such that $f^{(1)} \ge c$ on $I_k$ for every $k \ge M$ for some sufficiently large $M$. Thus, using $1-e^{-t} \ge t/2$ for $t \in [0,1]$
$$\int_{I_k} f^{(1)} \diff x  \ge c |I_k|= \frac ck \quad \implies \quad p_k \ge 1 - e^{-c/k} \ge \frac c{2k} \fstop$$
Thus 
\begin{align} \label{e:DPEE0}
\sum_{k \in \N} p_k = +\infty \fstop
\end{align}
Due to the negative association property of determinantal point processes (\cite[Thm.~1.4]{Gho15}), for any finite $J \subset \N$,  
\begin{align} \label{e:DPEE1}
\QP\Bigl( \bigcap_{k \in J} \{\gamma(I_k)=0\}\Bigr) \le \prod_{k \in J} \QP\Bigl(  \{\gamma(I_k)=0\} \Bigr)  = \prod_{k \in J} (1-p_k) \fstop
\end{align}
Let $B_m=\bigcap_{k \ge m } \{\gamma(I_k)=0\}$ (no hit in any shell from $m$ onward). Applying~\eqref{e:DPEE1} with $J=\{m, m+1, \ldots, n\}$ and then $n \to \infty$, we have 
$$\QP(B_m) \le \prod_{k \ge m} (1-p_k) \le \exp\Bigl(- \sum_{k \ge m} p_k \Bigr) = 0 \cquad m \in \N \comma$$
where we used $1-x \le e^{-x}$ and \eqref{e:DPEE0}.  Let $B=\cup_{m \in \N}B_m$ (the event that only finitely many $I_k$ are hit). Then, 
$$\QP(B)=0\comma$$
which completes the proof. 
\end{proof}

{\it We resume the proof of Prop.~\ref{p:ENT}.} Define $a_k:=-k-\frac{1}{2k}$ (the midpoint in the shell $I_k$) and $\Phi: \Upsilon \to \Upsilon$ as follows: For a configuration $\gamma$, for each $k$ with $\gamma(I_k) \ge 1$, pick the leftmost point in $\supp[\gamma] \cap I_k$ and move it to $a_k$; if $\gamma(I_k)=0$, do nothing. By construction, 
$$\mssd_\U(\gamma, \Phi(\gamma))^2 \le \sum_{k \in \N} \frac{1}{k^2}<+\infty \comma$$
which is uniformly bounded in $\gamma$. Define $\nu:=\Phi_\#\QP$. Then,  
$$\mssW_{2, \mssd_\U}(\QP, \nu)^2 \le  \int_{\Upsilon} \mssd_\U(\gamma, \Phi(\gamma))^2 \diff \QP(\gamma)<+\infty \fstop$$
Define the following mixed measure 
$$\sigma:=\frac{1}{2} \nu + \frac{1}{2}\QP \fstop$$
Let $\pi$ be the optimal coupling of $\nu$ and $\QP$, and $\Delta_\#\QP:=({\rm id}, {\rm id})_\#\QP$ be the diagonal coupling. Then, $(1/2) \pi + (1/2)\Delta_\#\QP$ is a coupling of $\sigma$ and $\QP$, which yields
$$\mssW_{2, \mssd_\U}(\sigma, \QP)^2 \le\frac{1}{2} \mssW_{2, \mssd_\U}(\nu, \QP)^2 + \frac{1}{2} \mssW_{2, \mssd_\U}(\QP, \QP)^2<+\infty\fstop$$
Consider the tail event 
$$T=\{\gamma \in \Upsilon: \gamma(a_k) =1 \ \text{for infinitely many $k$}\} \fstop$$
Due to Claim 1, $\nu(T)=1$ but $\QP(T)=0$. Thus, $\sigma(T)=1/2$, which disproves the tail-triviality. 
\end{proof}
\begin{rem}
A remaining  question is whether $\mcP_e(\U)$ contains an element $\sigma$ that is not number rigid. A simpler form of this question is the following: Suppose that $\QP=\sine_2$ or $\QP=\Airy_2$. Can we find $\nu$ such that  $\mssW_{2, \mssd_\U}(\nu, \QP)<+\infty$ but $\nu$ is not number rigid. We conjecture that such $\nu$ does not exist and $\mssW_{2, \mssd_\U}(\nu, \QP)<+\infty$ forces~$\nu$ to be number rigid. 
\end{rem}

\section{Verifications of the main assumptions for $\sine_2$ and $\Airy_2$} \label{sec: Ver}
In this section, we verify the main assumptions \ref{ass:CE}, ~\ref{ass:ConditionalClos} and Assumption~\ref{a:AP} for several concrete point processes $\QP$.

\subsection{Quasi-Gibbs measures} 
Let~$\Phi\colon \R^n\rar\R\cup\{+\infty\}$ be a Borel measurable function and $\Psi\colon (\R^n)^\tym{2}\rar\R \cup\{+\infty\}$ be a Borel measurable and symmetric function.
We define a \emph{Hamiltonian}~$\msH_r\colon \dUpsilon(B_r)\rar \R$ as
\begin{align*}
\msH_r\colon \gamma\longmapsto \int_{B_r} \Phi(x) \diff \gamma(x)+\int_{(B_r)^{\times 2} \setminus {\sf diag}}\Psi(x,y) \diff \gamma^{\otimes 2}(x,y) \cquad \gamma\in \dUpsilon(B_r)\comma
\end{align*}
where ${\sf diag}:=\{(x, x) \in B_r^{\times 2}\}$ is the diagonal set. The function~$\Phi$ (resp.~$\Psi$ ) is called \emph{free potential} (resp.~\emph{interaction potential}) for $\msH$.
We define 
$$\mathcal K_r^\eta:=\{k \in \N_0: \QP_r^\eta(\U^k(B_r))>0\} \fstop$$
 Recall that $\mssm_r$ is the $n$-dimensional Lebesgue measure restricted on $B_r \subset \R^n$.
\begin{defs}[Quasi-Gibbs measures, {\cite[Dfn.~2.2]{OsaTan20}}]\label{d:QuasiGibbs}
We say that a Borel probability~$\QP$ on~$\U$ is a \emph{$(\Phi,\Psi)$-quasi-Gibbs measure} if, for $\QP$-a.e.~$\eta\in\dUpsilon$, every~$r>0$ and every~$k\in \mathcal K_r^\eta$, there exists a constant~$c_{r,\eta,k} \in (0, \infty)$ such that
\begin{align}\label{eq:localACquasiGibbs}
c_{r,\eta,k}^{-1}\, e^{-\msH_r} \cdot \mssm_r^{\odot k} \ \leq \QP_r^{\eta, k} \ \leq \ c_{r,\eta,k}\, e^{-\msH_r} \cdot \mssm_r^{\odot k}\fstop
\end{align}

\end{defs}

\begin{rem}\label{r:QuasiGibbs} \ 
\begin{enumerate}[$(a)$]
\item\label{i:r:QuasiGibbs:1} The definition of quasi-Gibbs measures in~\cite[Dfn.~2.2]{OsaTan20} looks slightly different from Dfn.~\ref{d:QuasiGibbs} as we assume~\eqref{eq:localACquasiGibbs} only for~$k\in \mathcal K_r^\eta$ instead of every $k\in \N_0$. These two definitions are, however, equivalent since the definitions of~$\QP_r^{\eta, k}$ in this article is {\it the restriction} on $\U^k(B_r)$:
$$\QP_r^{\eta, {k}}:=\QP_r^\eta\mrestr{\U^k(B_r)} \comma$$ 
 while the corresponding measure in~\cite[Dfn.~2.2]{OsaTan20} has been defined as the measure {\it conditioned} on $\U^k(B_r)$.
\item\label{i:r:QuasiGibbs:3}``$\QP$ belongs to $(\Phi,\Psi)$-quasi-Gibbs measures'' {\it does not necessarily} mean that $\QP$ is governed by the free potential $\Phi$ in the sense of the DLR equation. The symbol $\Phi$ here just plays a role as {\it representative} of the class of $(\Phi,\Psi)$-quasi-Gibbs measures {\it modulo perturbations by adding locally finite free potentials}. 
Noting that the constant $c_{r,\eta,k}$ can depend on $r, \eta, k$, if $\QP$ is $(\Phi,\Psi)$-quasi-Gibbs, then $\QP$ is $(\Phi+\Phi',\Psi)$-quasi-Gibbs as well whenever $\Phi'|_{B_r}$ is bounded for every $r \in \N$. Thus, if $\Phi$ is locally finite, we may write $(0,\Psi)$-quasi-Gibbs instead of $(\Phi,\Psi)$-quasi-Gibbs. 
\end{enumerate}
\end{rem}

The class of quasi-Gibbs measures includes all canonical Gibbs measure, and the laws of some determinantal/permanental point processes, as for instance: $\sine_\beta$ and $\Airy_{\beta}$. We discuss these two examples below. 
\begin{ese}[$\mathsf{sine}_\beta$] \label{exa: S}
By \cite[Thm.\ 2.2]{Osa13} and \cite[Thm.~1.1]{DerHarLebMai20}, $\mathsf{sine}_\beta$  with~$\beta>0$ is a Gibbs (thus also quasi-Gibbs) measure with the interaction potential
$$\Psi(x, y):=-\beta\log|x-y|, \quad x, y \in \R \fstop$$ 
The corresponding SDE is given by 
\begin{align*}
\diff X^i_t = \frac{\beta}{2} \lim_{r \to \infty}\sum^\infty_{j: |X_t^i-X_t^j|<r} \frac{\diff t}{X_t^i-X_t^j} + \diff B_t^i \cquad i \in \N  \comma
\end{align*}
where $B_t^1, B_t^2, \ldots$ are independent Brownian motions in~$\R$.  The pathwise uniqueness and the existence of the strong solution have been proven in~ \cite[Thm.~1.2]{Tsa16} for $1 \le \beta <\infty$. When $\beta=1,2,4$, the unlabelled solution is associated with the upper Dirichlet form, see \cite[Lem.~8.4]{Tsa16} and \cite[Thm.~24, 25]{Osa12}. Since the upper Dirichlet form coincides with  $(\E^{\U, \QP}, \dom{\E^{\U, \QP}})$ due to Thm.~\ref{t:IUL} with Rem.~\ref{r:TCO}, the unlabelled solution is associated with $(\E^{\U, \QP}, \dom{\E^{\U, \QP}})$.
\end{ese}

\begin{ese}[$\mathsf{Airy}_\beta$] \label{exa: A}
By \cite[Thm.\ 4.7]{OsaTan14}, the $\mathsf{Airy}_\beta$ point process with $\beta=1, 2, 4$ belongs to the class of $(0,\Psi)$-quasi-Gibbs measures with the interaction potential
$$\Psi(x, y):=-\beta\log|x-y|, \quad x, y \in \R.$$
The corresponding SDE is the following:  letting $\hat{\rho}(x)=\frac{\sqrt{-x}}{\pi} \1_{\{x<0\}}$, 
\begin{align*}   
\diff X^i_t = \frac{\beta}{2}\lim_{r \to \infty}\Biggl( \biggl(\sum_{\substack{j: j \neq i\\  |X_t^j|<r}}^\infty  \frac{1}{X_t^i-X_t^j} \biggr) -\int_{-r}^r \frac{\hat{\rho}(x)}{-x} \diff x \Biggr) \diff t + \diff B_t^i \cquad i \in \N  \fstop
\end{align*}
The unlabelled solution to this SDE has been identified with the diffusion process associated with $(\E^{\U, \QP}, \dom{\E^{\U, \QP}})$ in~\cite[Thm.~2.2]{OsaTan14} combined with Thm.~\ref{t:IUL} with Rem.~\ref{r:TCO}. 
%
For the existence and the pathwise uniqueness of the strong solution for $\beta=1,2,4$, see~\cite[Thm.~2.3, 2.4]{OsaTan14}. 
\end{ese}

\subsection{Assumption~\ref{ass:CE} and~\ref{ass:ConditionalClos}}
For quasi-Gibbs measures, \ref{ass:CE} follows immediately by \eqref{eq:localACquasiGibbs}. 
 The condition~\ref{ass:ConditionalClos} holds if for $k \in \N_0$ and $r \in \N$,
\begin{align} \label{a:CCC}
e^{-\msH_r}|_{\U^k(B_r)} \ \text{ is bounded and lower semi-continuous in $\U^k(B_r)$}  \fstop
\end{align}
See \cite[Lem.~3.2]{Osa96}. In particular, Examples~\ref{exa: S} and \ref{exa: A} satisfy this condition. 
\subsection{Verification of  Assumption~\ref{a:AP} for $\sine_2$ and $\Airy_2$} 
In this section, we introduce a set $\mathfrak Y_\QP \subset \U$, which play the role of $\Theta$ in Assumption~\ref{a:AP} for $\QP \in \{\sine_2, \Airy_2\}$. 
We briefly recall $\sine_2$ and $\Airy_2$, which are particular determinantal point processes in $\R$. For $A \subset \R$ and $\gamma \in \U$, we write~$\mssn_{A}(\gamma):=\int_\R \1_A \diff \gamma$, which counts the number of points in $\gamma$ belonging to~$A$.  
\begin{defs} \label{d:DPP} Let $\QP$ be a Borel probability measure on~$(\U, \tau_\mrmv)$.
\begin{itemize}
\item A locally Lebesgue-integrable function $f_\QP^{(k)}: \R^k \to \R_+$ is called {\it $k$-point correlation function} of $\QP$ if for any $m \in \N$, any disjoint bounded Borel subsets $A_1, A_2, \ldots, A_m \subset \R$  and any $k_1, k_2, \ldots, k_m \in \N$ with $k=\sum_{i=1}^m k_i$, we have 
$$\int_{\U}\prod_{i=1}^m\frac{\mssn_{A_i}(\gamma)!}{(\mssn_{A_i}(\gamma)-k_i)!}\diff \QP(\gamma) = \int_{A_1^{\times k_1}\times \cdots \times A_m^{\times k_m} }f_\QP^{(k)}(x_1, \ldots, x_k)\diff x_1 \cdots \diff x_k \fstop$$
\item $\QP$ is (the law of) {\it a determinantal point process with kernel $K: \R^2 \to \mathbb C$} if $\QP$ has the $k$-point correlation function given by
$$f_\QP^{(k)}(x_1, x_2, \ldots, x_k)={\rm det}\bigl[K(x_i, x_j)\bigr]_{i, j=1}^k \cquad k \in \N \fstop$$
\item $\sine_2$ is the determinantal point process whose kernel is 
$$K(x, y)=\frac{\sin(\pi(x-y))}{\pi(x-y)} \fstop$$
\item $\Airy_2$ is the determinantal point process whose kernel is 
$$K(x, y)=\frac{{\rm Ai}(x){\rm Ai}'(y)-{\rm Ai}'(x){\rm Ai}(y)}{x-y} \comma$$
where ${\rm Ai}'(x)=\frac{\diff {\rm Ai}(x)}{\diff x}$ and 
$$\Ai(x)=\frac{1}{\pi}\int_0^\infty\cos\Bigl(\frac{t^3}{3}+xt\Bigr) \diff t \fstop$$
\end{itemize}
\end{defs}
Let $\rho_\QP$  be the intensity measure of $\QP$: 
$$\rho_\QP(A):=\int_{\U} \mssn_A(\gamma) \diff \QP(\gamma) = \int_A f^{(1)}_\QP \diff x \in [0,\infty] \cquad A \in \msB(\R) \fstop$$ 
By computing the $1$-correlation function~$f_\QP^{(1)}$, we have that 
\begin{itemize}
\item $\rho_{\sine_2}$ is the Lebesgue measure in $\R$;
 \item $\rho_{\Airy_2}(A)=\int_A (\Ai'(x)^2-x\Ai(x)^2) \diff x$;
\item the $1$-correlation function $f^{(1)}_{\Airy_2}=\frac{\mssd \rho_{\Airy_2}}{\mssd x}$ has the following asymptotic formula (\cite[Lem.~5.3]{OsaTan14}):
\begin{align}\label{e:IE}
f^{(1)}_{\Airy_2}(x) = 
\begin{cases}
O_{x \to \infty}(e^{-\frac{4}{3}x^{3/2}}) \ &\quad \text{on} \quad [0,\infty) \semicolon
\\
\frac{(-x)^{1/2}}{\pi}\bigl(1+O_{x \to \infty}((-x)^{-3/2})\bigr)& \quad  \text{on} \quad (-\infty, 0] \fstop 
\end{cases}
\end{align}
In particular, $f^{(1)}_{\Airy_2}$ is bounded in $[0, \infty)$. 
\end{itemize}
We set the list of assumptions for $\gamma$:
\begin{itemize}
\item  $\gamma(\R)=\infty$, and there exist $L_0 \in \N$, $\e>0$  such that
\begin{align} \label{e:sum11}
\bigl|\rho_\QP([0,L])-\gamma([0,L])\bigr| \lesssim L^\e \cquad \bigl|\rho_\QP([-L,0])-\gamma([-L,0])\bigr| \lesssim  L^\e 
\end{align}
for every $L \ge L_0$. Here $a(L) \lesssim b(L)$ means $a(L) \le cb(L)$ for some constant $c>0$ independent of $L$.
\item there exist  $\kappa \in (1/2, 1)$ and $m \in \N$ such that
\begin{align}\label{e:sum2}
m(\gamma, \kappa):=\max_{k \in \mathbb Z}\gamma\Bigl([g^\kappa(k), g^\kappa(k+1)]\Bigr) \le m \comma
\end{align}
where $g^\kappa(x):={\rm sgn}(x)|x|^\kappa$ for $x \in \R$. 
\end{itemize}
We define the following sets:
\begin{itemize}
 \item $\mathfrak Y_{m, L_0}^{\kappa, \e}(\sine_2)$ is the set of configurations $\gamma$ satisfying~\eqref{e:sum11} with $\QP=\sine_2$ and~\eqref{e:sum2};
 \item  $\mathfrak Y_{m, L_0}^{\kappa, \e}(\Airy_2)$ is the set of configurations $\gamma$ satisfying~\eqref{e:sum11} with $\QP=\Airy_2$ and~\eqref{e:sum2}.
 \end{itemize}
\begin{defs} \label{d:EMS} Let $\QP \in \{\sine_2, \Airy_2\}$. Define
 $$\mathfrak Y_{\QP}:=\bigcup_{\substack{\kappa \in (0, \kappa_\QP)\\  \e\in(0,1)}}\bigcup_{L_0, m \in \N}\mathfrak Y_{m, L_0}^{\kappa, \e}(\QP) \comma$$
 where $\kappa_{\sine_2}=1$ and $\kappa_{\Airy_2}=\frac23$. 
 \end{defs}

\begin{thm}\label{t:DB}
 Assumption~\ref{a:AP} holds with $\Theta=\mathfrak Y_\QP$ and $K=0$ for $\QP=\sine_2$ and $\QP=\Airy_2$.
\end{thm}
\begin{proof}
Due to~\cite[Line -14 on p.~190]{OsaTan16}, we have $\QP(\mathfrak Y_\QP)=1$ for~$\QP\in \{\sine_2, \Airy_2\}$.
In the following, we verify \ref{AP-11}--\ref{AP-4} in~Assumption~\ref{a:AP}.

\paragraph{Verification of \ref{AP-11}} Let $\QP$ be either $\sine_2$ or $\Airy_2$. Let $\gamma\in \U$ and $\eta \in \mathfrak Y_\QP$ with $\mssd_\U(\gamma, \eta)<\infty$.
We take a labelling $\gamma=\sum_{i \in \Z} \delta_{x_i}$ with $x_i \le x_{i+1}$ for $i \in \Z$. 
Take an optimal matching~$\mssc=\sum_{i \in \Z} \delta_{(x_i, y_i)}$ for $\mssd_\U(\gamma, \eta)$, i.e., $\eta=\sum_{i \in \Z}\delta_{y_i}$ and 
\begin{align}\label{e:IED}
\mssd_\U(\gamma, \eta)^2=\sum_{i \in \Z }|x_i-y_i|^2 <\infty\fstop
\end{align}
We may assume $y_i \le y_{i+1}$: Indeed, if there exists $i \in \Z$ such that $y_{i} > y_{i+1}$, then we can swap the coupling pairs $(x_i, y_i)$ and $(x_{i+1}, y_{i+1})$ to $(x_i, y_{i+1})$ and $(x_{i+1}, y_{i})$ to obtain a smaller sum than (or the same sum as) the original sum. 

\smallskip
We show that $\gamma \in \mathfrak Y_{\QP}$ by checking that $\gamma$ satisfies \eqref{e:sum11} and~\eqref{e:sum2}.
\smallskip
 \\
 \underline{Proof of \eqref{e:sum11}}:  
 Let $x_{i_0^+} \ge 0$ (resp.~$x_{i_0^-} \le 0$) be the non-negative (resp.~non-positive) element in $\gamma$ closest to the origin~$0$ and define $D_+:=|x_{i_0^+}-y_{i_0^+}|$ and $D_-:=|x_{i_0^-}-y_{i_0^-}|$, where $y_{i_0^+}$ (resp.~$y_{i_0^-}$) is the element in $\eta$ matched to $x_{i_0^+}$ (resp.~$x_{i_0^-}$)~in \eqref{e:IED}. If there are multiple such elements $x_{i_0^+} \ge 0$ (resp.~$x_{i_0^-} \le 0$), we choose one attaining the largest~$D_+$ (resp.~$D_-$). Define $D'_+:=x_{i_0^+}+D_+$ (resp.~$D'_-:=x_{i_0^-}-D_-$). The finiteness of the infinite series~\eqref{e:IED} particularly implies that, for every $\delta>0$, there exists $L_1=L_1(\delta) \in \N$ depending on $\delta$ such that, for every $L \ge L_1$, 
 \begin{align} \label{e:CEP}
 &\eta\bigl((D'_+, L-\delta]\bigr) \le \gamma\bigl([0, L]\bigr) \le \eta\big([-D_+, L+\delta]\big) 
 \\
 & \eta\big([-L+\delta, -D_-')\big) \le \gamma\big([-L, 0]\big) \le \eta\big([-L-\delta, D_-]\big)\fstop  \notag
 \end{align}
 Since $\eta$ satisfies \eqref{e:sum11}, we can find some $L_0=L_0(\delta)$ and $\e \in (0,1)$ such that~for every $L \ge L_0 \vee L_1$, we have 
 \begin{align*}
&\rho_{\QP}([0, L])-\gamma([0,L]) 
\\
& \le  \rho_{\QP}([0, L])-\eta((D_+',L-\delta])  
\\
&=  \rho_{\QP}([0, L])-\eta([0,L-\delta]) +\eta([0,D_+']) 
 \\
 &= \rho_{\QP}([0, L])- \rho_{\QP}([0, L-\delta])  + \rho_\QP([0, L-\delta])-\eta([0,L-\delta]) +\eta([0,D_+']) 
   \\
 &=  \rho_{\QP}([L-\delta, L]) + \rho_\QP([0, L-\delta]) - \eta([0,L-\delta]) + \eta([0,D_+'])
  \\
 &\lesssim  \rho_{\QP}([L-\delta, L]) + L^\e \fstop
 \end{align*}
 Similarly, 
 \begin{align*}
&\gamma([0,L])- \rho_{\QP}([0, L]) 
\\
& \le \eta([-D_+,L+\delta])-  \rho_{\QP}([0, L])
\\
& \le \eta([0,L+\delta]) + \eta([-D_+, 0])-  \rho_{\QP}([0, L+\delta])  + \rho_{\QP}([0, L+\delta])  -\rho_{\QP}([0, L])
  \\
 &= \rho_{\QP}([L, L+\delta]) +  \eta([0,L+\delta]) -  \rho_{\QP}([0, L+\delta])+ \eta([-D_+, 0])
  \\
 &\lesssim \rho_{\QP}([L, L+\delta]) + (L+\delta)^\e  
   \\
 &\lesssim \rho_{\QP}([L, L+\delta]) + L^\e  \comma
 \end{align*}
 where the last line follows because there exists a constant $L_2=L_2(\delta)$  depending on~$\delta$ such that for every $L \ge L_2$, we have $(L+\delta)^\e \le CL^\e$ with some constant~$C=C(\delta)>0$.
 Thus, we conclude
 \begin{align*}
 &\Bigl|\gamma([0,L])- \rho_{\QP}([0, L])  \Bigr| 
 \\
 &\lesssim\max\Bigl\{ \rho_{\QP}([L-\delta, L]) + L^\e,  \rho_{\QP}([L, L+\delta]) + L^\e\Bigr\} \fstop
 \end{align*}
Having a similar argument wth $D_-$ and $D_-'$ in \eqref{e:CEP} instead of~$D_+$ and $D_+'$ , we have the following estimate for the negative interval:
  \begin{align*}
 &\Bigl| \gamma([-L, 0])- \rho_{\QP}([-L, 0])    \Bigr| 
 \\
 &\lesssim \max\Bigl\{ \rho_{\QP}([-L, -L+\delta]) + L^\e,  \rho_{\QP}([-L-\delta, -L]) + L^\e \Bigr\} \fstop
 \end{align*}
 Due to the fact that $\rho_{\sine_2}$ is the Lebesgue measure in $\R$,  the sought conclusion~\eqref{e:sum11}  holds for $\QP=\sine_2$.  Using  the asymptotic formula \eqref{e:IE}, we can see that   all the terms $\rho_{\QP}([L-\delta, L])$, $\rho_{\QP}([L, L+\delta])$, $\rho_{\QP}([-L, -L+\delta])$ and $\rho_{\QP}([-L-\delta, -L])$ have the asymptotic rate at most $O_{L \to \infty}\bigl(L^{1/2}\bigr)$. This provides the sought conclusion~\eqref{e:sum11} for $\QP=\Airy_2$. 

\vspace{2mm}
 \noindent\underline{Proof of~ \eqref{e:sum2}}: Using a similar argument to \eqref{e:CEP} and the hypothesis that $\eta$ satisfies~\eqref{e:sum2}, it is easy to see that there exists $m' \in \N$  such that 
\begin{align*}
m(\gamma, \kappa):=\max_{k \in \mathbb Z}\gamma\Bigl([g^\kappa(k), g^\kappa(k+1)]\Bigr) \le m' \comma
\end{align*}
which concludes that  $\gamma$ satisfies \eqref{e:sum2}. 
\vspace{2mm}

\paragraph{Verification of \ref{AP-2}} Thanks to Thm.~\ref{t:IUL} and Rem.~\ref{r:TCO}, our Dirichlet form $\E^{\U, \QP}$ coincides with the upper Dirichlet form $\overline{\E}^{\U, \QP}$ used in \cite[Thm.~2.1]{OsaTan16}. 
Thus, the existence of the kernel $p_t^{\U, \QP}(\gamma, \diff \eta)$ follows by~\cite[Thm.~2.6]{KatTan11} and \cite[Thm.~2.1]{OsaTan16} for $\QP=\sine_2$, and ~\cite[Thm.~2.8]{KatTan11} and \cite[Thm.~2.1]{OsaTan16} for $\QP=\Airy_2$, where the kernel is obtained as the transition probability of the strong Markov process constructed there. 
Furthermore, it is associated with $\overline{\E}^{\U,\QP}$ (thus also with~${\E}^{\U,\QP}$) due to~\cite[Thm.~2.2]{OsaTan16}. Thus, $\tilde{T}_t^{\U, \QP}u$ is a $\QP$-representative of $T_t^{\U, \QP}u$, which completes the verification.
\paragraph{Verification of \ref{AP-4}} 
Let $\gamma, \eta \in \mathfrak Y_\QP$ with $\mssd_\U(\gamma, \eta)<\infty$ and take the same labelling as  in~\eqref{e:IED}. 
Define $\gamma^{(2k)}=\sum_{i=-k}^{k}\delta_{x_i}$ and $\eta^{(2k)}=\sum_{i=-k}^{k}\delta_{y_i}$. Then, we have 
$$\mssd_{\U}(\gamma, \eta)^2 = \lim_{k \to \infty} \sum_{i=-k}^k|x_i-y_i|^2 \ge \limsup_{k \to \infty} \mssd_{\U^{2k}}(\gamma^{(2k)}, \eta^{(2k)}) \fstop$$
Due to the weak convergence result~\cite[Prop.~3.4]{OsaTan16} that  the laws of the infinite particles of interacting Brownian motions  corresponding to $(\E^{\U, \QP}, \dom{\E^{\U, \QP}})$  can be weakly approximated by the laws of the finite-particle counterpart (see \eqref{d:kDBM} and \eqref{d:kABM} below),  we have a probability measure $\QP^k$ on $\U^k$ (see \eqref{d:kDBMIV} and \eqref{d:kABMIV} below) and the heat  kernel $p_t^{\U^k, \QP^k}(\gamma, \diff \zeta)$ such that, as $k \to \infty$
\begin{align*}
&p^{\U^{2k}, \QP^{2k}}_t(\gamma^{(2k)}, \diff \zeta) \xrightarrow{\tau_\mrmw} p^{\U, \QP}_t(\gamma, \diff \zeta) 
\\
&p^{\U^{2k}, \QP^{2k}}_t(\eta^{(2k)}, \diff \zeta) \xrightarrow{\tau_\mrmw} p^{\U, \QP}_t(\eta, \diff \zeta)  \fstop
\end{align*}
Here the convergence of the kernels follows by the general fact that the weak convergence of the laws of stochastic processes implies the weak convergence of the one-dimensional time-marginal laws. 
The kernel $p_t^{\U^k, \QP^k}(\gamma, \diff \eta)$ is  the transition probability kernel of the unlabelled solution to the following $k$-particle Dyson-type SDE with the invariant probability measure $\QP^k$: 
\smallskip
\\
\underline{Case $\QP=\sine_2$} 
 (\cite[(3.9)]{OsaTan16}):
\begin{align}  \label{d:kDBM}
\diff X_t^i=    - \frac{X^i_t}{k}\diff t + \sum_{j: j \neq i}^k \frac{\diff t}{X_t^i-X_t^j} + \diff B^i_t, \quad i \in \{1,2,\ldots, k\} \comma 
\end{align}
where $B^1,B^2,\ldots, B^k$ are independent Brownian motions. By using Ito's formula, we can readily see that the Dirichlet form~$(\E^{\U^{k}, \mu^{k}}, \dom{\E^{\U^{k}, \mu^{k}}})$ corresponding to the unlabelled solution to \eqref{d:kDBM}  is written as 
$$\E^{\U^{k}, \mu^{k}}:=\frac{1}{2}\int_{\U^{k}} |\nabla^{\odot k} u|^2 \mssd \QP^k \cquad u \in \Lip_{c}(\U^{k}, \mssd_{\U^k}) \comma$$
where $\dom{\E^{\U^{k}, \mu^{k}}}$ is the closure of $\Lip_{c}(\U^{k}, \mssd_{\U^k})$ and 
\begin{align} \label{d:kDBMIV}
&\frac{\diff \QP^k}{\diff \mssm^{\odot k}} = \frac{1}{Z}e^{-\Psi^{k}-\Phi^k} \cquad Z= \int_{\U^k}  e^{-\Psi^{k}-\Phi^k} \diff \mssm^{\odot k} \comma
\\
&\Psi^k(\gamma):=-2\sum_{i<j}^k\log|x_i-x_j| \cquad \Phi^k(\gamma)=\frac{1}{2k}\sum_{i=1}^k|x_i|^2 \cquad \gamma =\sum_{i=1}^k \delta_{x_i} \fstop \notag
\end{align}
\smallskip
\\
\underline{Case $\QP=\Airy_2$}
 (\cite[(3.12)]{OsaTan16}): 
\begin{align} \label{d:kABM}
\diff X^i_t = -\frac{X_t^i+2k^{2/3}}{2k^{1/3}} \diff t + \sum_{\substack{j: j \neq i} }^k \frac{\diff t}{X_t^i-X_t^j}   + \diff B_t^i \cquad i \in \{1,\ldots, k\}  \comma
\end{align}
where $B^1,B^2,\ldots, B^k$ are independent Brownian motions. 
In this case, the Dirichlet form $(\E^{\U^{k}, \mu^{k}}, \dom{\E^{\U^{k}, \mu^{k}}})$ corresponding to the unlabelled solution to \eqref{d:kABM} is written as 
$$\E^{\U^{k}, \mu^{k}}:=\frac{1}{2}\int_{\U^{k}} |\nabla^{\odot k} u|^2 \diff \QP^k \cquad u \in \Lip_{c}(\U^{k}, \mssd_{\U^k}) \comma$$
where $\dom{\E^{\U^{k}, \mu^{k}}}$ is the closure of $\Lip_{c}(\U^{k}, \mssd_{\U^k})$ and 
\begin{align} \label{d:kABMIV}
&\frac{\diff \QP^k}{\diff \mssm^{\odot k}} = \frac{1}{Z}e^{-\Psi^{k}-\Phi^k}  \cquad  Z= \int_{\U^k}  e^{-\Psi^{k}-\Phi^k} \diff \mssm^{\odot k} \comma
\\
&\Psi^k(\gamma):=-2\sum_{i<j}^k\log|x_i-x_j| \cquad \Phi^k(\gamma)=\frac{1}{4k^{1/3}}\sum_{i=1}^k\bigl(x_i+2k^{2/3}\bigr)^2 \cquad \gamma =\sum_{i=1}^k \delta_{x_i} \fstop \notag
\end{align}
The $\RCD$ property of the approximating space $(\U^k, \mssd_{\U^k},\QP^k)$ can be verified as follows:
Recalling that $\U^k \cong \R^k/\mathfrak S_k$, we may identify $\U^k $ with the convex set~$W:=\{(x_i)_{i=1}^k \in \R^k: x_1 \ge x_2 \ge \cdots \ge x_k\}$. 
As the Hessian matrix of~$\msH^k=\Psi^k+\Phi^k$ for each $\QP \in \{\sine_2, \Airy_2\}$ is non-negative definite and smooth in the interior $\mathring{W}$ of $W$, the Hamiltonian $\msH^k$ is geodesically convex in $W\cong\U^k$. Hence, the space $(\U^k, \mssd_{\U^k}, \QP^{k})$ is an $\RCD(0,\infty)$ and the Dirichlet form $(\E^{\U^{k}, \mu^{k}}, \dom{\E^{\U^{k}, \mu^{k}}})$ coincides with the Cheeger energy $(\Ch^{\mssd_{\U^{k}}, \mu^{k}}, \dom{\Ch^{\mssd_{\U^{k}}, \mu^{k}}})$, see e.g., the proof of \cite[Prop.~3.3]{Suz22b}, where $\RCD(0,\infty)$ property was proved by approximating the singular Hamiltonian $\msH^k$ by non-singular Hamiltonians away from the diagonal of $W$ and using the stability of the $\RCD$ property under this approximation.  \qedhere
\end{proof}
\section*{Appendix: Measurable selection of geodesics in extended metric measure spaces}
In this Appendix, we  construct a measurable map selecting geodesics in extended metric measure spaces. The arguments in this section are independent of configuration spaces and may be of general interest  for optimal transportation problems in extended metric measure spaces. 
\begin{proof}[Proof of Lemma~\ref{l:MS}] 
Consider the (possibly multi-valued) map $G: D \to \mathsf{Geo}(X, \mssd)$ that associates to each pair $(x, y) \in D$ the set $G(x, y)$ of constant speed $\mssd$-geodesics connecting $x$ and $y$. Take $\mssd_n \nearrow \mssd$ witnessing  Definition~\ref{d:EMM} with $n \in \mathbb N$.  We consider the (possibly multi-valued) map $G_{n, \e}: D \to \mathsf{Geo}(X, \mssd_n, \e)$
 that associates to each pair $(x, y) \in D$ the set $G_{n, \e}(x, y)$ of $\e$-constant speed $\mssd_n$-geodesics connecting $x$ and $y$,
  where 
$$\mathsf{Geo}(X, \mssd_n, \e):=\Bigl\{g \in C\bigl([0,1];(X, \tau)\bigr): \mssd_n(g_t, g_s) \le |t-s|(\mssd_n(g_0, g_1)+\e) , t, s \in [0,1]\Bigr\}$$
and $G_{n, \e}(x, y)$ is defined as 
$$\Bigl\{g \in C\bigl([0,1];(X, \tau)\bigr): \mssd_n(g_t, g_s) \le |t-s|(\mssd_n(x, y)+\e),\ g_0=x, \ g_1=y,\  t, s \in [0,1]\Bigr\} \fstop$$
We prove that the graph $\mathsf{Graph}(G_{n, \e})\subset X^{\times 2} \times  \mathsf{Geo}(X, \mssd_n, \e)$ defined as
$$\mathsf{Graph}(G_{n, \e}):=\bigl\{\bigl((x, y), g\bigr): (x, y) \in X^{\times 2},\ g \in G_{n, \e}(x, y)\bigr\}$$
is  closed with respect to the topology $\tau^{\times 2} \times \tau_\mrmu$. Fix $n \in \N$ and $\e>0$ and take $\bigl((x^{(m)}, y^{(m)}), g^{(m)}\bigr) \in \mathsf{Graph}(G_{n, \e})$  converge to $\bigl((x, y), g\bigr)$ with respect to~$\tau^{\times 2} \times \tau_\mrmu$. By the $\tau^{\times 2}$-continuity of~$\mssd_n$, 
\begin{align*}
\mssd_n(g_t, g_s) &= \lim_{m \to \infty}\mssd_n(g^{(m)}_t, g^{(m)}_s) 
\\
&\le |t-s|\lim_{m \to \infty}(\mssd_n(x^{(m)}, y^{(m)})+\e)
\\
& = |t-s|(\mssd_n(x, y)+\e) \comma
\end{align*}
which concludes $\bigl((x, y), g\bigr) \in \mathsf{Graph}(G_{n, \e})$ and the $\tau^{\times 2} \times \tau_\mrmu$-closedness of~$\mathsf{Graph}(G_{n, \e})$. 
We now prove that the graph 
$$\mathsf{Graph}(G):=\bigl\{\bigl((x, y), g\bigr): (x, y) \in D,\ g \in G(x, y)\bigr\} \subset  D \times  \mathsf{Geo}(X, \mssd)$$ is  $\mathscr B(\tau^{\times 2} \times \tau_\mrmu)$-measurable. It suffices to prove that 
$$\mathsf{Graph}(G)=\bigcap_{k \in \N} \liminf_{n\to \infty}\mathsf{Graph}(G_{n, \e_k}) \bigcap  D \times  \mathsf{Geo}(X, \mssd) \comma$$
where $\e_k=\frac{1}{k}$. 
We prove the inclusion $\subset$. Take $((x, y), g) \in \mathsf{Graph}(G)$. By definition, $((x, y), g) \ \in D \times  \mathsf{Geo}(X, \mssd)$ and 
$$\mssd(g_t, g_s)=|t-s|\mssd(x, y) \fstop$$
By the monotonicity~$\mssd_n \nearrow \mssd$ as $n \nearrow \infty$, we have that, for any $\e_k$, there exists $n_0=n_0(x, y, \e_k)>0$ such that for every $n \ge n_0$, $\mssd(x, y) \le \mssd_n(x, y) + \e_k$. Thus, 
$$\mssd_n(g_t, g_s) \le \mssd(g_t, g_s)=|t-s|\mssd(x, y) \le |t-s|(\mssd_n(x, y)+\e_k) \fstop$$
Thus, for every~$\e_k$ and $x, y \in X$ with $\mssd(x, y)<\infty$, there exists a constant $n_0=n_0(x, y, \e_k) \in \N$ depending on $x, y$ and $\e_k$ such that, for every $n \ge n_0(x, y, \e_k)$,  we have $((x, y), g) \in \mathsf{Graph}(G_{n, \e_k})$, which precisely means 
$$((x, y), g) \in \bigcap_{k \in \N} \bigcup_{n=1}^\infty\bigcap_{m=n}^\infty\mathsf{Graph}(G_{m, \e_k}) = \bigcap_{k \in \N } \liminf_{n\to \infty}\mathsf{Graph}(G_{n, \e_k}) \fstop$$
We now prove the opposite inclusion $\supset$. Take 
$$((x, y), g) \in \bigcap_{k \in \N} \liminf_{n\to \infty}\mathsf{Graph}(G_{n, \e_k}) \bigcap D \times  \mathsf{Geo}(X, \mssd) \fstop$$ 
By definition, for every $\e_k$, there exists a constant $n_1=n_1(x, y, \e_k) \in \N$ such that for every $n \ge n_1$, 
$$\mssd_n(g_t, g_s) \le |t-s|(\mssd_n(x, y)+\e_k) \fstop$$
By the monotone convergence $\mssd_n \nearrow \mssd$ as $n \nearrow \infty$, taking $n \to \infty$ and then $\e_k \to 0$, we have 
$$\mssd(g_t, g_s) \le |t-s|\mssd(x, y) \comma$$
which concludes $((x, y), g) \in \mathsf{Graph}(G)$ by recalling \eqref{e:GEI}. 
Thus, $\mathsf{Graph}(G)$ is $\mathscr B(\tau^{\times 2} \times \tau_\mrmu)$-measurable. Since $\tau$ and  $\tau_\mrmu$ are both second countable, $\mathscr B(\tau^{\times 2} \times \tau_\mrmu)=\mathscr B(\tau^{\times 2})\otimes \msB( \tau_\mrmu)$, thus $\mathsf{Graph}(G)$ is also  $\mathscr B(\tau^{\times 2})\otimes \msB( \tau_\mrmu)$-measurable.  By Aumann/Sainte-Beuve's measurable selection theorem~\cite[Thm.~6.9.13, Vol.~II]{Bog07}, for every~$\nu \in \mathcal P(D)$ there exists a $\mathscr B(\tau^{\times 2})^{\nu}/\mathscr B(\tau_\mrmu)$-measurable single-valued map $\mathsf{GeoSel}:D \to \mathsf{Geo}(X, \mssd)$  such that $\mathsf{GeoSel}(x, y)$ is a constant speed geodesic connecting $x$ and $y$. 
\end{proof}
\begin{rem} \ 
\begin{enumerate}[(a)]
\item In contrast to metric spaces (cf.~\cite[Lem.~2.11]{AmbGig12}), $\mathsf{Graph}(G)$ is not necessarily $\tau^{\times 2} \times \tau_\mrmu$-closed in the case of extended metric  spaces because $\mssd$ does not metrise $\tau$ in general. 

\item To apply the Aumann/Sainte-Beuve measurable selection theorem (e.g., in \cite[Thm.~6.9.13, Vol.~II]{Bog07}), we need the completed $\sigma$-algebra $\mathscr B(\tau^{\times 2})^{\nu}$. We do not know whether we can choose a Borel selection $\mathsf{GeoSel}(x, y)$ in this generality. However, the measurability with respect to the completion is sufficient for  this paper, particularly to discuss push-forward measures by the map~$\mathsf{GeoSel}$.
\end{enumerate}  
\end{rem}
\begin{proof}[Proof of Proposition~\ref{p:EMW}] 
The space~$(\mathcal P(X), \tau_\mrmw, \mssW_{p, \mssd})$ is a complete extended metric topological space due to~\cite[Prop.~5.3, 5.4]{AmbErbSav16} when $p=2$. The proof for general $1\le p<\infty$ is verbatim, so we omit the proof.  We prove the geodesic property. 
Take $\nu, \sigma \in \mcP(X)$ with $\mssW_{p, \mssd}(\nu, \sigma)<\infty$, and take an optimal coupling $\mssc \in {\rm Opt}(\nu , \sigma)$: 
$$\mssW_{p, \mssd}(\nu, \sigma)^p = \int_{X^{\times 2}} \mssd(x, y)^p \diff \mssc(x, y)<\infty \fstop$$
Thus, $\mssc(\mssD)=1$ and  we can regard $\mssc \in \mcP(D)$ by restricting $\mssc$ on~$D$. Due to~Lem.~\ref{l:MS}, we can define the push-forward~$\mathbf c:=\mathsf{GeoSel}_\#\mssc \in \mcP(\mathsf{Geo}(X, \mssd))$. By \cite[Prop.~4.1]{Lis16}, the push-forward $\nu_t:=(e_t)_\# \mssc \in \mcP(X)$ is a constant speed geodesic in $(\mcP(X), \mssW_{p, \mssd})$ connecting $\nu$ and $\sigma$, where $e_t: \mathsf{Geo}(X, \mssd) \to X$ is defined as $e_t(g)=g_t$. 
\end{proof}

\begingroup
\small         
\bibliographystyle{alpha}
\bibliography{/Users/suzukikouhei/MasterBib.bib}

\begin{thebibliography}{AGMR15}

\bibitem[AES16]{AmbErbSav16}
{Ambrosio, L.}, {Erbar, M.}, and {Savar\'e, G.}
\newblock {Optimal transport, Cheeger energies and contractivity of dynamic
  transport distances in extended spaces}.
\newblock {\em {Nonlinear Anal.}}, 137:77--134, 2016.

\bibitem[AG12]{AmbGig12}
{Ambrosio, L.} and {Gigli, N.}
\newblock {\em {A User’s Guide to Optimal Transport}}, volume {2062} of {\em
  {Modelling and Optimisation of Flows on Networks, Lecture Note in
  Mathematics}}.
\newblock {Springer}, {2012}.

\bibitem[AGMR15]{AmbGigMonRaj12}
{Ambrosio, Luigi}, {Gigli, Nicola}, {Mondino, Andrea}, and {Rajala, Tapio}.
\newblock {Riemannian Ricci curvature lower bounds in metric measure spaces
  with $\sigma$-finite measure}.
\newblock {\em {Trans.\ Amer.\ Math.\ Soc.}}, 367:4661--4701, 2015.

\bibitem[AGS08]{AmbGigSav08}
{Ambrosio, L.}, {Gigli, N.}, and {Savar\'e, G.}
\newblock {\em {Gradient Flows in Metric Spaces and in the Space of Probability
  Measures}}.
\newblock {Lectures in Mathematics - ETH Z\"urich}. {Birkh\"{a}user},
  {$2^{\textrm{nd}}$} edition, 2008.

\bibitem[AGS14a]{AmbGigSav14}
{Ambrosio, L.}, {Gigli, N.}, and {Savar\'e, G.}
\newblock {Calculus and heat flow in metric measure spaces and applications to
  spaces with Ricci bounds from below}.
\newblock {\em {Invent.\ Math.}}, 395:289--391, 2014.

\bibitem[AGS14b]{AmbGigSav14b}
{Ambrosio, L.}, {Gigli, N.}, and {Savar{\'{e}}, G.}
\newblock {Metric measure spaces with Riemannian Ricci curvature bounded from
  below}.
\newblock {\em {Duke Math.~J.}}, 163(7):1405--1490, 2014.

\bibitem[AGS15]{AmbGigSav15}
{Ambrosio, L.}, {Gigli, N.}, and {Savar{\'{e}}, G.}
\newblock {Bakry--{\'{E}}mery Curvature-Dimension Condition and Riemannian
  Ricci Curvature Bounds}.
\newblock {\em {Ann.~Probab.}}, 43(1):339--404, 2015.

\bibitem[AH23]{AggHua23}
{Aggarwal. A} and {Huang, J}.
\newblock {Strong Characterization for the Airy Line Ensemble}.
\newblock {\em arXiv:2308.11908}, 2023.

\bibitem[AKR98a]{AlbKonRoe98}
{Albeverio, S.}, {Kondratiev, Yu. G.}, and {R{\"{o}}ckner, M.}
\newblock {Analysis and Geometry on Configuration Spaces}.
\newblock {\em {J.\ Funct.\ Anal.}}, 154(2):444--500, 1998.

\bibitem[AKR98b]{AlbKonRoe98b}
{Albeverio, S.}, {Kondratiev, Yu.~G.}, and {R{\"{o}}ckner, M.}
\newblock {Analysis and Geometry on Configuration Spaces: The Gibbsian Case}.
\newblock {\em {J.\ Funct.\ Anal.}}, 157:242--291, 1998.

\bibitem[BGL14]{BakGenLed14}
{Bakry, D.}, {Gentil, I.}, and {Ledoux, M.}
\newblock {\em {Analysis and Geometry of Markov Diffusion Operators}}, volume
  348 of {\em {Grundlehren der mathematischen Wissenschaften}}.
\newblock {Springer}, 2014.

\bibitem[BH91]{BouHir91}
{Bouleau, N.} and {Hirsch, F.}
\newblock {\em {Dirichlet forms and analysis on Wiener space}}.
\newblock {De Gruyter}, 1991.

\bibitem[BNQ19]{BufNikQiu19}
{Bufetov, A. I.}, {Nikitin, P. P.}, and {Qui, Y.}
\newblock {On number rigidity for Pfaffian point processes}.
\newblock {\em {Mosc. Math. J. }}, 19(2):217--274, 2019.

\bibitem[{Bog}07]{Bog07}
{Bogachev, V.~I.}
\newblock {\em {Measure Theory I, II}}.
\newblock {Springer-Verlag}, {Berlin}, {2007}.

\bibitem[BQS21]{BufQiuSha21}
{Bufetov, A.I.}, {Qui, Y.}, and {Shamov, A.}
\newblock {Kernels of conditional determinantal measures and the Lyons--Peres
  completeness conjecture}.
\newblock {\em {J. Eur. Math. Soc. (JEMS) }}, 23(5):1477--1519, 2021.

\bibitem[{Buf}16]{Buf16}
{Bufetov, A.I.}
\newblock {Rigidity of determinantal point processes with the Airy, the Bessel
  and the Gamma kernel}.
\newblock {\em {Bulletin of Mathematical Sciences}}, 6:1631--72, 2016.

\bibitem[CF11]{CheFuk11}
{Chen, Z-Q.} and {Fukushima, M.}
\newblock {\em {Symmetric Markov Processes, Time Change, And Boundary Theory}},
  volume Lecture Notes in Math 1123 of {\em London Mathematical Society
  Monographs}.
\newblock Princeton University Press, 2011.

\bibitem[CH11]{CorHam11}
{Corwin, I} and {Hammond, A.}
\newblock {Brownian Gibbs property for Airy line ensembles}.
\newblock {\em Inventiones mathematicae}, 195:441--508, 2011.

\bibitem[CN18]{ChhNaj18}
{Chhaibi, R.} and {Najnundel, J.}
\newblock {Rigidity of the Sine$_\beta$ process}.
\newblock {\em {Electron.\ Commun.\ Probab.}}, 23:1--8, 2018.

\bibitem[CS14]{CorSun14}
{Corwin, I.} and {Sun, X.}
\newblock {Ergodicity of the Airy line ensemble}.
\newblock {\em Electronic Communications in Probability}, 19:1--11, 2014.

\bibitem[{Dav}89]{Dav89}
{Davies, E.B.}
\newblock {\em {Heat Kernels and Spectral Theory}}.
\newblock {Cambridge University Press}, 1989.

\bibitem[DHLM20]{DerHarLebMai20}
{Dereudre, D.}, {Hardy, A.}, {Leblé, T.}, and {Maïda, M}.
\newblock {DLR Equations and Rigidity for the Sine-Beta Process}.
\newblock {\em {Commun. Pure Appl. Math.}}, pages 172--222, 2020.

\bibitem[DHS24]{LzDSHerSuz24}
{Dello Schiavo, L.}, {Herry R.}, and {Suzuki K.}
\newblock {Wasserstein Geometry and Ricci Curvature Bounds for Poisson Spaces}.
\newblock {\em Journal de l’École polytechnique -- Mathématiques},
  11:957--1010, 2024.

\bibitem[DS08]{DanSav08}
{Daneri, S.} and {Savar\'{e}, G.}
\newblock {Eulerian Calculus for the Displacement Convexity in the Wasserstein
  Distance}.
\newblock {\em SIAM Journal on Mathematical Analysis}, 40(3):1104--1122, 2008.

\bibitem[DS21a]{LzDSSuz21}
{Dello Schiavo, L.} and {Suzuki, K.}
\newblock {Configuration spaces over singular spaces --I. Dirichlet-Form and
  Metric Measure Geometry --}.
\newblock {\em {arXiv:2109.03192v2} (version 2)}, 2021.

\bibitem[DS21b]{LzDSSuz20}
{Dello Schiavo, L.} and {Suzuki, K.}
\newblock {On the Rademacher and Sobolev-to-Lipschitz Properties for Strongly
  Local Dirichlet Spaces}.
\newblock {\em {J. Func. Anal.}}, 281(11):Online first, 2021.

\bibitem[DS22]{LzDSSuz22}
{Dello Schiavo, L.} and {Suzuki, K.}
\newblock {Configuration Spaces over Singular Spaces II -- Curvature}.
\newblock {\em {arXiv:2205.01379}}, 2022.

\bibitem[EH15]{ErbHue15}
{Erbar, M.} and {Huesmann, M.}
\newblock {Curvature bounds for configuration spaces}.
\newblock {\em {Calc.\ Var.}}, 54:307--430, 2015.

\bibitem[EHJM25]{ErbHueJalMul25}
{Erbar, M}, {Huesmann, M.}, {Jalowy, J.}, and {M\"uller, B.}
\newblock {Optimal Transport Of Stationary Point Processes: Metric Structure,
  Gradient Flow And Convexity Of The Specific Entropy}.
\newblock {\em {J.~Funct.~Anal.~}}, {289, Issue 4}, 2025.

\bibitem[FOT11]{FukOshTak11}
{Fukushima, M.}, {Oshima, Y.}, and {Takeda, M.}
\newblock {\em {Dirichlet forms and symmetric Markov processes}}, volume~19 of
  {\em {De Gruyter Studies in Mathematics}}.
\newblock {de Gruyter}, extended edition, 2011.

\bibitem[{Fre}08]{Fre00}
{Fremlin, D.~H.}
\newblock {\em {Measure Theory -- Volume I - IV, V Part I \& II}}.
\newblock {Torres Fremlin (ed.)}, 2000-2008.

\bibitem[FS13]{FunShi13}
{Funano, K.} and {Shioya, T.}
\newblock {Concentration, Ricci curvature, and eigenvalues of Laplacian}.
\newblock {\em Geometric and Functional Analysis}, 23(3):888--936, 2013.

\bibitem[{Geo}11]{Geo11}
{Georgii, H-O.}
\newblock {\em {Gibbs measures and phase transitions}}, volume~9 of {\em
  Studies in Mathematics}.
\newblock {de Gruyter}, 2011.

\bibitem[Gho15]{Gho15}
S.~Ghosh.
\newblock {Determinantal processes and completeness of random exponentials: the
  critical case}.
\newblock {\em Probability Theory and Related Fields}, 163(3):643--665, 2015.

\bibitem[{Gig}15]{Gig12a}
{Gigli, N.}
\newblock {On the differential structure of metric measure spaces and
  applications}.
\newblock {\em {Memoirs AMS}}, {236}({1113}), 2015.

\bibitem[{Gig}18]{Gig18}
{Gigli, N.}
\newblock {Nonsmooth Differential Geometry -- An Approach Tailored for Spaces
  with Ricci Curvature Bounded from Below}.
\newblock {\em {Mem. Am. Math. Soc.}}, 251(1196), 2018.

\bibitem[GL17]{GhoLeb17}
{Ghosh, S.} and {Lebowitz, J.L.}
\newblock {Fluctuations, large deviations and rigidity in hyperuniform systems:
  a brief survey}.
\newblock {\em {Indian J. Pure Appl.~ Math.~}}, 48:609--631, 2017.

\bibitem[GMS15]{GigMonSav15}
{Gigli, N.}, {Mondino, A.}, and {Savar\'e, G.}
\newblock {Convergence of pointed non-compact metric measure spaces and
  stability of Ricci curvature bounds and heat flows }.
\newblock {\em {Proc. London Math. Soc.}}, 111(3):1071--1129, 2015.

\bibitem[GP17]{GhoPer17}
{Ghosh, S.} and {Peres, Y.}
\newblock {Rigidity and tolerance in point processes: Gaussian zeros and
  Ginibre eigenvalues}.
\newblock {\em {Duke Math. J.}}, 166 (10):1789--1858, 2017.

\bibitem[GV24]{GigVin24}
{Gigli, N.} and {Vincini, S.}
\newblock {Stability of the heat flow under convergence in concentration and
  consequences}.
\newblock {\em arXiv preprint arXiv:2410.05011}, 2024.

\bibitem[HL25]{HerLeb25}
{Herry, R.} and {Leblé, T.}
\newblock {Gradient flow of the infinite-volume free energy for lattice systems
  of continuous spins}.
\newblock {\em arXiv:2502.06500}, 2025.

\bibitem[HM23]{HueMue23}
{Huesmann, M.} and {M\"uller, B.}
\newblock {A Benamou-Brenier formula for transport distances between stationary
  random measures}.
\newblock {\em arXiv:2402.04842}, 2023.

\bibitem[HO15]{HonOsa15}
{Honda R.} and {Osada, S.}
\newblock {Infinite-dimensional stochastic differential equations related to
  Bessel random point fields}.
\newblock {\em {Stoc.\ Proc.\ Appl.}}, 125:3801--3822, 2015.

\bibitem[HS25]{HueSta25}
{Huesmann, M.} and {Stange, H.}
\newblock {Non-local Wasserstein Geometry, Gradient Flows, and Functional
  Inequalities for Stationary Point Processes}.
\newblock {\em arXiv:2402.04842}, 2025.

\bibitem[JKO98]{JorKinOtt98}
{Jordan, R}, {Kinderlehrer, D.}, and {Otto F.}
\newblock {The Variational Formulation of the Fokker--Planck Equation}.
\newblock {\em {SIAM Journal on Mathematical Analysis}}, 29:1--17, 1998.

\bibitem[{Joh}02]{Joh02}
{Johansson, K}.
\newblock {Discrete Polynuclear Growth and Determinantal Processes}.
\newblock {\em Communications in Mathematical Physics}, 242:277--329, 2002.

\bibitem[KOT21]{KawOsaTan21}
{Kawamoto Y.}, {Osada H.}, and {Tanemura H.}
\newblock {Uniqueness Of Dirichlet Forms Related To Infinite Systems Of
  Interacting Brownian Motions}.
\newblock {\em {Potential Anal.}}, 55:639--676, 2021.

\bibitem[KT09]{KatTan09}
{Katori, M.} and {Tanemura, H.}
\newblock {Zeros of Airy function and relaxation process}.
\newblock {\em {J. Stat.~ Phys.~}}, 136:1177--1204, 2009.

\bibitem[KT10]{KatTan10}
{Katori, M.} and {Tanemura, H.}
\newblock {Non-equilibrium dynamics of Dyson’s model with an infinite number
  of particles}.
\newblock {\em {Comm.~Math.~Phys.}}, 293(2):469--497, 2010.

\bibitem[KT11]{KatTan11}
{Katori, M.} and {Tanemura, H.}
\newblock {Markov Property of Determinantal Processes with Extended Sine, Airy,
  and Bessel Kernels}.
\newblock {\em {Markov Proc. Relat. Fields}}, 17(4):541--580, 2011.

\bibitem[Leb24]{Leb24}
Thomas Leblé.
\newblock {DLR equations, number-rigidity and translation-invariance for
  infinite-volume limit points of the 2DOCP}.
\newblock {\em arXiv:2410.04958}, 2024.

\bibitem[Li15]{Li15}
H.~Li.
\newblock {Dimension-Free Harnack Inequalities on $\RCD(K, \infty)$ Spaces}.
\newblock {\em {J. Theoret. Probab.}}, 29:1280--1297, 2015.

\bibitem[{Lis}16]{Lis16}
{Lisini, S.}
\newblock {Absolutely continuous curves in extended Wasserstein--Orlicz
  spaces}.
\newblock {\em {ESAIM: COCV}}, 22:670--687, 2016.

\bibitem[LV09]{LotVil09}
{Lott, J.} and {Villani, C.}
\newblock {Ricci curvature for metric-measure spaces via optimal transport}.
\newblock {\em {Annals of Mathematics}}, (169):903--991, 2009.

\bibitem[{Lyo}18]{Lyo18}
{Lyons, R.}
\newblock {A note on tail triviality for determinantal point processes}.
\newblock {\em {Electron.\ Commun.\ Probab.}}, 23:1--3, 2018.

\bibitem[MR90]{MaRoe90}
{Ma, Z.-M.} and {R\"ockner, M.}
\newblock {\em Introduction to the Theory of (Non-Symmetric) Dirichlet Forms}.
\newblock Springer, 1990.

\bibitem[MR00]{MaRoe00}
{Ma, Z.-M.} and {R{\"{o}}ckner, M.}
\newblock {Construction of Diffusions on Configuration Spaces}.
\newblock {\em {Osaka J.~Math.}}, 37:273--314, 2000.

\bibitem[MS20]{MurSav20}
{Muratori, M.} and {Savar\'e, G.}
\newblock {Gradient flows and Evolution Variational Inequalities in metric
  spaces. I: Structural properties}.
\newblock {\em {J. Funct. Anal.}}, 278 (4), 2020.

\bibitem[NF98]{NagFor98}
{Nagao, T.} and {Forrester, P. J.}
\newblock {Multilevel dynamical correlation functions for Dyson's Brownian
  motion model of random matrices}.
\newblock {\em {Physics Letters A}}, 247:801--850, 1998.

\bibitem[OO18]{OsaOsa18}
{Osada, H.} and {Osada, S.}
\newblock {Discrete approximations of determinantal point processes on
  continuous spaces: Tree representations and tail triviality}.
\newblock {\em {J.\ Stat.\ Phys.\ }}, 170:421--435, 2018.

\bibitem[OO23]{OsaOsa23}
{Osada, H.} and {Osada, S.}
\newblock {Ergodicity of unlabeled dynamics of Dyson’s model in infinite
  dimensions}.
\newblock {\em {J. Math. Phys.}}, 64(4), 2023.

\bibitem[OO25]{OsaOsa25}
{Osada, H.} and {Osada, S.}
\newblock {Infinite-Dimensional Stochastic Differential Equations and Diffusion
  Dynamics of Coulomb Random Point Fields}.
\newblock {\em arXiv preprint arXiv:2508.21658}, 2025.

\bibitem[{Osa}96]{Osa96}
{Osada, H.}
\newblock {Dirichlet Form Approach to Infinite-Dimensional Wiener Processes
  with Singular Interactions}.
\newblock {\em {Comm.\ Math.\ Phys.}}, 176:117--131, 1996.

\bibitem[{Osa}12]{Osa12}
{Osada, H.}
\newblock {Infinite-dimensional stochastic differential equations related to
  random matrices}.
\newblock {\em {Prob.\ Theory Relat. Fields}}, 153(1):471--509, 2012.

\bibitem[{Osa}13]{Osa13}
{Osada, H.}
\newblock {Interacting Brownian Motions in Infinite Dimensions with Logarithmic
  Interaction Potentials}.
\newblock {\em {Ann.\ Probab.}}, 41(1):1--49, 2013.

\bibitem[OT16]{OsaTan16}
{Osada, H.} and {Tanemura, H.}
\newblock {Strong Markov property of determinantal processes with extended
  kernels}.
\newblock {\em {Stochastic Process.~Appl.}}, 126:186--208, 2016.

\bibitem[OT20]{OsaTan20}
{Osada, H.} and {Tanemura, H.}
\newblock {Infinite-dimensional stochastic differential equations and tail
  $\sigma$-fields}.
\newblock {\em {Probab.\ Theory Relat.\ Fields}}, 177:1137--1242, 2020.

\bibitem[OT24]{OsaTan14}
{Osada, H.} and {Tanemura, H.}
\newblock {Infinite-dimensional stochastic differential equations arising from
  Airy random point fields}.
\newblock {\em {Stochastics and Partial Differential Equations: Analysis and
  Computations}}, page online, 2024.

\bibitem[OY19]{OzaYok19}
{Ozawa, R,} and {Yokota, T.}
\newblock {Stability of RCD condition under concentration topology}.
\newblock {\em Calculus of Variations and Partial Differential Equations},
  58(4):151, 2019.

\bibitem[PS01]{PraSpo01}
{Pr\"{a}hofer, M} and {Spohn, H.}
\newblock {Scale Invariance of the PNG Droplet and the Airy Process}.
\newblock {\em Journal of Statistical Physics}, 108:1071--1106, 2001.

\bibitem[RS99]{RoeSch99}
{R{\"o}ckner, M.} and {Schied, A.}
\newblock {Rademacher's Theorem on Configuration Spaces and Applications}.
\newblock {\em {J.\ Funct.\ Anal.}}, 169(2):325--356, 1999.

\bibitem[{Rue}70]{Rue70}
{Ruelle, D.}
\newblock {Superstable Interactions in Classical Statistical Mechanics}.
\newblock {\em {Comm.\ Math.\ Phys.}}, 18:127--159, 1970.

\bibitem[{Sav}14]{Sav14}
{Savar{\'{e}}, G.}
\newblock {Self-Improvement of the Bakry--{\'{E}}mery Condition and Wasserstein
  Contraction of the Heat Flow in $\mathrm{RCD}(K,\infty)$ Metric Measure
  Spaces}.
\newblock {\em {Discr.\ Cont.\ Dyn.\ Syst.}}, 34(4):1641--1661, 2014.

\bibitem[{Sav}21]{Sav19}
{Savar{\'{e}}, G.}
\newblock {Sobolev spaces in extended metric-measure spaces.}
\newblock {\em {New Trends on Analysis and Geometry in Metric Spaces}}, pages
  117--276, 2021.

\bibitem[{Spo}87]{Spo87}
{Spohn, H.}
\newblock {Interacting Brownian Particles: A Study of Dyson’s Model}.
\newblock {\em {Hydrodynamic Behavior and Interacting Particle Systems}}, pages
  151--179, 1987.

\bibitem[ST03]{ShiTak03b}
{Shirai, T.} and {Takahashi, Y.}
\newblock {Random point fields associated with certain Fredholm determinants
  II: Fermion shifts and their ergodic and Gibbs properties}.
\newblock {\em {Ann. Probab.}}, 31(3):1533--1564, 2003.

\bibitem[{Stu}06a]{Stu06a}
{Sturm, K.-T.}
\newblock {On the geometry of metric measure spaces. I}.
\newblock {\em {Acta Math.}}, 196:65--131, 2006.

\bibitem[{Stu}06b]{Stu06b}
{Sturm, K.-T.}
\newblock {On the geometry of metric measure spaces. II}.
\newblock {\em {Acta Math.}}, 196:133--177, 2006.

\bibitem[{Suz}23]{Suz23b}
{Suzuki, K.}
\newblock {On The Ergodicity Of Interacting Particle Systems Under Number
  Rigidity}.
\newblock {\em {Probab. Theory and Relat. Fields}}, 2023.

\bibitem[{Suz}24]{Suz24}
{Suzuki, K.}
\newblock {The Infinite Dyson Brownian Motion with beta=2 Does Not Have a
  Spectral Gap}.
\newblock {\em {Bulletin of the London Mathematical Society}}, 57(2), 2024.

\bibitem[{Suz}25]{Suz22b}
{Suzuki, K.}
\newblock {Curvature Bound of Dyson Brownian Motion}.
\newblock {\em {Communication in Mathematical Physics}}, 406 (154), 2025.

\bibitem[TS25]{TasSuz25}
{Tashiro, K.} and {Suzuki, K.}
\newblock {A phase transition in the Bakry-Émery gradient estimate for Dyson
  Brownian motion}.
\newblock {\em arXiv preprint arXiv:2506.04424}, 2025.

\bibitem[{Tsa}16]{Tsa16}
{Tsai, L.-C.}
\newblock {Infinite dimensional stochastic differential equations for Dyson’s
  model}.
\newblock {\em {Probability Theory and Related Fields}}, 166:801--850, 2016.

\bibitem[{Vil}09]{Vil09}
{Villani, C.}
\newblock {\em {Optimal transport, old and new}}, volume 338 of {\em
  {Grundlehren der mathematischen Wissenschaften}}.
\newblock {Springer-Verlag}, 2009.

\bibitem[WY11]{WanYua11}
{Wang, F-Y.} and {Yuan, C}.
\newblock {Harnack inequalities for functional SDEs with multiplicative noise
  and applications}.
\newblock {\em {Stochastic Process.~Appl.,}}, 121:2692--2710, 2011.

\bibitem[{Yoo}05]{Yoo05}
{Yoo, H. J.}
\newblock {Dirichlet forms and diffusion processes for fermion random point
  fields}.
\newblock {\em {J.\ Funct.\ Anal.}}, 219:143--160, 2005.

\bibitem[{Yos}96]{Yos96}
{Yoshida, M.~W.}
\newblock {Construction of infinite dimensional interacting diffusion processes
  through Dirichlet forms}.
\newblock {\em {Probab.\ Theory Relat.\ Fields}}, 106:265--297, 1996.

\end{thebibliography}
\endgroup

\end{document}